\documentclass[a4,12pt]{amsart}
\usepackage{amsfonts,textcomp,amssymb,amsmath,amsthm}
\usepackage{color,ulem}
\usepackage{fullpage}
\usepackage{graphicx}
\usepackage[active]{srcltx}
\usepackage{mathabx}
\usepackage{dsfont,mathrsfs,stmaryrd,wasysym,amsbsy}
\usepackage{enumerate}

\renewcommand{\labelitemi}{\textbullet}

\newcommand{\esssuptext}[1]{\textrm{\rm ess sup}_{#1}}
\newcommand{\esssupmath}[1]{{\underset{#1}{\textrm{\rm ess sup}} }}

\newcommand{\EE}{\mathbb{E}}
\newcommand{\E}{\mathbb{E}}

\newcommand{\NN}{\mathbb{N}}
\newcommand{\nn}{\mathbb{N}}

\newcommand{\PP}{\mathbb{P}}

\newcommand{\RR}{\mathbb{R}}
\newcommand{\rr}{\mathbb{R}}
\newcommand{\cc}{\mathbb{C}}

\newcommand{\bone}{\mathbf{1}}

\theoremstyle{plain}
\newtheorem{theorem}{Theorem}[section]
\newtheorem{corollary}[theorem]{Corollary}
\newtheorem{lemma}[theorem]{Lemma}
\newtheorem{proposition}[theorem]{Proposition}

\theoremstyle{definition}
\newtheorem{remark}[theorem]{Remark}

\numberwithin{equation}{section}

\begin{document}
\title{Global solutions to the supercooled Stefan problem with blow-ups: regularity and uniqueness}
\author{Fran\c{c}ois Delarue, Sergey Nadtochiy, and Mykhaylo Shkolnikov}
\address{Universit\'{e} C\^ote d'Azur, CNRS, Laboratoire J.A. Dieudonn\'{e}, Nice, France.}
\email{delarue@unice.fr}
\address{Department of Applied Mathematics, Illinois Institute of Technology, Chicago, IL 60616.}
\email{snadtochiy@iit.edu}
\address{ORFE Department, Bendheim Center for Finance, and Program in Applied \& Computational Mathematics, Princeton University, Princeton, NJ 08544.}
\email{mshkolni@gmail.com}
\footnotetext[1]{S.~Nadtochiy is partially supported by the NSF CAREER grant DMS-1651294.}
\footnotetext[2]{M.~Shkolnikov is partially supported by the NSF grant DMS-1811723 and a Princeton SEAS innovation research grant.}

\begin{abstract}
We consider the supercooled Stefan problem, which captures the freezing of a supercooled liquid, in one space dimension. A probabilistic reformulation of the problem allows to define global solutions, even in the presence of blow-ups of the freezing rate. We provide a complete description of such solutions, by relating the temperature distribution in the liquid to the regularity of the ice growth process. The latter is shown to transition between (i) continuous differentiability, (ii) H\"older continuity, and (iii) discontinuity. In particular, in the second regime we rediscover the square root behavior of the growth process pointed out by Stefan in his seminal paper \cite{Stefan1} from 1889 for the ordinary Stefan problem. In our second main theorem, we establish the uniqueness of the global solutions, a first result of this kind in the context of growth processes with singular self-excitation when blow-ups are present.  
\end{abstract}

\maketitle


\section{Introduction}

The systematic study of free boundary problems for the heat equation, now referred to as Stefan problems, was initiated by \textsc{Stefan} in 1889, see his series of papers \cite{Stefan1}, \cite{Stefan2}, \cite{Stefan3}, \cite{Stefan4}, as well as the precursor \cite{LC} by \textsc{Lam\'{e}} and \textsc{Clapeyron}. Motivated by the process of ice formation in the polar sea, \textsc{Stefan} formulated and solved the free boundary problem describing the freezing of a liquid in the half-space $\{x_1>0\}$ when a constant temperature below its freezing point is maintained at the surface $\{x_1=0\}$, assuming immediate freezing of the liquid at its freezing point. Subsequently, he also formulated and solved similar problems associated with evaporation and condensation. After a period of dormancy, Stefan problems for the heat equation attracted renewed interest as a result of a lecture by \textsc{Brillouin} at the Institut Henri Poincar\'{e} in 1929 and its publication \cite{Bri}. Investigations of existence, uniqueness and numerical approximation of solutions followed (see \cite[introduction, section 1]{Rub} for a detailed historical review), culminating in the article \cite{Kam} by \textsc{Kamenomostskaja}, who proved the existence and uniqueness of bounded measurable generalized solutions and provided an explicit difference scheme for their numerical approximation, in any dimension and in the presence of an arbitrary number of phases. 

\medskip

Much less is known about the \textit{supercooled} Stefan problem for the heat equation, which captures the freezing of a supercooled liquid. In this problem, the initial temperature of the liquid is taken to be lower than the temperature maintained at the surface $\{x_1=0\}$ that, in turn, lies below the freezing point of the liquid. As first noted in \cite{Sher}, already the one-phase problem in dimension one may exhibit a finite time \textit{blow-up} of the liquid freezing rate, leading to a concurrent instantaneous temperature spike along the surface $\{x_1=0\}$, a physically observed phenomenon. Later works were focused on the distinction between (see \cite{FP1}, \cite{FP2}, \cite{LO}, \cite{FPHO}) and the analysis of the two possible cases: (i) the existence of a unique solution without blow-ups for all time or until the time the entire liquid freezes (see \cite{FP1}, \cite{FP2}, \cite{DF}, \cite{ChSw}, \cite{ChKi}); and (ii) the existence of a unique solution until the blow-up time, at which both (the liquid and the solid) phases are present. Naturally, much subsequent attention has been devoted to the analysis of the arguably more intriguing case (ii), specifically to the behavior just before and at the blow-up time (see \cite{HV}, \cite{KE}, \cite{CK}) and to the regularization of the problem through modifications of the boundary condition (see \cite{Vis}, \cite{DHOX},  \cite{HX}, \cite{FPHO2}, \cite{Wei}). However, the methods available in the literature do not allow a global analysis of the actual supercooled Stefan problem in the presence of blow-ups, the objective of our paper. This point is underpinned by the results in \cite{DF}, \cite{Lu}, \cite{GZ}, \cite[theorem 3.2]{CDGP} (see \cite[displays (3.3.2), (3.3.3)]{CDGP} for the connection with the supercooled Stefan problem) where the notion of global solution is too weak to yield uniqueness, cf.~\cite{DF}, \cite{Lu}, \cite{GZ}, or the well-posedness is only established for a functional of the supercooled Stefan problem solution that does not determine the solution uniquely, cf.~\cite[remarks after theorem 3.2]{CDGP}.

\medskip

In contrast to the previous literature on the subject, rather than to regularize the supercooled Stefan problem
\begin{equation}\label{eq:Stefan}
\begin{split}
& \partial_t u = \frac{1}{2}\partial_{xx}u\quad\text{on}\quad D:=\{(t,x)\in[0,\infty)^2:\,x\ge\Lambda_t\}, \\
& \dot{\Lambda}_t = \frac{\alpha}{2}\partial_xu(t,\Lambda_t),\quad t\ge0, \\
& u(0,x)=f(x),\quad x\ge0 \quad\text{and}\quad u(t,\Lambda_t)=0,\quad t\ge0,
\end{split}
\end{equation}
where $f\ge0$ and $\alpha>0$, we consider the \textit{global solutions} of \eqref{eq:Stefan} \textit{in the presence of blow-ups}. Here, $u(t,\cdot)$ and $\Lambda_t$ represent the negative of the temperature profile and the location of the solid-liquid frontier at time $t$, respectively. The global solutions of the supercooled Stefan and other closely related problems arise not only from the physics of supercooled liquids, but have been recently discovered to play an important role in the contexts of integrate-and-fire models in neuroscience (see 
\cite{lewis_dynamics_2003}
and
\cite{ostojic_synchronization_2009}
for neuroscience papers, 
\cite{Caceres2011,Car2013,Capesas2015}
for a PDE approach to those models, and 
\cite{DIRT1}, as well as \cite{DIRT2}, \cite{DIRT3}, for a probabilistic approach), interbank lending network models in finance (see \cite{NadShk1}, \cite{HLS}, \cite{NadShk2}, \cite{LS}, \cite{KR}, \cite{LKR}), and growth processes in probability theory (see \cite{DT}). In particular, \cite[theorem 4.4 and remark 4.5]{DIRT1} guarantee, for a variant of \eqref{eq:Stefan}, the existence of global solutions in which the intervals of instantaneous freezing are chosen to be minimal (more details can be found below in this introduction), referred to as \textit{physical} solutions. Our aim herein is two-fold: (i) to supply a comprehensive description of the physical solutions, including \textit{regularity} estimates for the free boundary $\Lambda$ in the vicinity of blow-ups (that is, near $t>0$ with $\dot{\Lambda}_t=\infty$); (ii) to establish the \textit{uniqueness} of the physical solution for given $f$ and $\alpha$.

\medskip

It is important to stress that the global well-posedness of the supercooled Stefan problem is shown herein \textit{without} the assumption that the initial density $f$ is bounded above by $1/\alpha$. The latter assumption, in particular, excludes discontinuities in $\Lambda$, and it is crucial for the well-posedness proofs in \cite{FP2}, \cite{ChSw}, \cite{LS2}. The results of this paper cover the general case and, hence, require the use of novel arguments. Needless to say that, although it is out of the scope of this work, the adaptation of our approach to the higher-dimensional setting is an exciting prospect.

\medskip

Our key tool in the study of the supercooled Stefan problem \eqref{eq:Stefan} is the following probabilistic reformulation. For a random variable $X_{0-}\ge 0$, an independent standard Brownian motion $B$, and a constant $\alpha>0$, consider the problem of finding a non-decreasing right-continuous function $\Lambda:\,[0,\infty)\to\mathbb{R}$ with left limits such that
\begin{equation}
\label{Stefan_prob}
\begin{split}
& X_t=X_{0-}+B_t-\Lambda_t,\quad t\ge0, \\
& \Lambda_t=\alpha\mathbb{P}(\tau\le t),\quad t\ge0,\quad\text{where}\quad \tau=\inf\{t\ge 0:\,X_t\le 0\}. 
\end{split}
\end{equation}
Assume that $X_{0-}$ possesses a density $f$ in the Sobolev space $W^1_2([0,\infty))$ with $f(0)=0$, and that the derivative $\dot{\Lambda}$ exists as a function in $L^2([0,T])$, for some $T\in(0,\infty)$. Then, for every $t\in[0,T]$, the law of the random variable $X_t\,\bone_{\{\tau\ge t\}}$ admits a density $p(t,\cdot)$ on $(0,\infty)$, and these combine to give the unique solution in the Sobolev space $W^{1,2}_2([0,T]\times[0,\infty))$ of the Cauchy-Dirichlet problem  
\begin{equation}
\partial_t p = \frac{1}{2}\partial_{xx} p + \dot{\Lambda}_t\partial_xp,\quad p(0,\cdot)=f,\quad p(\cdot,0)=0,\quad\text{with}\quad \dot{\Lambda}_t=\frac{\alpha}{2}\partial_x p(t,0),\quad t\in[0,T]
\end{equation}  
(cf.~\cite[proof of proposition 4.2(b)]{NadShk1}). Thus, 
\begin{equation}\label{u_to_p}
u(t,x):=p(t,x-\Lambda_t)
\end{equation}
is a solution in $W^{1,2}_2(\{(t,x)\in[0,T]\times[0,\infty):\,x\ge\Lambda_t\})$ of the supercooled Stefan problem \eqref{eq:Stefan} on the time interval $[0,T]$. For a further elaboration of the connection between the problems \eqref{eq:Stefan} and \eqref{Stefan_prob} we point to the upcoming Remarks \ref{rmk.conn1} and \ref{rmk.conn2}. 

\medskip

Two striking features of the probabilistic problem \eqref{Stefan_prob} are: (i) the necessary presence of discontinuities in $\Lambda$ (leading to blow-ups in the supercooled Stefan problem \eqref{eq:Stefan}) for certain pairs $(X_0,\alpha)$, such as the ones satisfying $\mathbb{E}[X_0]<\alpha/2$ (see \cite[theorem 1.1]{HLS}); (ii) the non-uniqueness of the jump sizes $X_{t-}-X_t:=\lim_{s\uparrow t} X_s-X_t=\Lambda_t-\Lambda_{t-}$ at the times of discontinuity (cf.~\cite[discussion preceding definition 2.2]{DIRT2}, as well as \cite[p.~7, last paragraph]{NadShk2}). The physical choice of the jump sizes $\Lambda_t-\Lambda_{t-}$ in the supercooled Stefan problem \eqref{eq:Stefan} amounts to picking the smallest non-negative numbers so that the total energy of the system is conserved. On the other hand, the interpretation of the probabilistic problem \eqref{Stefan_prob} in neuroscience, finance and probability theory motivates the selection of each $X_{t-}-X_t$ as the smallest non-negative number that allows for a right-continuous continuation of $X_s$, $s\in[0,t)$ to $[t,\infty)$ (cf.~\cite[paragraph preceding definition 2.2]{DIRT2}, \cite[p.~7, last paragraph]{NadShk2} and \cite[p.~2, last paragraph]{DT}). A straightforward adaptation of \cite[proposition 2.7, theorem 4.4 and remark 4.5]{DIRT2} to the setting of \eqref{Stefan_prob} shows that both of these minimality conventions result in 
\begin{equation}\label{phys_cond}
X_{t-}-X_t=\inf\Big\{x>0:\;\mathbb{P}\big(\tau\geq t,\,X_{t-}\in(0,x]\big)<\frac{x}{\alpha}\Big\},\quad t\ge0.  
\end{equation}
We refer to solutions of \eqref{Stefan_prob} fulfilling the jump condition \eqref{phys_cond} as \textit{physical}. The global existence of physical solutions is known under certain (natural) assumptions on the density of the initial condition $X_{0-}>0$, see e.g.~\cite[Subsection 4.2]{DIRT2}, \cite[Theorem 2.3]{NadShk1}, \cite[Theorem 3.2]{LS}.

\medskip

Our first main theorem provides a comprehensive description of the physical solutions $(X,\Lambda)$ to the probabilistic problem \eqref{Stefan_prob}.

\begin{theorem}\label{thm1}
Let $X_{0-}$ possess a density $f$ on $[0,\infty)$ that is bounded and changes monotonicity finitely often on compacts (and, in particular, may and will be assumed to be right-continuous). Then, for any physical solution $(X,\Lambda)$ of \eqref{Stefan_prob} started from $X_{0-}$, and for any $t>0$, the density $\rho(t,\cdot)$ of the restriction of the distribution of $X_{t-}\,\bone_{\{\tau \geq t\}}$ to $(0,\infty)$ is real analytic on $(0,\infty)$ and possesses the properties of $f$ stated above on $[0,\infty)$.
Moreover, every $t\geq0$ falls into exactly one of the three categories: 
\begin{enumerate}[(i)]
\item If $\limsup_{x\downarrow0} x^{-1}\rho(t,x)<\infty$, then $\Lambda\in C^1([t,t+\epsilon))$ for some $\epsilon>0$.\smallskip
\item If $\limsup_{x\downarrow0} x^{-1}\rho(t,x)=\infty$ but $\lim_{x\downarrow0} \rho(t,x)<\frac{1}{\alpha}$, then $\Lambda$ is $1/2$-H\"older continuous on $[t,t+\epsilon)$ for some $\epsilon>0$. \smallskip
\item If $\lim_{x\downarrow0} \rho(t,x)\ge\frac{1}{\alpha}$, then
\begin{equation}\label{jump_size}
\Lambda_t-\Lambda_{t-}=-(X_t-X_{t-})=\inf\Big\{x>0:\;\mathbb{P}\big(\tau\geq t,\,X_{t-}\in(0,x]\big)<\frac{x}{\alpha}\Big\}\ge 0. 
\end{equation}
\end{enumerate}
In all cases, there exists an $\epsilon>0$ such that $\Lambda\in C^1((t,t+\epsilon))$ and the densities $p(s,\cdot)$, $s\in(t,t+\epsilon)$ of the restrictions of the distributions of $X_s\,\bone_{\{\tau\geq s\}}$ to $(0,\infty)$, $s\in(t,t+\epsilon)$ are real analytic on $(0,\infty)$ and form a classical solution of the Dirichlet problem 
\begin{equation}\label{Dir_problem}
\partial_t p=\frac{1}{2}\partial_{xx}p+\dot{\Lambda}_t\partial_xp,\;\; p(\cdot,0)=0\;\;\text{on}\;\;(t,t+\epsilon),\quad\text{with}\quad \dot{\Lambda}_s=\frac{\alpha}{2}\partial_x p(s,0),\;s\in(t,t+\epsilon).
\end{equation}
\end{theorem}

\begin{remark}\label{rmk.conn1}
The interpretation of Theorem \ref{thm1} goes as follows. For all $t\ge0$, there exists a non-trivial open interval $(t,t+\epsilon)$ on which the densities $p(s,\cdot)$, $s\in(t,t+\epsilon)$ evolve according to \eqref{Dir_problem}, thus, the corresponding $u(s,\cdot)$ combine to a classical solution of the supercooled Stefan problem \eqref{eq:Stefan} on these intervals. For $t\ge0$ as in item (i), the continuous differentiability of the free boundary $\Lambda$ and the classical solution $u$ extend to $[t,t+\epsilon)$. In contrast, items (ii), (iii) address the blow-ups in the supercooled Stefan problem: at the times $t\ge0$ of items (ii) and (iii) with $\Lambda_t-\Lambda_{t-}=0$, the free boundary $\Lambda$ has infinite speed, but remains continuous, and the solution immediately returns to the classical regime; at the times $t\ge0$ of item (iii) with $\Lambda_t-\Lambda_{t-}>0$, the free boundary $\Lambda$ has infinite speed and here it triggers the minimal discontinuity of the free boundary ensuring the conservation of the total energy in the system, as encapsulated by \eqref{jump_size}. The discontinuity is  succeeded by an immediate comeback to the classical solution regime. We observe that the set of discontinuity times is countable but, in principle, may have accumulation points.
Finally, it is worth mentioning that, while there exist several local results connecting \eqref{Stefan_prob} to the Stefan PDE \eqref{Dir_problem} (see e.g.~\cite{DIRT1}, \cite{NadShk1}, \cite{HLS}), Theorem \ref{thm1} is the first result establishing such a connection for \textit{all} times $t$: indeed, it shows that the density of a physical solution to \eqref{Stefan_prob}, at any time $t$, can be viewed as the boundary value of a classical solution to the associated PDE \eqref{Dir_problem}.
\end{remark}  

\begin{remark}
Item (iii) in Theorem \ref{thm1} can be split further into two sub-items. If $\rho(t,\cdot)\geq\frac{1}{\alpha}$ on a right neighborhood of $0$, then $\Lambda_t-\Lambda_{t-}>0$. If $\lim_{x\downarrow0} \rho(t,x)=\frac{1}{\alpha}$ but $\rho(t,\cdot)<\frac{1}{\alpha}$ on a non-trivial interval $(0,\delta)$, then $\Lambda$ is right-continuous but not $1/2$-H\"older right-continuous at $t$ (unlike item (ii) in Theorem \ref{thm1}). In the latter situation, the magnitude of the increments $\Lambda_s-\Lambda_t$, for sufficiently small $s>t$, is controlled by the decay of $\rho(t,\cdot)$ near $0$, as can be inferred from the proofs of Propositions \ref{prop:sec2.p0.upperbd} and \ref{prop:Holder.1} below.
\end{remark}

Our second main theorem guarantees the uniqueness of the physical solution for any fixed initial condition. 

\begin{theorem}\label{thm:uniq}
Under the assumptions of Theorem \ref{thm1}, the physical solution $(X,\Lambda)$ of \eqref{Stefan_prob} started from $X_{0-}$ is unique.
\end{theorem}

\begin{remark}\label{rmk.conn2}
In \cite{DT}, the authors study an interacting particle system on the non-negative integers which can be regarded as a discretization of the problem \eqref{Stefan_prob}. More specifically, the negative of the initial temperature profile is discretized into ``heat particles'' subsequently performing independent simple symmetric random walks and advancing a discrete version of the solid-liquid frontier $\Lambda$ via a discrete analogue of \eqref{jump_size}. That is, the discrete solid-liquid frontier of \cite{DT} moves in the minimal fashion preserving the total energy in the system, as dictated by the physics of supercooled liquids. By \cite[theorem 1.6]{DT} and our Theorem \ref{thm:uniq} the scaling limit of the particle system in \cite{DT} gives the unique physical solution of \eqref{Stefan_prob}. Consequently, the latter captures the actual physical notion of solution to the supercooled Stefan problem \eqref{eq:Stefan} in the presence of blow-ups. 
\end{remark}

Most importantly, Theorem \ref{thm:uniq} yields the first global uniqueness result for the supercooled Stefan problem with blow-ups, formulated as in \eqref{Stefan_prob}. In addition, the problem \eqref{Stefan_prob} is expected to describe the critical regime for a one-dimensional multiparticle diffusion limited aggregation process (cf.~\cite[conjecture 1.4]{DT}), and Theorem \ref{thm:uniq} is a crucial step in the rigorous derivation of the scaling limits in such and related settings, beyond the special case treated in \cite{DT}. Furthermore, in view of \cite[proposition 2.2(i)]{Szn}, our Theorem \ref{thm:uniq} at once settles the propagation of chaos for the constant coefficients version of the mean field particle system in \cite[equation (1.2)]{HLS}; and, while we do not pursue this direction here, we are confident that suitable variants of Theorem \ref{thm:uniq} can be established (and, hence, will complete the proof of the propagation of chaos) for the mean field particle systems in \cite[equation (3.1)]{DIRT2}, motivated by integrate-and-fire models in neuroscience, as well as for the ones in \cite[equation (2.6)]{NadShk1} and the full generality of \cite[equation (1.2)]{HLS}, motivated by interbank lending network models in finance.
It is worth mentioning that, while the global existence results for \eqref{Stefan_prob}, and for related systems, have appeared in the existing literature (e.g., in \cite{DIRT2}, \cite{NadShk1}, \cite{LS}), the question of global uniqueness remained open until now. This is due to the challenging nature of the latter problem, which, in particular, requires an understanding of the exact structure of physical solutions, provided by Theorem \ref{thm1}.
Finally, together with Theorem \ref{thm1}, Theorem \ref{thm:uniq} may serve as the basis for the design and investigation of global numerical schemes for the problems \eqref{eq:Stefan} and \eqref{Stefan_prob}, extending the local numerical schemes proposed in \cite{KR}, \cite{LKR}.  

\medskip

The rest of the paper is structured as follows. In Section \ref{se:2}, we prepare a priori H\"older estimates on the function $\Lambda$ and the boundary behavior of the densities $p(t,\cdot)$. Our main tools include a stochastic comparison method for \eqref{Stefan_prob}
and the Krylov-Safonov estimates \cite{KrySaf}.
The bounds of Section \ref{se:2} are then improved in Section \ref{se:3} to Lipschitz estimates on $\Lambda$ and the boundary behavior of the densities $p(t,\cdot)$, away from the times of blow-ups (Subsection \ref{subse:3.1}). The $C^\infty$-property along with the real analyticity in $x$ of $u$ in the interior of $D$ are shown in Subsection \ref{subse:3.2}, and supplementary features of the derivative $\partial_x p$ are deduced in Subsection \ref{subse:3.3}. These rely on the findings in \cite{DIRT3}, \cite{DIRT1}, \cite{HLS}, Weyl's lemma in the form of \cite[p.~90, Step 4]{McK} and the analyticity assertion of \cite[theorem 1]{Kom}. Section \ref{se:4} contains the proof of Theorem \ref{thm1}, which combines the conclusions of Sections \ref{se:2} and \ref{se:3} with a careful analysis of the zero set of $\partial_x u$ in the spirit of \cite[proof of theorem 5.1]{AF}. Lastly, in Section \ref{se:5}, the proof of Theorem \ref{thm:uniq} is carried out by using Theorem \ref{thm1}.     

\bigskip

\noindent\textbf{Acknowledgement.} We thank James Nolen for pointing one of us to the literature on zero sets of solutions to the heat equation, leading to the completion of a substantial step in the proof of Theorem \ref{thm1}.


\renewcommand{\labelitemi}{\textbullet}


\section{H\"older continuity}
\label{se:2}
We assume that we are given a physical solution $X$ satisfying (\ref{Stefan_prob}), with the associated $\Lambda$ and $\tau$.
For any $t\geq 0$, we denote by $p(t,\cdot)$ and $\rho(t,\cdot)$, respectively, 
the densities of
the restrictions of the distributions of $X_t\,\bone_{\{\tau\geq t\}}$ and $X_{t-}\,\bone_{\{\tau\geq t\}}$
to $(0,\infty)$ (in particular, their integrals may be less than one; we sometimes refer to them as sub-densities to emphasize this fact). (Notice that 
$p(t,\cdot)$ is also the density of 
the restriction of the distribution of $X_t\,\bone_{\{\tau > t\}}$ to $(0,\infty)$.) The existence and global boundedness of such densities, for all $t\geq0$, is shown in Lemma 5.1 of \cite{NadShk1}, under the assumption that the initial condition $X_{0-}$ has a bounded density.
At this stage of the paper, $p(t,\cdot)$ and $\rho(t,\cdot)$ must be regarded as mere measurable functions for which we do not have any obvious canonical version. Later, in Proposition \ref{prop:analytic.1}, we will see that, for any $t>0$, $p(t,\cdot)$ and $\rho(t,\cdot)$ have analytic versions on $(0,\infty)$.
In this section, we show that, for any $t\geq0$, under an additional assumption that is verified later in the paper, there exists a neighborhood $(t,t+\epsilon)$ on which the free boundary $\Lambda$ is H\"older continuous and the densities $p(t,\cdot)$ and $\rho(t,\cdot)$ (which, therefore, coincide) are vanishing and H\"older continuous at zero.

\subsection{Upper bound on the marginal density at zero.}
\smallskip
We begin with the following proposition, which shows that, for any time $t$ at which the profile of $\rho(t,\cdot)$ satisfies an additional assumption (which we finally succeed to check in Section \ref{se:4}, for a large class of initial conditions), there exists a neighborhood $(t,t+\epsilon)$ on which the marginal density at zero remains strictly below $1/\alpha$. This, in particular, implies that $\Lambda$ cannot jump in that neighborhood. 

\begin{proposition}\label{prop:sec2.p0.upperbd}
Fix an arbitrary $t\geq0$ and assume that $\rho(t,\cdot)$ satisfies at least one of the following two conditions: (i) 
$\lim_{\eta \downarrow 0} \esssuptext{x \in (0,\eta)} \rho(t,x) <1/\alpha$, or (ii) 
$\rho(t,\cdot)$ has a version that is locally monotone in a right neighborhood of any point in $[0,\infty)$.
Then, there exist $\epsilon,\,\delta>0$ and $\beta: (0,\epsilon) \rightarrow [0,1)$ such that, for any $z\in(0,\epsilon)$,
$$
\PP(\tau\geq s,\,X_{s-}\leq x) \leq \frac{\beta(z)}{\alpha} x
$$
holds for all $x\in[0,\delta]$ and all $s\in[t+z,t+\epsilon]$. In case $(i)$, 
$\beta$ has an extension to $[0,\epsilon)$ (with values in $[0,1)$) and the above bound is true for $z=0$.
\end{proposition}

\begin{proof}
Let $\widetilde{B}_s:=B_s-B_t$ and $\widetilde{\Lambda}_s:= \Lambda_s - \Lambda_t$, for $s\geq t$.
We also recall a useful (elementary) identity
$$
\PP(X_{s-}\leq x,\,\tau\geq s)
=\PP(X_{s-}\leq x,\,\inf_{r\in[0,s)}X_r >0).
$$

\textit{First case}.
Assume that $\lim_{\eta \downarrow 0} \esssuptext{x \in (0,\eta)} \rho(t,x) <1/\alpha$.
Note from \eqref{phys_cond} that, in this case, $X_t=X_{t-}$ and $p(t,\cdot)=\rho(t,\cdot)$.
For any $s\geq t$, we will use the following bound:
\begin{equation*}
\PP \bigl( X_{s-} \leq x , \inf_{r \in [0,s)} X_{r} >0\bigr) 
\leq \PP \Bigl( \widetilde{\Lambda}_{s-} \leq X_{t} {\mathbf 1}_{\{\tau \geq t\}}+ \widetilde{B}_{s} \leq x + \widetilde{\Lambda}_{s-}, \,X_{t} {\mathbf 1}_{\{\tau \geq t\}} >0 \Bigr),
\end{equation*}
which follows from the definition of a physical solution (\ref{Stefan_prob}).
Then, 
\begin{equation*}
\begin{split}
\PP \bigl( X_{s-} \leq x , \inf_{r \in [0,s)} X_{r} >0\bigr) 
&\leq  \PP \Bigl( \widetilde{\Lambda}_{s-} \leq X_{t} {\mathbf 1}_{\{\tau \geq t\}}+ \widetilde{B}_{s} \leq x + \widetilde{\Lambda}_{s-}, \,X_{t} {\mathbf 1}_{\{\tau \geq t\}} >0 \Bigr)\\
&= \int_{\RR}
\PP \Bigl( \widetilde{\Lambda}_{s-} \leq X_{t} {\mathbf 1}_{\{\tau \geq t\}}+ y \leq x + \widetilde{\Lambda}_{s-},  \,X_{t} {\mathbf 1}_{\{\tau \geq t\}} >0 \Bigr)g(s\!-\!t,y)\,\mathrm{d}y,
\end{split}
\end{equation*}
where $g(s,\cdot)$ is the Gaussian kernel of variance $s$ (and zero mean). Let $F$ be the cumulative distribution function of $\rho(t,\cdot)$. 
Then,
\begin{equation*}
\begin{split}
\PP \bigl( X_{s-} \leq 
x , \inf_{r \in [0,s)} X_{r} >0\bigr) 
&\leq
\int_{\RR}
 \bigl( F(x + \widetilde{\Lambda}_{s-} - y) - F(\widetilde{\Lambda}_{s-} - y) \bigr) g(s-t,y)\,\mathrm{d}y
\\
&= \int_{-\infty}^{x+\widetilde{\Lambda}_{s-}}
 \bigl( F(x + \widetilde{\Lambda}_{s-} - y) - F(\widetilde{\Lambda}_{s-} - y) \bigr) g(s-t,y)\,\mathrm{d}y.
\end{split}
\end{equation*}
We split the above term into three parts:
\begin{equation}
\label{eq:main:one}
\begin{split}
\PP \bigl( X_{s-} \leq x , \inf_{r \in [0,s)} X_{r} >0\bigr) 
&\leq \int_{\widetilde{\Lambda}_{s-}}^{x+\widetilde{\Lambda}_{s-}}
 \bigl( F(x + \widetilde{\Lambda}_{s-} - y) - F(\widetilde{\Lambda}_{s-} - y) \bigr) g(s-t,y)\,\mathrm{d}y
 \\
 &\hspace{15pt} 
 +  \int_{-\varepsilon}^{\widetilde{\Lambda}_{s-}}
 \bigl( F(x + \widetilde{\Lambda}_{s-} - y) - F(\widetilde{\Lambda}_{s-} - y) \bigr) g(s-t,y)\,\mathrm{d}y
 \\
&\hspace{15pt} +
 \int_{-\infty}^{-\varepsilon}
 \bigl( F(x + \widetilde{\Lambda}_{s-} - y) - F(\widetilde{\Lambda}_{s-} - y) \bigr) g(s-t,y)\,\mathrm{d}y.
\end{split}
\end{equation}

\smallskip

The first term on the right-hand side of \eqref{eq:main:one} is less or equal to 
\begin{equation*}
F(x) \int_{\widetilde{\Lambda}_{s-}}^{x+\widetilde{\Lambda}_{s-}} g(s-t,y)\,\mathrm{d}y
\leq C_1 x\, \int_{\widetilde{\Lambda}_{s-}}^{x+\widetilde{\Lambda}_{s-}} g(s-t,y)\,\mathrm{d}y,
\end{equation*}
for all $s\geq t$ and all $x\in[0,\delta]$, where $C_1<1/\alpha$ and $\delta>0$ are chosen so that 
$\esssuptext{(0,\delta)} \rho(t,\cdot)\leq C_1$ (which is possible due to 
$\lim_{\eta \downarrow 0} \esssuptext{x \in (0,\eta)} \rho(t,x) <1/\alpha$). 
\smallskip

As for the second term on the right-hand side of (\ref{eq:main:one}), we choose $\epsilon$ to be sufficiently small, so that $\widetilde{\Lambda}_s\leq \delta/3$, for all $s\in[t,\,t+\epsilon]$, $x\in(0,\delta/3]$, and $\varepsilon\in(0,\delta/3]$ (here, we also use the right-continuity of $\Lambda$).
Then, $\esssuptext{(0,x+\widetilde{\Lambda}_{s}+\varepsilon)} \rho(t,\cdot)\leq C_1$
and
\begin{equation*}
F\bigl(x + \widetilde{\Lambda}_{s-} - y\bigr) - F\bigl(\widetilde{\Lambda}_{s-} - y\bigr)
\leq C_1 x, 
\end{equation*}
hence, the second term on the right-hand side of \eqref{eq:main:one} is less or equal to 
$$
C_1 x \int_{-\varepsilon}^{\widetilde{\Lambda}_{s-}} g(s-t,y)\,\mathrm{d}y.
$$

\smallskip

Consider the last term on the right-hand side of (\ref{eq:main:one}).
Due to the fast decay, as $s\downarrow t$, of $g(s-t,x)/g(s-t,y)$, for $x<y<0$, 
$$
\int_{-\infty}^{-2\varepsilon} g(s-t,y)\,\mathrm{d}y
\leq e^{-\varepsilon^2/(2(s-t))} \int_{-\infty}^{-\varepsilon} g(s-t,y)\,\mathrm{d}y,
$$
so that, decreasing if necessary $\epsilon=\epsilon(\varepsilon)>0$, we obtain
$$
\int_{-\infty}^{-2\varepsilon} g(s-t,y)\,\mathrm{d}y
\leq \gamma \int_{
-\infty}^{-\varepsilon} g(s-t,y)\,\mathrm{d}y,
$$
for all $s\in[t,t+\epsilon]$, with $\gamma$ being small enough, so that (using the global boundedness of the density)
 $\gamma\,\|\rho(t,\cdot)\|_{L^{\infty}} + C_1<1/\alpha$.
Then, 
for all $s\in[t,\,t+\epsilon]$, $x\in(0,\delta/3]$, and $\varepsilon=\delta/6$,
\begin{equation*}
\begin{split}
&\int_{-\infty}^{-\varepsilon}
\bigl( F(x + \widetilde{\Lambda}_{s-} - y) - F(\widetilde{\Lambda}_{s-} - y) \bigr) g(s-t,y)\,\mathrm{d}y
\\
&\quad \leq   \|\rho(t,\cdot)\|_{L^{\infty}}\, x\, \int_{-\infty}^{-2\varepsilon} g(s-t,y)\,\mathrm{d}y
+ C_1\, x\, \int_{-2\varepsilon}^{-\varepsilon} g(s-t,y)\,\mathrm{d}y,
\\
&\quad \leq x\,\big(\gamma\,\|\rho(t,\cdot)\|_{L^{\infty}} + C_1 \big) \int_{
-\infty}^{-\varepsilon} g(s-t,y)\,\mathrm{d}y.
\end{split}
\end{equation*}
Collecting the above, we conclude that there exist $\epsilon>0$ and $C_2<1/\alpha$ such that 
\begin{equation*}
\PP \bigl( X_{s-} \leq x , \inf_{r \in [0,s)} X_{r} >0\bigr) \leq C_2\, x
\end{equation*}
holds for all $s\in[t,\,t+\epsilon]$ and $x\in(0,\delta/3]$.
Thus, the statement of the proposition holds with $\beta:=C_{2} \alpha$.

\medskip

\textit{Second case}.
Assume now that $\rho(t,\cdot)$
has a version that is locally monotone in a right neighborhood of any point in $[0,\infty)$. Then, without loss of generality we can assume that it is right-continuous. Resolving the jump (if it occurs), we switch from $\rho(t,\cdot)$ to $p(t,\cdot)$. Since this transition amounts to a shift of variables, we conclude from the assumption that 
$$
p(t,x) = p(t,0) - \psi(x),\quad x\geq0,
$$ 
where $\psi$ is monotone in a right neighborhood of zero (say $[0,\delta]$). By right-continuity, $\lim_{x \downarrow 0} \psi(x)=0$. Obviously, if $p(t,0)<1/\alpha$, we are led back to the first case. If $p(t,0)=1/\alpha$, the fact that we have resolved the jump forces $\psi$ to be non-decreasing and strictly positive in a right neighborhood of zero.
Because of a possible jump at time $t$, the inequality \eqref{eq:main:one} holds provided we now denote by $F$ the cumulative distribution function of $p(t,\cdot)$.
Let us estimate the terms on the right-hand side of \eqref{eq:main:one}.

\smallskip 
 
Repeating the same arguments as in the first case, we conclude that, for any $\varepsilon\in(0,\delta/6]$, there exists an $\epsilon=\epsilon(\varepsilon)$ such that,
for all $s\in[t,t+\epsilon]$ and all $x\in[0,\delta/3]$, $\widetilde{\Lambda}_{s-}\leq \delta/3$ and the first and the last terms on the right-hand side of (\ref{eq:main:one}) add up to at most
$$
\frac{x}{\alpha} \bigg( \int_{-\infty}^{-\varepsilon} g(s-t,y)\,\mathrm{d}y + \int_{\widetilde{\Lambda}_{s-}}^{x+\widetilde{\Lambda}_{s-}} g(s-t,y)\,\mathrm{d}y \bigg).
$$

\smallskip

It only remains to estimate the second term on the right-hand side of (\ref{eq:main:one}). Since $\psi$ is decreasing on $[0,\delta]$, we have:
\begin{equation*}
\begin{split}
F(x+\widetilde{\Lambda}_{s-}-y) - F(\tilde{\Lambda}_{s-}-y) &= \int_{0}^{x} p(t,z+\widetilde{\Lambda}_{s-} - y)\,\mathrm{d}z
\\
&= x\, p(t,0) - \int_0^x \psi(z+\widetilde{\Lambda}_{s-}-y)\,\mathrm{d}z
\leq \frac{x}{\alpha} - x \psi(\widetilde{\Lambda}_{s-}-y),
\end{split}
\end{equation*}
for all $s\in[t,t+\epsilon]$, all $x\in[0,\delta/3]$, and all $y\in[-\delta/3,\widetilde{\Lambda}_{s-}]$. Thus, for any $\varepsilon\in(0,\delta/3]$, we conclude that the second term on the right-hand side of (\ref{eq:main:one}) is less or equal to
$$
x\bigg(\frac{1}{\alpha}\int_{-\varepsilon}^{\widetilde{\Lambda}_{s-}} g(s-t,y)\,\mathrm{d}y - \int_{-\varepsilon}^{\widetilde{\Lambda}_{s-}} \psi(\widetilde{\Lambda}_{s-}-y)\,g(s-t,y)\,\mathrm{d}y\bigg),
$$
for all $s\in[t,t+\epsilon]$ and all $x\in[0,\delta/3]$.
Fixing $\varepsilon=\delta/6$, we notice that, for any $z\in(0,\epsilon)$, there exists $h(z)>0$, such that
$$
\int_{-\varepsilon}^{\widetilde{\Lambda}_{s-}} \psi(\widetilde{\Lambda}_{s-}-y)\,g(s-t,y)\,\mathrm{d}y
\geq \int_{0}^{\varepsilon} \psi(y)\,
g(s-t,\widetilde{\Lambda}_{s-}-y)\,\mathrm{d}y\geq h(z)
$$
holds for all $s\in[t+z,t+\epsilon]$.
Thus, we conclude that, for any $z\in(0,\epsilon)$,
$$
\PP \bigl( X_{s-} \leq x , \inf_{r \in [0,s)} X_{r} >0\bigr) \leq x\,\bigg(\frac{1}{\alpha} - h(z) \bigg)
$$
holds for  all $x\in[0,\delta/3]$ and all $s\in[t+z,t+\epsilon]$.
Thus, the statement of the proposition holds with $\beta(z)=(1-\alpha h(z))^+$.
\end{proof}

\subsection{H\"older continuity of the free boundary.}
\smallskip
Next, we show that, whenever the marginal sub-density $\rho$ at zero is strictly below $1/\alpha$ on a given time interval, the free boundary $\Lambda$ is 1/2-H\"older continuous on the same interval.

\begin{proposition}\label{prop:Holder.1}
Fix an arbitrary $t\geq0$ and assume that there exist $\epsilon,\,\delta>0$ and $\beta\in[0,1)$ such that
$$
\PP(\tau\geq s,\,X_{s-}\leq x) \leq \frac{\beta}{\alpha} x,
$$
for all $x\in[0,\delta]$ and all $s\in[t,t+\epsilon]$.
Then, $\Lambda$ is 1/2-H\"older continuous in $[t,t+\epsilon]$.
\end{proposition}

\begin{proof}
The proof follows the strategy outlined in Section 5 of \cite{NadShk1}, 
and relies on a sequence of auxiliary processes, whose limit will be shown to dominate the physical solution.


It is clear that $\PP(\tau\geq t)>0$ for all $t\geq0$, whenever $\PP(\tau>0)>0$. Since the statement of the proposition holds trivially if $\PP(\tau>0)=0$, we assume that $\PP(\tau\geq t)>0$ for all $t\geq0$. Let us fix an arbitrary $\varepsilon\in(0,\epsilon)$ and consider the sequence of processes $X^n$, $n\in\nn$ defined recursively as follows:
\begin{eqnarray}
&& X^1_s = (X_{t-} + \widetilde{B}_{s})\,\bone_{\{\tau\geq t\}},\;\; s\in[0,\varepsilon], \\
&&X^n_s = (X_{t-} + \widetilde{B}_s - L^{n-1})\,\bone_{\{\tau\geq t\}},\;\;s\in[0,\varepsilon],\;\;n\ge2, \\
&& L^n
= \alpha \PP(\tau\geq t) - \alpha \PP\bigl(\tau\geq t,\,\inf_{s\in[0,\varepsilon]} X^n_s > 0 \bigr),
\;\; n\geq 1, \label{eq.PhysSolReg.Ln.def}
\end{eqnarray}
where $\widetilde{B}_s:=B_{t+s}-B_t$, $s\in[0,\varepsilon]$. (Note that $L^n$ does not depend on the time parameter $s$.)

It is easy to see that $X^2_s\leq X^1_s$, for all $s\in[0,\varepsilon]$, with probability one. Then, by induction, we conclude that the sequences $X^n_s$, $n\in\nn$ are non-increasing, for all $s\in[0,\varepsilon]$, with probability one.
Hence, by Lemma \ref{le:PhysSolReg.4} below, for $\varepsilon\in(0,\infty)$ sufficiently small, the sequence $L^n$, $n\in\nn$, (which is non-decreasing by \eqref{eq.PhysSolReg.Ln.def}) has a limit $\widetilde{L}$. 
Hence, the processes $X^n$, $n\in\nn$ converge uniformly on $[0,\varepsilon]$ to the process $\widetilde{X}$ satisfying 
\begin{equation}\label{eq.PhysSolReg.tildeY.def}
\widetilde{X}_s=(X_{t-} + \widetilde{B}_{s} - \widetilde{L})\,\bone_{\{\tau\geq t\}},\;\; s\in[0,\varepsilon],
\end{equation}
where $\widetilde{L} := \lim_{n\to\infty} L^n$.
Notice that $\inf_{s\in[0,\varepsilon]} X^n_s$, $n\in\nn$ tend almost surely to $\inf_{s\in[0,\varepsilon]} \widetilde{X}_s$.
Since the conditional distribution of the latter random variable, given $\{\tau\ge t\}$, has no atoms, we conclude that 
\begin{equation}
\lim_{n\rightarrow\infty}\,\PP\bigl(\tau\geq t,\,\inf_{s\in[0,\varepsilon]} X^n_s > 0\bigr)
= \PP\bigl(\tau\geq t,\,\inf_{s\in[0,\varepsilon]} \widetilde{X}_s > 0\bigr),
\end{equation}
which yields 
\begin{equation}\label{eq.PhysSolReg.tildeL.def}
\widetilde{L}
= \alpha \PP(\tau\geq t) - \alpha \PP\bigl(\tau\geq t,\,\inf_{s\in[0,\varepsilon]}\widetilde{X}_s > 0\bigr).
\end{equation}
By Lemma \ref{le:PhysSolReg.4}, there exist $C_L<\infty$ and $\varepsilon_0\in(0,\epsilon]$ such that
\begin{equation}\label{eq.PhysSolReg.tildeL.Holder}
\widetilde{L} \leq C_L\,\sqrt{\varepsilon},\quad\varepsilon\le\varepsilon_0.
\end{equation}

\smallskip

Combining (\ref{eq.PhysSolReg.tildeL.Holder}) and (\ref{eq.sec2.deltaLambda.est}) in the statement of Lemma 
\ref{le:PhysSolReg.3} below, and recalling that they hold for any $\varepsilon\in(0,\varepsilon_0]$, we conclude that
\begin{equation}
\label{eq:21}
\Lambda_{t+s} - \Lambda_{t} \leq C_L\,\sqrt{s},\;\;
s\in[0,\varepsilon_0].
\end{equation}
The statement of the proposition follows by repeating the above arguments for arbitrary $t'\in[t,t+\epsilon)$ in place of $t$ and recalling that $C_L$ and $\varepsilon_0$ can be chosen independently of $t'$.
\end{proof}

\begin{lemma}\label{le:PhysSolReg.4}
Let the assumptions of Proposition \ref{prop:Holder.1} hold. Then, there exist $C_L<\infty$ and $\varepsilon_0\in(0,\epsilon]$, depending only on $\alpha$, $\beta$, $\delta$, and $\|\rho(t,\cdot)\|_{L^\infty([0,\infty))}$, such that 
\begin{equation}
0\leq L^n \le C_L\,\sqrt{\varepsilon}, 
\end{equation}
for all $n\in\nn$ and all $\varepsilon\in(0,\varepsilon_0]$.
\end{lemma}

\begin{proof} 
We have the estimates 
\begin{equation} \label{eq.sec2.le1.eq1}
\begin{split}
0&\leq \PP(\tau\ge t) - \PP\bigl(\tau\geq t,\,\inf_{s\in[0,\varepsilon]} X^1_s > 0 \bigr)
 = \int_0^\infty \Big(1 - \PP\bigl(\inf_{s\in[0,\varepsilon]} \widetilde{B}_s > -y\bigr)\Big)\,\rho(t,y)\,\mathrm{d}y \\
&\leq 2\sqrt{\varepsilon} \int_0^{\infty} \Phi\left(- y\right)\,\rho(t,y\sqrt{\varepsilon})\,\mathrm{d}y
\leq 2\sqrt{\varepsilon}\,C_{\rho} \int_0^{\infty} \Phi\left(- y\right)\,\mathrm{d}y
=\sqrt{\varepsilon}\,C_\rho\,\sqrt{\frac{2}{\pi}}=: \sqrt{\varepsilon}\,C_0,
\end{split}
\end{equation}
where
$$
C_{\rho}:=\|\rho(t,\cdot)\|_{L^\infty},
$$
and $\Phi$ stands for the standard Gaussian cumulative distribution function.

For $n\geq 2$, we find
\begin{equation}\label{eq.PhysSolReg.4.eq1}
\begin{split}
\frac{1}{\alpha}L^n
&= \PP(\tau\ge t) - \PP\bigl(\tau\geq t,\,\inf_{s\in[0,\varepsilon]} X^n_s > 0 \bigr) \\
& = \int_0^\infty \Big(1-\PP\bigl(\inf_{s\in[0,\varepsilon]} \widetilde{B}_s - L^{n-1}> -y\bigr) \Big)\,\rho(t,y)\,\mathrm{d}y \\
& \leq 
\int_0^{L^{n-1}} \rho(t,y)\,\mathrm{d}y
+2\sqrt{\varepsilon} \int_0^{\infty} \Phi\left( - y\right)\,\rho(t,y\sqrt{\varepsilon} +L^{n-1})\,\mathrm{d}y
\\ 
& \leq 
\int_0^{L^{n-1}} \rho(t,y)\,\mathrm{d}y
+\sqrt{\varepsilon} \,C_{0}.
\end{split}
\end{equation}
Assume that $\varepsilon$ is sufficiently small, so that $\alpha C_{0} \sqrt{\varepsilon}\leq (1-\beta)\delta$.
Then, (\ref{eq.sec2.le1.eq1}) implies $L^1\leq \delta$ and, hence, by the main assumption in the statement of Proposition \ref{prop:Holder.1},
$$
\int_0^{L^{1}} \rho(t,y)\,\mathrm{d}y \leq \frac{\beta}{\alpha} L^1,
$$
and
$$
L^2 \leq \beta\, L^1 + \alpha C_{0} \sqrt{\varepsilon}
\leq \delta.
$$
Thus, by induction, $L^n\leq\delta$, for all $n$. Repeating the above estimate, we obtain
$$
L^n \leq \beta\, L^{n-1} + \alpha C_{0} \sqrt{\varepsilon},\quad n\ge 2,
$$
which yields 
$$
L^n \leq  \alpha\, C_0 \, \sqrt{\varepsilon} \left(1 + \frac{1}{1-\beta} \right),\quad n\ge 1
$$
and completes the proof of the lemma. 
\end{proof}

\begin{lemma}\label{le:PhysSolReg.3}
Let the assumptions of Proposition \ref{prop:Holder.1} hold. 
Then,
\begin{equation}
\label{eq.sec2.deltaLambda.est}
\Lambda_{t+s} - \Lambda_{t} \leq \widetilde{L},\;\;
s\in[0,\epsilon].
\end{equation}
\end{lemma}

\begin{proof}
First, we notice that $\Lambda$ is continuous on $[t,t+\epsilon]$ and $\Lambda_{t-}=\Lambda_t$, due to the assumption of the proposition and \eqref{phys_cond}.
Suppose that there exists an $s\in[0,\epsilon]$ such that $\Lambda_{t+s}-\Lambda_t > \widetilde{L}$. Since $\widetilde{L}>0$, we must have $s>0$.
Due to the continuity of $\Lambda$ we can further find $s' \in [0,\epsilon)$ such that $\Lambda_{t+s'} - \Lambda_t = \widetilde{L}$ and $\Lambda_{t+s''}- \Lambda_t < \widetilde{L}$ for all $s''\in[0,s')$. Therefore, for any $s''\in[0,s']$, the definitions of $X$, $\widetilde{X}$, and the properties of Brownian motion, give
\begin{equation*}
\mathbf{1}_{\{\tau>t+s''\}} - \mathbf{1}_{\{\tau\geq t,\,\inf_{r\in[0,\varepsilon]}\widetilde{X}_r > 0\}} \geq 0,
\quad \PP\bigl(\mathbf{1}_{\{\tau>t+s''\}} - \mathbf{1}_{\{\tau\geq t,\,\inf_{r\in[0,\varepsilon]} \widetilde{X}_r > 0\}} >0 \bigr)>0.
\end{equation*}
Taking $s''=s'$ and taking expectations in the left inequality, we end up with $\Lambda_{t+s'} - \Lambda_t < \widetilde{L}$, which is the desired contradiction.
\end{proof}

\subsection{H\"older continuity of the marginal density at zero.}
\label{subse:regLambda.1}
\smallskip
Finally, we show that the 1/2-H\"older continuity of $\Lambda$ implies the H\"older continuity of the marginal density at $0$.

\begin{proposition}
\label{prop:reg:Holder}
Fix an arbitrary $t\geq0$ and assume that there exists an $\epsilon>0$ such that $\Lambda$ is 1/2-H\"older continuous in $(t,t+\epsilon)$.
Then, for any $\eta\in(0,\epsilon/2)$, there exist constants $C \geq 0$ and $\chi \in (0,1)$ such that
\begin{equation*}
p(s,x) \leq C x^{\chi}
\end{equation*}
holds for all $s \in [t+\eta,t+\epsilon-\eta]$ and almost every $x>0$.
\end{proposition}


\begin{proof}
As before, we assume without loss of generality that ${\mathbb P}(\tau>s)>0$ holds for all $s\geq0$.

\textit{First Step}.
The strategy is based upon Krylov and Safonov estimates, as implemented in the proof of Lemma 5.5 in \cite{DIRT2}. 
However, there is a significant difference with the proof in \cite{DIRT2} since, at this stage, the function $\Lambda$ is not known to be differentiable on $(t,t+\epsilon)$: we only know that it is $1/2$-H\"older continuous.
To overcome the lack of differentiability of $\Lambda$, we use the following mollification argument. 
For every $n \geq 1$, we 
choose $\Lambda^n$ as an increasing smooth process on $[0,\infty)$, starting from $0$ at time $0$, such that, 
for any $s \in (t,t+\epsilon)$, $\lim_{n \rightarrow \infty} \Lambda_{s}^n = \Lambda_{s}$, the convergence being uniform on any compact subset of $(t,t+\epsilon)$. Without any loss of generality, we can assume that 
the processes $\{\Lambda^n\}_{n \geq 1}$ are uniformly $1/2$-H\"older continuous on $[t+\eta/2,t+\epsilon-\eta/2]$.

	
Then, for any $n \geq 1$, we let
\begin{equation}
\label{eq:Holder:proof:Xn}
X_{s}^n := X_{t+\eta/2} - \bigl( \Lambda^n_{s} - \Lambda^n_{t+\eta/2}\bigr)  + B_{s} - B_{t+\eta/2}, \quad s \in [t+\eta/2,t+\epsilon-\eta/2],
\end{equation}
together with $\tau^n:= \inf\{ s \geq t+\eta/2 : X_{s}^n \leq 0\}$. 
Using standard arguments, it is easy to deduce that, for any $s \in (t+\eta/2,t+\epsilon-\eta/2)$, the restriction of the distribution of $X_{s \wedge \tau^n}^n$ to $(0,\infty)$
admits a (sub-)density $p^n(s,\cdot)$, which satisfies the Fokker-Planck equation:
\begin{equation}
\label{eq:fokker:planck:regularized}
\partial_{s} p^n- \frac12 \partial_{xx} p^n - \dot{\Lambda}^n \partial_{x} p^n = 0, \quad 
(s,x) \in (t+\eta/2,t+\epsilon-\eta/2)\times(0,\infty),
\end{equation}
where $\dot{\Lambda}^n$ is the time derivative of the regularized function $\Lambda^n$. Recall that 
$$
p^n\in C^{1,2}\left((t+\eta/2,t+\epsilon-\eta/2) \times (0,\infty)\right)
$$
and is continuous on $(t+\eta/2,t+\epsilon-\eta/2) \times [0,\infty)$, 
with the Dirichlet boundary condition $p^n(s,0)=0$ for $s \in (t+\eta/2,t+\epsilon-\eta/2)$.

\medskip
	
\textit{Second Step}.
We now prove that, for any $\eta\in(0,\epsilon/2)$, there exist two positive constants 
$C$ and $\chi$ such that, for any $n \geq 1$, $s \in [t+\eta,t+\epsilon-\eta]$, and $x >0$,  
\begin{equation}
\label{eq:bound:holder:decay}
p^n(s,x) \leq C \min(1,x^\chi).
\end{equation}
In order to prove \eqref{eq:bound:holder:decay}, we fix arbitrary $t_{0} \in (t+\eta/2,t+\epsilon-\eta/2]$ and $x_{0} >0$, and consider, for any $n \geq 1$, the process $Y^n$ given by:
\begin{equation*}
\mathrm{d}Y_{s}^n = \dot{\Lambda}_{t_{0}-s}^n\,\mathrm{d}s + \mathrm{d}B_{s}, \quad s \in [0,t_0-t-\eta/2],
\end{equation*}
with $Y_{0}^n=x_{0}$ as the initial condition. 	
Using \eqref{eq:fokker:planck:regularized}, we deduce from It\^o's formula that 
\begin{equation}
\label{eq:feynman:kac:p}
p^n(t_{0},x_{0}) = {\mathbb E}
\bigl[ p^n (t_{0} - \varrho^n ,Y^n_{\varrho^n})\bigr]
= {\mathbb E} \bigl[ p^n(t_{0} - \varrho^n,Y^n_{\varrho^n}) {\mathbf 1}_{\{Y^n_{\varrho^n} >0\}} \bigr],
\end{equation}
where $\varrho^n$ is any stopping time not exceeding $\varrho_{0}^n \wedge \delta^2$,
with an arbitrary fixed $\delta^2 \in (0,t_{0}-t-\eta/2)$ and with $\varrho_{0}^n:=\inf\{ s>0 : Y_{s}^n \leq 0\}$. 

\smallskip	

Next, we consider another free parameter $L \geq 1$, whose value is determined below in terms of $\eta$, and choose
$\varrho^n = \inf\{s > 0 : Y_{s}^n \geq L \delta \}\wedge \varrho^n_0 \wedge \delta^2$. Then, \eqref{eq:feynman:kac:p} yields
\begin{equation*}
p^n(t_{0},x_{0}) \leq \bigl( 1 - {\mathbb P}(Y^n_{\varrho^n}=0) \bigr) 
\sup_{(s,y) \in {\mathcal Q}(\delta,L)} p^n(s,y),
\end{equation*}
with 
\begin{equation*}
{\mathcal Q}(\delta,L) := [t_{0} - \delta^2,t_{0}] \times [0,L \delta]. 
\end{equation*}
Denoting by $\kappa$ a common $1/2$-H\"older bound of the paths $(\Lambda^n)_{n \geq 1}$ on $[t+\eta/2,t+\epsilon-\eta/2]$, we have, for any $s \in [0,\varrho^n]$ and $x_0\leq \delta$,
\begin{equation*}
Y_{s}^n \leq \delta + \kappa \delta + B_s.
\end{equation*}
Therefore,
\begin{equation*}
\begin{split}
\{Y_{\varrho^n}^n = 0\} \supset &\Bigl\{ \inf_{0 \leq s \leq \delta^2}  B_{s}  < -(1+\kappa) \delta \Bigr\} 
\cap
\Bigl\{ \sup_{0 \leq s \leq \delta^2}  B_{s}  < (L - (1+\kappa)) \delta \Bigr\}.
\end{split}
\end{equation*}
Choosing $L=2(1+\kappa)$, we easily deduce by a scaling argument that there exists a constant $c \in (0,1)$ only depending on 
$\kappa$ (in particular, $c$ is independent of $n$, $\delta$, $t_{0}$ and $x_{0}$) such that 
\begin{equation*}
{\mathbb P} \bigl( Y_{\varrho^n}^n = 0 \bigr) \geq c,
\end{equation*}
from which we conclude that 
\begin{equation*}
p^n(t_{0},x_{0}) \leq ( 1 - c ) 
\sup_{t_{0}-\delta^2 \leq s \leq t_{0},\, 0 \leq y \leq L \delta} p^n(s,y). 
\end{equation*}
The above holds true under the sole assumption that $x_{0} \leq \delta$ and $t+\eta/2+\delta^2 < t_{0} \leq t+\epsilon - \eta/2$. 
Assuming $2\eta< t_{0}-t$ (notice that we can always make $\eta$ arbitrarily small) and iterating the above estimates, we deduce that, 
as long as $\delta^2(1+L^2+\cdots+L^{2k}) \leq \eta$, we have
\begin{equation*}
p^n(t_{0},x_{0}) \leq ( 1 - c )^{k+1} 
\sup_{t_{0}-(1+L^2+\cdots+L^{2k})\delta^2 \leq s \leq t_{0},\, 0 \leq y \leq \delta L^{k+1}} p^n(s,y).
\end{equation*} 
Hence, as long as $\delta^2 L^{2(k+1)} \leq \eta$, we have
\begin{equation*}
p^n(t_{0},x_{0}) \leq ( 1 - c )^k 
\sup_{(s,y) \in [t+\eta,t+\epsilon-\eta/2] \times [0,\sqrt{\eta}]} p^n(s,y). 
\end{equation*} 
That is, the above bound holds true if 
$k+1 \leq \ln(\delta^{-2} \eta)/\ln(L^2)$, which leads to 
(choosing $k+1 = \lfloor \ln(\delta^{-2} \eta)/\ln(L^2) \rfloor$) 
\begin{equation*}
\begin{split}
p^n(t_{0},x_{0}) &\leq (1-c)^{-2} ( 1 - c )^{\ln(\delta^{-2} \eta)/\ln(L^2)} 
\sup_{(s,y) \in [t+\eta,t+\epsilon-\eta/2] \times [0,\sqrt{\eta}]} p^n(s,y)
\\
&= (1-c)^{-2} \bigl( \delta/\sqrt{\eta} \bigr)^{\chi} 
\sup_{(s,y) \in [t+\eta,t+\epsilon-\eta/2] \times [0,\sqrt{\eta}]} p^n(s,y),
\end{split}
\end{equation*} 
with $\chi := -  \ln(1-c)/\ln(L)$. For $x_{0} \in (0,\sqrt{\eta})$, we can choose $\delta = x_{0}$, which yields
\begin{equation*}
\begin{split}
p^n(t_{0},x_{0}) \leq C\, x_{0}^{\chi} 
\sup_{(s,y) \in [t+\eta,t+\epsilon-\eta/2] \times [0,\sqrt{\eta}]} p^n(s,y),
\end{split}
\end{equation*}  
for a constant $C$ only depending on $\eta$. 
	
	
In order to complete the proof of \eqref{eq:bound:holder:decay}, it suffices to provide a bound for 
$$
\sup_{(s,y) \in [t+\eta,t+\epsilon-\eta/2] \times [0,\infty)} p^n(s,y),
$$ 
uniformly over $n \geq 1$. 
The latter follows easily from the following observation. For any Borel $A\subset(0,\infty)$ and $s\geq t+\eta/2$, 
\begin{equation*}
\begin{split}
{\mathbb P} ( X_{s \wedge \tau^n}^n \in A) \leq {\mathbb P} \bigl( X_{t+\eta/2} - (\Lambda^n_{s} - \Lambda^n_{t+\eta/2}) + B_{s} - B_{t+\eta/2} \in A  \bigr),
\end{split}
\end{equation*}
and the latter is clearly less or equal to $1/\sqrt{2 \pi (s-t-\eta/2)} \vert A \vert$, where $\vert A \vert$ stands for the Lebesgue measure of $A$.
We deduce that $p^n(s,\cdot) \leq 1/\sqrt{2 \pi (s-t-\eta/2)}$, which completes the proof of \eqref{eq:bound:holder:decay}. 

\medskip
	
\textit{Third Step}.
In order to complete the proof, it remains to take the limit as $n \rightarrow \infty$. 
Recall \eqref{eq:Holder:proof:Xn}, together with the identity
\begin{equation*}
X_{s} = X_{t+\eta/2} - \bigl( \Lambda_{s} - \Lambda_{t+\eta/2}) + \bigl(B_{s}-B_{t+\eta/2}\bigr), \quad s \in [t+\eta/2,t+\epsilon-\eta/2].
\end{equation*} 
Since $(\Lambda^n)_{n \geq 1}$ converges to $\Lambda$ uniformly on $[t+\eta/2,t+\epsilon-\eta/2]$, we deduce that, for any 
$\eta \in (0,\epsilon)$, the sequence of laws $(\PP \circ ((X^n_{s})_{s \in [t+\eta/2,t+\epsilon - \eta/2]})^{-1})_{n \geq 1}$ (seen as probability measures on $C([t+\eta/2,t+\epsilon-\eta/2];\RR)$) converges in the weak sense to 
$\PP \circ ((X_{s})_{s \in [t+\eta/2,t+\epsilon- \eta/2]})^{-1}$. Since the process $X$ goes (with probability 1) into the negative when touching $0$, we deduce that 
$(\PP \circ ((X^n_{s})_{s \in [t+\eta/2,t+\epsilon - \eta/2]},\tau^n)^{-1})_{n \geq 1}$ (seen as probability measures on $C([t+\eta/2,t+\epsilon-\eta/2];{\mathbb R}) \times {\mathbb R})$ converges weakly to 
$\PP \circ ((X_{s})_{s \in [t+\eta/2,t+\epsilon- \eta/2]},\tau)^{-1}$.
Hence, for any $s \in [t+\eta/2,t+\epsilon-\eta/2]$ and for any bounded and continuous real-valued function $\varphi$ on $\RR$, with support in $(0,\infty)$,
\begin{equation*}
\lim_{n \rightarrow \infty} {\mathbb E} \bigl[ \varphi \bigl(X_{s \wedge \tau^n}^n\bigr) \bigr]
= {\mathbb E} \bigl[ \varphi \bigl( X_{s \wedge \tau} \bigr) \bigr]
= 
 \int_{0}^{\infty} \varphi \bigl(x\bigr) p(s,x)\,\mathrm{d}x. 
\end{equation*}
From the above, we easily deduce that $p$ inherits the bound \eqref{eq:bound:holder:decay}, which completes the proof. 
\end{proof}

Combining Propositions \ref{prop:sec2.p0.upperbd}, \ref{prop:Holder.1}, and \ref{subse:regLambda.1}, we obtain the following corollary, which summarizes the results of this section.

\begin{corollary}
\label{cor:2.6}
Fix an arbitrary $t\geq0$ and assume that $\rho(t,\cdot)$ satisfies at least one of the following two conditions: 
(i) $\lim_{\eta \downarrow 0} \esssuptext{x \in (0,\eta)} \rho(t,x) <1/\alpha$, or (ii) 
$\rho(t,\cdot)$
has a version that is locally monotone in a right neighborhood of any point in $[0,\infty)$.
Then, there exists an $\epsilon>0$ such that $\Lambda$ is 1/2-H\"older continuous on $(t,t+\epsilon)$ and $p(s,\cdot)$ (has a {version} that) is vanishing and H\"older continuous at $0$, uniformly over $s$ in any compact sub-interval of $(t,t+\epsilon)$.
\end{corollary}


\section{Lipschitz and Higher Order Regularity}
\label{se:3}

In this section, we keep the same notation as in the previous one: $X=(X_{t})_{t \geq 0}$ is a physical solution of  \eqref{Stefan_prob}; $p(t,\cdot)$ is the density of the restriction of the distribution of $X_{t} \, \bone_{\{\tau\geq t\}}$ to $(0,\infty)$; and $\rho(t,\cdot)$ is the density of the restriction of the distribution of $X_{t-} \,  
\bone_{\{\tau\geq t\}}$ to $(0,\infty)$. 

Our objective is to provide further regularity properties of $p$: Lipschitz property at the boundary, regularity of the gradient up to the boundary and (a form of) smoothness/analyticity inside the domain. Throughout the section, we use repeatedly the following notation: 
For a given $T>0$ (which shall be understood as a finite time horizon) and a given $t \in [0,T)$, we let 
$$T_{reg}(t) := \inf \bigl\{ s > t : \lim_{\eta \downarrow 0} \esssupmath{y \in (0,\eta)}  \ \rho(s,y) >0 \bigr\} \wedge T.$$
As a consequence of Corollary 
\ref{cor:2.6}, we deduce:

\begin{corollary}
\label{cor:3.1}
Fix an arbitrary time horizon $T>0$, together with 
a time $t \in [ 0,T)$, and assume that $\rho(t,\cdot)$ has a version that is
locally monotone in a right neighborhood of any point in $[0,\infty)$. Then, 
$T_{reg}(t) >t$. Moreover, for any $\eta \in (0,(T_{reg}(t)-t)/2)$, $\Lambda$ is 
$1/2$-H\"older continuous on $[t+\eta,T_{reg}(t)-\eta]$
and there exist a constant $C_{t,T,\eta} \geq 0$ and an exponent $\chi_{t,T,\eta}>0$ such that 
\begin{equation*}
p(s,x) \leq C_{t,T,\eta} \min( x^{\chi_{t,T,\eta}},1), \quad s \in [t+\eta,T_{reg}(t)-\eta], \ x \geq 0.
\end{equation*}
\end{corollary}

\begin{proof}
Fix $\eta$ as in the statement and identify $\rho(t,\cdot)$ with its locally monotone version.
{It is an immediate consequence of  
Corollary \ref{cor:2.6}
that $T_{reg}(t) >t$. Moreover,}
by construction, we know that, for any $s \in [t+\eta/2,T_{reg}(t)-\eta/2]$,  
{$\lim \esssuptext{y \downarrow 0} \rho(s,y)=0$}. 
Therefore, by Corollary \ref{cor:2.6} {again},
for any $s \in [t+\eta/2,T_{reg}(t)-\eta/2]$, 
we can find three constants $C_{s},\varepsilon_{s}>0$ and $\chi_{s} \in (0,1)$ such that $\Lambda$ is 
$1/2$-H\"older continuous on $(s,s+\varepsilon_{s})$, the H\"older semi-norm being less than $C_{s}$, and $p$ satisfies
\begin{equation*}
p(r,x) \leq C_{s} \min(x^{\chi_{s}},1), \quad r \in (s,s+\varepsilon_{s}), \ x \geq 0. 
\end{equation*}
{By compactness,} we can find $N \geq 1$, $s_{1},\ldots,s_{N} \in [t+\eta/2,T_{reg}-\eta/2]$ so that $[t+\eta,T_{reg}-\eta] \subset \cup_{i=1}^N (s_{i},s_{i}+\varepsilon_{s_{i}})$. It remains to let 
\begin{equation*}
C = \max_{i = 1,\ldots,N} C_{s_{i},T,\eta}, \quad \chi = \inf_{i=1,\ldots,N} \chi_{s_{i},T,\eta}. 
\end{equation*}
We easily deduce that $\Lambda$ is $1/2$-H\"older continuous on $[t+\eta,T_{reg}(t)-\eta]$, 
the H\"older semi-norm being less than $C$, and that $p$ satisfies 
$p(r,x) \leq C \max (1,x^{\chi})$, $r \in [t+\eta,T_{reg}(t)-\eta]$, $x \geq 0$. 
\end{proof}

\subsection{Lipschitz regularity}
\label{subse:3.1}
We now prove
\begin{proposition}
	\label{prop:reg:barrier}
	Fix an arbitrary time horizon $T>0$ together with a time $t \in [0,T)$ and assume that (a version of) $\rho(t,\cdot)$ is 
locally monotone in a right neighborhood of any point in $[0,\infty)$. Then,  the function $(t,T_{reg}(t)) \ni s \mapsto \Lambda_{s}$ is continuously differentiable and, for any 
	$\eta \in (0,(T_{reg}(t)-t)/2)$, there exists a constant $C_{t,T,\eta} \geq 0$ such that, for almost every $s \in [t+\eta,T_{reg}(t) - \eta]$, 
	\begin{equation*}
	\dot{\Lambda}_{s} \leq C_{t,T,\eta}.
	\end{equation*}
	In particular, $\Lambda$ is $C_{t,T,\eta}$-Lipschitz continuous on $[t+\eta,T_{reg}(t)-\eta]$. 
	
	Moreover, $p\in C^{1,2}((t,T_{reg}(t)) \times (0,\infty))$. 
	For any $\eta \in (0,(T_{reg}(t)-t)/2)$,
	it is bounded and continuous on $[t+\eta,T_{reg}(t)-\eta] \times [0,\infty)$ and the space derivative $\partial_{x} p$ is also bounded and continuous on $[t+\eta,T_{reg}(t)-\eta] \times [0,\infty)$. In particular, for any $\eta \in (0,(T_{reg}(t)-t)/2)$, there exists a constant $C_{t,T,\eta}' \geq 0$ such that, for any $s \in [t+\eta,T_{reg}(t)-\eta]$ and $x \geq 0$, 
	\begin{equation*}
	p(s,x) \leq C_{t,T,\eta}' \min(x,1).
	\end{equation*}
\end{proposition}

Following the proof of Proposition
\ref{prop:reg:Holder}, we may assume throughout that ${\mathbb P}(\tau>t)>0$ for any $t >0$. 
The proof of Proposition \ref{prop:reg:barrier} relies upon the recent results of \cite{HLS}. First, we state the following lemma, which follows directly from \cite{HLS}. 

\begin{lemma}
	\label{lem:Lip:1}
	Fix $T$ and $t$ as in the statement of Proposition \ref{prop:reg:barrier}. Then,	
	for any $\eta \in (0,(T_{reg}(t)-t)/2)$, there exist three positive constants $K_{t,T,\eta}$, $\epsilon_{t,T,\eta} \in (0,\eta)$ and 
	$\chi_{t,T,\eta} \in (0,1)$ such that, 
	for any $s \in [t+\eta,T_{reg}(t)-\eta]$,
	the function $(s,s + \epsilon_{t,T,\eta}) \ni r \mapsto \Lambda_{r}$ is absolutely continuous and satisfies
	\begin{equation*}
	\esssupmath{r \in (s,s+\epsilon_{t,T,\eta})} (r-s)^{(1-\chi_{t,T,\eta})/2}\dot{\Lambda}_{r} \leq K_{t,T,\eta}. 
	\end{equation*} 
\end{lemma}

\begin{proof}
	By Corollary \ref{cor:3.1}, there exist a constant $C_{t,T,\eta}$ and an exponent $\chi_{t,T,\eta}$ such that
	\begin{equation}
	\label{eq:p:Holder:infinity}
	p(s,x) \leq C_{t,T,\eta} \min(x^{\chi_{t,T,\eta}},1), \quad x \geq 0,\ s \in [t+\eta,T_{reg}(t)-\eta].  
	\end{equation} 
	For a given $s \in [t+\eta,T_{reg}(t)-\eta]$, we now consider the process $(X_{r+s},\Lambda_{r+s}-\Lambda_{s})_{r \geq 0}$.
	It is a solution of the state equation \eqref{Stefan_prob}, 
	{with 
	$p(s,\cdot)$ as the initial (sub-)density on $(0,\infty)$.}	
	Equation \eqref{eq:p:Holder:infinity}
	implies that the latter (sub-)density 
	has a H\"older decay at the boundary and is bounded on $[0,\infty)$. 
%
	The fact that the constant on the right-hand side of \eqref{eq:p:Holder:infinity} remains independent of $s$ is the key point to invoke the results of  \cite{HLS}. 
	
	For the sake of completeness, we introduce the following space (whose definition is taken from \cite{HLS}):
	For two constants $A$ and $\epsilon$, we denote by ${\mathcal S}(A,\epsilon)$ the collection of elements $\ell$ of $H^1((0,\epsilon))$ (the space of absolutely continuous paths on $(0,\epsilon)$ whose derivative is square integrable) such that $\esssuptext{r \in [0,\epsilon]} r^{(1-\chi_{t,T,\eta})/2} \dot{\ell}(r) \leq A$. Then, \cite[Theorem 1.7]{HLS} says that there exist a constant $K$ and a time $\epsilon>0$, only depending on $t$, $T$ and $\eta$ (through the constants in \eqref{eq:p:Holder:infinity})\footnote{We draw the reader's attention to the fact that the statement in \cite{HLS} is not entirely clear on the dependence of $\epsilon$ upon the shape of $p$, but a careful inspection of the argument shows that our claim is indeed correct.} such that the state equation \eqref{Stefan_prob}, whose initial condition has $p(s,\cdot)$
as its sub-density on $(0,\infty)$, has a unique solution in the space 
	${\mathcal S}(K,\epsilon)$. A priori, uniqueness is within ${\mathcal S}(K,\epsilon)$, but 
	\cite[Theorem 1.8]{HLS} shows that our solution $(\Lambda_{r+s}-\Lambda_{s})_{r \geq 0}$ must coincide with 
	the one in ${\mathcal S}(K,\epsilon)$ on $[0,\epsilon]$. This completes the proof. 
\end{proof} 

\begin{proof}[Proof of Proposition \ref{prop:reg:barrier}] \textit{First Step.} Replacing $\eta$ by $\eta/2$ and choosing $r\in(s+\epsilon_{t,T,\eta/2}/2, s+\epsilon_{t,T,\eta/2})$ in the supremum appearing in the statement of Lemma \ref{lem:Lip:1}, we deduce that $\Lambda$ is Lipschitz continuous on $[t+\eta,T_{reg}(t)-\eta]$, proving one of the statements in Proposition \ref{prop:reg:barrier}.
\vskip 4pt
	
\textit{Second Step.}
	Next, we deduce that $p$ has linear decay in $x$ at the boundary $x=0$ uniformly in $s \in [t+\eta,T_{reg}(t)-\eta]$, 
	for $\eta \in (0,(T_{reg}(t)-t)/2)$. To this end, we proceed as in the proof of Proposition 
	\ref{prop:reg:Holder} and consider mollified versions $((\Lambda^n_{s})_{s \in [t,T_{reg}(t))})_{n \geq 1}$
	of $\Lambda$. 
	By the first step, we can assume that, for a given $\eta \in (0,(T_{reg}(t)-t)/2)$, the functions 
	$((\Lambda^n_{s})_{s \in [t+\eta,T_{reg}(t)-\eta]})_{n \geq 1}$ are Lipschitz continuous, uniformly in $n$. 
	With a slight abuse of notation, we still denote the Lipschitz constant by $C_{t,T,\eta}$. 
	
	Consider the collection of stochastic processes:
	\begin{equation*}
	X^n_{s} = X_{t+\eta} - \bigl( \Lambda_{s}^n - \Lambda^n_{t+\eta} \bigr) + B_{s} - B_{t+\eta}, \quad 
	s \geq t+\eta,
	\end{equation*} 
	for $n \geq 1$.
	For each $n \geq 1$ and any $s \geq t+\eta$, 
	we denote by $p^n(s,\cdot)$ the density of the {restriction of the distribution  of $X^n_{s} \, \bone_{\{s \geq \tau^n\}}$ to $(0,\infty)$}, where 
	$\tau^n := \inf\{ r \geq s : X^n_{r} \leq 0\}$. Following \cite[Lemma 4.2]{DIRT1}, we know that 
	$p^n$ is continuous on $(t+\eta,\infty) \times [0,\infty)$ and that it is a classical solution of the PDE
	\begin{equation}
	\label{eq:PDE:pn}
	\partial_{t} p^n- \frac12 \partial_{xx} p^n - \dot{\Lambda}^n \partial_{x} p^n = 0, \quad s>t+\eta, \ x >0.
	\end{equation}

	By \cite[Lemmas 2.1 and 3.1]{DIRT3} (up to an obvious modification, as the absorption herein occurs at the boundary of $(0,\infty)$ and not at the boundary of $(-\infty,1)$), we know that, for each $n \geq 1$ and each $s \in (t+\eta,T_{reg}(t)-\eta]$, the function $p^n(s,\cdot)$ is differentiable at any point $x \geq 0$. Moreover,
	\begin{equation}
	\label{eq:formula:partialxpn}
	\begin{split}
	\partial_{x} p^n(s,x) &= \int_{0}^{\infty} \partial_{y} q(s-(t+\eta),z,x)\, p^n(t+\eta,z)\,\mathrm{d}z 
	\\
	&\hspace{12pt}+ \int_{t+\eta}^{s}   \int_{0}^{\infty} \dot{\Lambda}^n_{r}\, 
	\partial_{x} p^n(r,z)\, \partial_{y} q(s-r,z,x)\,\mathrm{d}r\,\mathrm{d}z,
	\end{split}
	\end{equation} 
	where 
	\begin{equation*}
	q(r,x,y) = g(r,x-y) - g(r,x+y), \quad r >0, \ x,y >0,
	\end{equation*}
	is the kernel of the heat equation with absorption at $x=0$, and the function $g(r,\cdot)$ denotes the usual 
	Gaussian kernel of variance $r$ (and of zero mean). We make the following key observations. First, we know that each 
	$p^n$ solves 
	\eqref{eq:PDE:pn}. As the $(p^n)_{n \geq 1}$ are bounded uniformly in $n$, we deduce from standard results on the smoothing effect of the heat equation that, on any closed ball included in $(t+\eta,T_{reg}(t)-\eta) \times (0,\infty)$, the functions $((s,x) \mapsto p^n(s,x))_{n \geq 1}$
	are in $C^{(1+\alpha)/2,1+\alpha}((t+\eta,T_{reg}(t)-\eta) \times (0,\infty))$, for some $\alpha \in (0,1)$, uniformly in $n \geq 1$ (namely $p^n$ and $\partial_{x} p^n$ are locally H\"older continuous in time and space, uniformly in $n \geq 1$). As $p^n$ converges to 
	$p$ on $(t+\eta,T_{reg}(t)-\eta) \times (0,\infty)$ ({see Proposition \ref{prop:reg:Holder}}), this shows that $p$ is differentiable in $x$ on $(t+\eta,T_{reg}(t)-\eta) \times (0,\infty)$
	and that $(\partial_{x} p^n(s,x))_{n \geq 1}$ converges to $\partial_{x} p(s,x)$ for any $(s,x) \in (t+\eta,T_{reg}(t)-\eta) \times (0,\infty)$. 
	
	Another observation is that \cite[Propositions 3.2 and 4.2]{DIRT3} imply the existence of a constant $C$ (possibly depending on $t$, $T$, $\eta$, but independent of $n$) such that, for any $n \geq 1$ and $(s,x) \in (t+\eta,T_{reg}(t)-\eta) \times (0,\infty)$, 
	\begin{equation}
	\label{eq:bound:partialxpn}
	\vert \partial_{x} p^n(s,x) \vert \leq \frac{C}{\sqrt{s-(t+\eta)}}.
	\end{equation}
	Also, we have, for all $r \in (t+\eta,s)$ {and $z>0$},
	\begin{equation}
	\label{eq:bound:partialxq}
	\vert \partial_{y} q(s-r,z,x) \vert \leq \frac{C}{s-r} \exp \Bigl( - \frac{\vert x-z \vert^2}{C(s-r)} \Bigr). 
	\end{equation}
	In particular, we have the following bound for the second integrand in \eqref{eq:formula:partialxpn} (allowing for a new value of the constant $C$):
	\begin{equation*}
	\bigl\vert \dot{\Lambda}^n_{r} 
	\partial_{x} p^n(r,z) \partial_{y} q(s-r,z,x)
	\bigr\vert \leq \frac{C}{\sqrt{r-(t+\eta)}\,(s-r)}\exp \Bigl( - \frac{\vert x-z \vert^2}{C(s-r)} \Bigr),
	\end{equation*}
	which is integrable in $(r,z) \in (t+\eta,s) \times (0,\infty)$. Since, after passing to a subsequence, $\dot{\Lambda}^n$ converges almost
	everywhere (in time) to $\dot{\Lambda}$, we can 
	take the limit in \eqref{eq:formula:partialxpn}
	as $n \rightarrow \infty$ and deduce that 
	\begin{equation}
	\label{eq:formula:partialxpn:2}
	\begin{split}
	\partial_{x} p(s,x) &= \int_{0}^{\infty} \partial_{y} q(s-(t+\eta),z,x)\,p(t+\eta,z)\,\mathrm{d}z 
	\\
	&\hspace{12pt}+ \int_{t+\eta}^{s}   \int_{0}^{\infty} \dot{\Lambda}_{r}\, 
	\partial_{x} p(r,z)\, \partial_{y} q(s-r,z,x)\,\mathrm{d}r\,\mathrm{d}z.
	\end{split}
	\end{equation}   
	Moreover, taking the limit in \eqref{eq:bound:partialxpn}, we also have 
	\begin{equation}
	\label{eq:bound:partialxp}
	\vert \partial_{x} p(s,x) \vert \leq \frac{C}{\sqrt{s-(t+\eta)}}, \quad s \in \bigl(t+\eta,T_{reg}(t)-\eta\bigr), \  x >0. 
	\end{equation}
	Therefore, we can take the limit as $x \downarrow 0$ in \eqref{eq:formula:partialxpn:2}. We deduce that, for any 
	$s \in (t+\eta,T_{reg}(t)-\eta)$, $\partial_{x} p(s,x)$ has a limit as $x \downarrow 0$. In particular, 
	$p(s,\cdot)$ is differentiable at the point 0 and $\partial_{x} p(s,0)$ is given by \eqref{eq:formula:partialxpn:2}. Since $\eta$ is arbitrary, this is true for any 
	$s \in (t,T_{reg}(t))$. Moreover, we deduce from 
	\eqref{eq:bound:partialxp} (with $\eta$ replaced by $\eta/2$), that 
	for any $\eta \in (0,(T_{reg}(t)-t)/2)$, we can find 
	a constant $C_{t,T,\eta}'$ such that
	\begin{equation}
	\label{eq:Lip:2}
	p(s,x) \leq C_{t,T,\eta}' \min(x,1), \quad s \in [t+\eta,t+T_{reg}(t)-\eta], \ x \geq 0.
	\end{equation} 
	\vskip 4pt
	
	\textit{Fourth Step.}
	By combining 
	the conclusions of the second and third steps, we deduce that, for any $\eta \in (0,(T_{reg}(t)-t)/2)$, the function
	$p(t+\eta,\cdot)$ is differentiable on $[0,\infty)$ and satisfies 
	\eqref{eq:Lip:2}. 
	In other words, it satisfies all the assumptions of the existence and uniqueness result stated in 
	\cite[Theorem 4.1]{DIRT1} (which is stated in a slightly different framework, but which obviously applies in our setting). 
	This latter result says that there exists a unique solution $(\tilde X,\tilde \Lambda)$ to the state equation \eqref{Stefan_prob} whose initial condition has 
	$p(t+\eta,\cdot)$ as its sub-density on $(0,\infty)$,
	and it is such that $\tilde \Lambda$ is continuously differentiable on $[t+\eta,t+\eta+\epsilon]$, for some $\epsilon \in (0,T_{reg}(t)-t-2\eta)$ only depending on the parameter $C_{t,T,\eta}'$ in \eqref{eq:Lip:2}. By \cite[Theorem 1.8]{HLS}, the process $\tilde \Lambda$ must coincide with $\Lambda-\Lambda_{t+\eta}$
	on $[t+\eta,t+\eta+\epsilon]$, which shows that $\Lambda$ is continuously differentiable on $(t,T_{reg}(t))$.  
	
	Moreover,  \eqref{eq:bound:partialxp} shows that $\partial_{x} p$ is bounded on $[t+\eta,T_{reg}(t)-\eta] \times [0,\infty)$, for any 
	$\eta \in (0,(T_{reg}(t)-t)/2)$.  
	
	Lastly, using the fact that, for $s>r$ (and with a new $C<\infty$),
	\begin{equation*}
	\bigl\vert \partial_{t} \partial_{y} q(s-r,z,x) \bigr\vert \leq \frac{C}{(s-r)^{2}} \exp\Bigl( - \frac{\vert x-z \vert^2}{C(s-r)} \Bigr),
	\end{equation*}
	we see from \eqref{eq:bound:partialxq} that, for $s>s'>r$,
	\begin{equation*}
	\begin{split}
	&\bigl\vert \partial_{y} q(s-r,z,x) - \partial_{y} q(s'-r,z,x)  \bigr\vert 
	\\
	&\leq \biggl\vert \int_{s}^{s'} \partial_{t} \partial_{y} q(s''-r,z,x)\, \mathrm{d} s'' \biggr\vert^{1/8}
	\Bigl(  \bigl\vert \partial_{y} q(s-r,z,x) \bigr\vert^{7/8} + \bigl\vert \partial_{y} q(s'-r,z,x)  \bigr\vert^{7/8}
	\Bigr)
	\\
	&\leq \frac{C \vert s-s'\vert^{1/8}}{(s'-r)^{9/8}} \exp\Bigl( - \frac{\vert x-z \vert^2}{C(s'-r)} \Bigr).
	\end{split}
	\end{equation*}
	Plugging the above bound in \eqref{eq:formula:partialxpn}, we deduce that, for any $\eta \in (0,(T_{reg}(t)-t)/2)$, 
	$\partial_{x} p$ is continuous in $s$, uniformly over $x$, for $(s,x)\in[t+\eta,T_{reg}(t)-\eta] \times [0,\infty)$. We conclude that 
	$\partial_{x} p$ is continuous on $(t,T_{reg}(t)) \times [0,\infty)$.  
\end{proof}

\subsection{{Smoothness and analyticity} inside the domain}
\label{subse:3.2}
Whilst Proposition 
	\ref{prop:reg:barrier} mostly concerns regularity at the boundary, the next result addresses the regularity of $p$ \textit{inside the domain}.

\begin{proposition}\label{prop:analytic.1}
For a physical solution $(X,\Lambda)$ of \eqref{Stefan_prob},
following 
\eqref{u_to_p}, we denote 
\begin{equation*}
u(t,x):=p(t,x-\Lambda_t), \quad t >0, \ x > \Lambda_{t}. 
\end{equation*}
Then, $u$ is  {$C^{\infty}$ in $(t,x)$} and satisfies $\partial_t u = \frac{1}{2}\partial_{xx} u$ pointwise on $\mathring{D}:=\{(t,x)\in(0,\infty)^2: x\!>\!\Lambda_t\}$.
Moreover, for any $(t,x) \in \mathring{D}$, there is a neighborhood of $(t,x)$ in ${\mathbb R} \times {\mathbb C}$ on 
which there exist extensions of
$u$ and its time derivatives of all orders, such that the extensions are analytic in the space variable 
with jointly continuous space derivatives of all orders. In particular, 
for any $t>0$, the function $u(t,\cdot)$ is real analytic in $x$ on $(\Lambda_{t},\infty)$, and the functions $p(t,\cdot)$, $\rho(t,\cdot)$ are real analytic in $x$ on $(0,\infty)$.
\end{proposition}

At first sight, the fact that $u$ solves the heat equation could appear as a simple reformulation of 
the fact that $p$ satisfies the equation \eqref{Dir_problem}, but it is not! As we have already explained, $p$ may have time discontinuities (and, hence, may not be smooth in time) at those times $t$ when $ \lim_{\eta \downarrow 0} \esssupmath{y \in (0,\eta)}  \ \rho(t,y) \geq 1/\alpha$. This is in contrast with the above statement: Therein, we assert that $u$ is {smooth} in $\mathring{D}$ whatever the behavior of $\rho$ at the boundary.
Regarding the smoothness in the sole space variable $x$, the real analyticity of $u$ in $x$  was already pointed out in the very recent preprint 
\cite{LS2}. In fact, the real analyticity in $x$ of the solutions to the heat equation is a  general result in PDE theory, see for instance \cite[exercise 8.4.7]{KrylovHolder}. Here, our result says more since we not only prove the real analyticity in $x$ of $u$ and of its time derivatives, but also extend $u$ to a complex domain in $x$, locally uniformly in $t$. We use the latter fact in 
Section \ref{se:4}. Also, our arguments are different from those used in \cite{LS2}, as we have developed our analysis independently. 

Lastly, it is worth noticing that the analyticity of $\rho(t,\cdot)$ is a not a direct consequence of the analyticity of $u(t,\cdot)$; indeed, the definition of $u$ yields $\rho(t,x) = u(t,x+\Lambda_{t{-}})$, but the latter only makes sense when 
$x+\Lambda_{t-} > \Lambda_{t}$, that is, $x > \Lambda_{t} - \Lambda_{t-}$. Hence, an additional (small) argument is needed to prove that, at any given discontinuity time $t$ of $\Lambda$, $u(t,\cdot)$ may be analytically extended 
to the entire $(\Lambda_{t-},\infty)$. 

\begin{proof}
Let $0<t_1<t_2<\infty$ and $0<x_1<x_2<\infty$ fulfill $[t_1,t_2]\times[x_1,x_2]\subset\mathring{D}$. Then, $u$ is a generalized solution of the PDE $\partial_t u = \frac{1}{2}\partial_{xx} u$ on $[t_1,t_2]\times[x_1,x_2]$. Indeed, for any $\varphi\in C_c^\infty((t_1,t_2)\times(x_1,x_2))$,
\begin{equation*}
\begin{split}
\int_{t_1}^{t_2} \int_{x_1}^{x_2} u\,\Big(\partial_t\varphi+\frac{1}{2}\partial_{xx}\varphi\Big)\,\mathrm{d}x\,\mathrm{d}t
&=\int_{t_1}^{t_2} \int_{x_1-\Lambda_t}^{x_2-\Lambda_t} p(t,y)\,\Big(\partial_t\varphi+\frac{1}{2}\partial_{xx}\varphi\Big)(t,y+\Lambda_t)\,\mathrm{d}y\,\mathrm{d}t 
\\
&=\int_{t_1}^{t_2} \E\Big[\Big(\partial_t\varphi+\frac{1}{2}\partial_{xx}\varphi\Big)(t,X_t+\Lambda_t)\,\mathbf{1}_{\{\tau \geq t\}}\Big]\,\mathrm{d}t 
\\
&=\E\bigg[\int_{t_1\wedge\tau}^{t_2\wedge\tau} \Big(\partial_t\varphi+\frac{1}{2}\partial_{xx}\varphi\Big)(t,X_0+B_t)\,\mathrm{d}t \bigg] 
\\
&=\E\big[\varphi(t_2\wedge\tau,X_0+B_{t_2\wedge\tau})
-\varphi(t_1\wedge\tau,X_0+B_{t_1\wedge\tau})\big]=0,
\end{split}
\end{equation*}
where in the second-last equality we used It\^o's formula and the optional sampling theorem (see e.g.~\cite[chapter II, corollary 3.6]{RY}), and in the last equality we used the fact that the support of $\varphi$ is included in $(t_{1},t_{2}) \times (x_{1},x_{2})$. We conclude that $u\in C^\infty([t_1,t_2]\times[x_1,x_2])$ and $\partial_t u = \frac{1}{2}\partial_{xx} u$ pointwise on $(t_1,t_2)\times(x_1,x_2)$ by virtue of Weyl's lemma in the form of \cite[p.~90, step 4]{McK}. The analyticity of $u$, and of its time derivatives, in space is now a consequence of Lemma \ref{lem:3.5} below, applied to the function
\begin{equation}
w:\;[0,\infty)\times[x_1,x_2]\to\rr,\quad (t,x)\mapsto \widetilde{w}(t+\epsilon,x), 
\end{equation}
where $\widetilde{w}$ is the classical solution of the Cauchy-Dirichlet problem
\begin{equation}
\label{eq:35}
\begin{split}
& \partial_t\widetilde{w}=\frac{1}{2}\partial_{xx}\widetilde{w}\quad\text{on}\quad [\epsilon,\infty)\times[x_1,x_2], 
\quad \textrm{with} \
  \left\{ \begin{array}{l}
\widetilde{w}(\epsilon,x)=u(\epsilon,x),\quad x_1\le x\le x_2, \\
\widetilde{w}(t,x_1)=u(t,x_1)\,\varphi(t),\quad t\ge\epsilon, \\
\widetilde{w}(t,x_2)=u(t,x_2)\,\varphi(t),\quad t\ge\epsilon, 
\end{array}
\right.
\end{split}
\end{equation}
and with $\varphi\in C^\infty([\epsilon,\infty))$ satisfying $\varphi\equiv 1$ on $[\epsilon,t_2]$ and $\varphi\equiv0$ on $[t_2+\epsilon,\infty)$, for some $\epsilon\in(0,t_1)$ such that $[0,t_2+\epsilon]\times[x_1,x_2]\subset \mathring{D}$. Clearly, $\tilde w$ coincides with $u$ on $[\epsilon,t_{2}]$. The analyticity of $p$ in $x \in (0,\infty)$ at positive times follows easily.

It now remains to address the analyticity of $\rho(t,\cdot)$ in $x \in (0,\infty)$, for $t>0$.
Clearly, it suffices to treat the case $\Lambda_{t} > \Lambda_{t-}$ and $x \in (0,\Lambda_{t}- \Lambda_{t-}]$. 
With the same notation as before, we consider a rectangle $[t_{1},t_{2}] \times [x_{1},x_{2}]$, with $t_{1}>0$, $t_{2}=t$ and $x_{1} > \Lambda_{t-} = \Lambda_{t_{2}-}$. For $s \in [t_{1},t_{2})$ and $x \in (x_{1},x_{2})$, we have
\begin{equation*}
u(s,x) = {\mathbb E} \bigl[ u\bigl( s- \sigma \wedge (s-t_{1}), x + B_{\sigma \wedge (s-t_{1})}\bigr) \bigr],
\end{equation*}
where $\sigma:= \inf\{ r \geq 0 : \, x + B_{r} \not \in (x_{1},x_{2})\}$. Since 
${\mathbb P}(\sigma >0)=1$ and since $u$ is continuous and bounded on $[t_{1},t_{2}) \times [x_{1},x_{2}]$, 
the right-hand side of the above has a limit as $s \uparrow t_{2}=t$. 
In fact, by interior estimates for the heat equation, we have 
uniform bounds on the derivatives (of any order) of $u$ on any 
$[t_{1}',t_{2}) \times [x_{1}',x_{2}']$, with $t_{1} < t_{1}' < t_{2}$ and $x_{1} < x_{1}' < x_{2}' < x_{2}$. 
As we may play with the choice of $t_{1}$, $x_{1}$ and $x_{2}$, we have 
uniform bounds on the derivatives (of any order) of $u$ on $[t_{1},t_{2}) \times [x_{1},x_{2}]$. This says that $u$
has a $C^{\infty}([t_{1},t_{2}] \times [x_{1},x_{2}])$-extension. Denoting this extension by $\tilde u$, we can repeat 
\eqref{eq:35} by extending $\tilde u(s,x_{1})$ and $\tilde u(s,x_{2})$ in a constant way for $s \geq t_{2}$. 
We deduce that $\tilde{u}(t,\cdot)$ is analytic on $(\Lambda_{t-},\infty)$ and, thus, $\rho(t,\cdot)$ is analytic on $(0,\infty)$. 
\end{proof}

\begin{lemma}
\label{lem:3.5}
Let $0<x_1<x_2<\infty$ and $w\in C^\infty([0,\infty)\times[x_1,x_2])$ be a classical solution of 
\begin{equation}
\partial_t w=\frac{1}{2}\partial_{xx}w\quad\text{on}\quad [0,\infty)\times[x_1,x_2]
\end{equation}
with $\varphi_1:=w(\cdot,x_1)\in C^\infty_c([0,\infty))$ and $\varphi_2:=w(\cdot,x_2)\in C^\infty_c([0,\infty))$. Then, for any $(t,x) \in (0,\infty) \times (x_{1},x_{2})$, there is a neighborhood of $(t,x)$ in ${\mathbb R} \times {\mathbb C}$ to which
$w$ and its time derivatives of any order can be extended, the extensions being analytic in the space variable with jointly continuous space derivatives. In particular, $w$ is real analytic in $x$ on $(x_1,x_2)$.
\end{lemma}

\begin{proof}
\textit{First Step.}
We pick any $x_1<\overset{\circ}{x}_1<\overset{\circ}{x}_2<x_2$ and any $[0,1]$-valued function $\varphi_0\in C^\infty([x_1,x_2])$ compactly supported in $(x_1,x_2)$ with $\varphi_0\equiv 1$ on $[\overset{\circ}{x}_1,\overset{\circ}{x}_2]$. Then, by \cite[chapter IV, Theorem 5.2]{Lady} there exists a unique classical solution $w^0\in C^\infty([0,\infty)\times[x_1,x_2])$ of the Cauchy-Dirichlet problem 	
\begin{equation}
\begin{split}
& \partial_t w^0=\frac{1}{2}\partial_{xx}w^0\quad\text{on}\quad [0,\infty)\times[x_1,x_2], 
\\
&
\textrm{with}
\
\left\{\begin{array}{l} 
w^0(0,x)=w(0,x)\,\varphi_0(x),\quad x_1\le x\le x_2, 
\\
w^0(t,x_1)=w^0(t,x_2)=0,\quad t\ge 0.
\end{array}
\right.
\end{split}
\end{equation}
Moreover, $w^0$ is real analytic in $(t,x)$ on $(0,\infty)\times[x_1,x_2]$ by \cite[Theorem 1]{Kom}. In particular, for any $(t,x) \in (0,\infty) \times (x_{1},x_{2})$, the function $w^0$ has a complex analytic extension to 
a complex neighborhood of $(t,x)$. Hence, the conclusions of the lemma hold for $w^0$. Thus, it suffices to show that the conclusions of the lemma also hold for $\Delta:=w-w^0$. 

\medskip
	
\textit{Second Step.} The function $\Delta\in C^\infty([0,\infty)\times[x_1,x_2])$ is a classical solution of the Cauchy-Dirichlet problem
\begin{equation}
\label{eq:Delta}
\begin{split}
& \partial_t \Delta=\frac{1}{2}\partial_{xx}\Delta\quad\text{on}\quad [0,\infty)\times[x_1,x_2], 
\\
&
\textrm{with} \
\left\{
\begin{array}{l} 
\Delta(0,x)=w(0,x)(1-\varphi_0(x)),\quad x_1\le x\le x_2, 
\\
\Delta(t,x_1)=\varphi_1(t),\quad t\ge0, 
\\
\Delta(t,x_2)=\varphi_2(t),\quad t\ge0.
\end{array}
\right.
\end{split}
\end{equation}
In particular, $\Delta,\partial_t\Delta,\partial_{tt}\Delta\in C^\infty([0,\infty)\times[x_1,x_2])$ are all classical solutions of the forward heat equation on $[0,\infty)\times[x_1,x_2]$ and, by the maximum principle, all three are globally bounded in absolute value. The latter is also true for $\partial_{xx}\Delta=2\partial_t\Delta\in C^\infty([0,\infty)\times[x_1,x_2])$, so that, with $\cc_+:=\{\lambda\in\cc:\,\mathrm{Re}\,\lambda>0\}$, the Laplace transform in time 
\begin{equation}
\label{eq:widehat:Delta}
\widehat{\Delta}(\lambda,x):=\int_0^\infty e^{-\lambda t}\,\Delta(t,x)\,\mathrm{d}t,\quad (\lambda,x)\in\cc_+\times[x_1,x_2], 
\end{equation}
solves
({using the fact that $\Delta(0,x)=0$ for $x \in [\overset{\circ}{x}_1,\overset{\circ}{x}_2]$})
\begin{equation}\label{ODE_Laplace}
\begin{split}
(\partial_{xx}\widehat{\Delta})(\lambda,\cdot)=2\lambda\widehat{\Delta}(\lambda,\cdot)\quad\text{on}\quad[\overset{\circ}{x}_1,\overset{\circ}{x}_2],
\end{split}
\end{equation} 
for all $\lambda\in\cc_+$. The solution of the linear \textit{ordinary} differential equation \eqref{ODE_Laplace} reads
\begin{equation}
\label{eq:widehat:Delta:cc}
\begin{split}
&\widehat{\Delta}(\lambda,x)=C_1(\lambda)\,e^{\sqrt{2\lambda}x}+C_2(\lambda)\,e^{-\sqrt{2\lambda}x}, \\
&C_1(\lambda)=\frac{1}{2\sinh(\sqrt{2\lambda}(\overset{\circ}{x}_1-\overset{\circ}{x}_2))}
\Big(\widehat{\Delta}(\lambda,\overset{\circ}{x}_1)\,e^{-\sqrt{2\lambda}\overset{\circ}{x}_2}-\widehat{\Delta}(\lambda,\overset{\circ}{x}_2)\,e^{-\sqrt{2\lambda}\overset{\circ}{x}_1}\Big), \\
&C_2(\lambda)=\frac{1}{2\sinh(\sqrt{2\lambda}(\overset{\circ}{x}_1-\overset{\circ}{x}_2))}
\Big(
-
\widehat{\Delta}(\lambda,\overset{\circ}{x}_1)\,e^{\sqrt{2\lambda}\overset{\circ}{x}_2}+\widehat{\Delta}(\lambda,\overset{\circ}{x}_2)\,e^{{\sqrt{2\lambda}}\overset{\circ}{x}_1}\Big),
\end{split}
\end{equation}
where $\sqrt{2\lambda}$ is the principal square root of $2\lambda$. 
It is easy to see that, for any $\lambda\in\cc_+$, $\widehat{\Delta}(\lambda,\cdot)$  {extends to a holomorphic function on} $(\overset{\circ}{x}_1,\overset{\circ}{x}_2)+i\rr$, where $i^2=-1$. 

\medskip

\textit{Third Step.} Next, we apply the complex Laplace inversion formula, e.g., in the form of 
 {\cite[chapter 19]{Rudin}}, to find
\begin{equation}\label{Lapl_inv}
\Delta(t,x)=\frac{1}{2\pi i}\int_{r-i\infty}^{r+i\infty} e^{t\lambda}\,\widehat{\Delta}(\lambda,x)\,\mathrm{d}\lambda,\quad(t,x)\in(0,\infty)\times[\overset{\circ}{x}_1,\overset{\circ}{x}_2],
\end{equation}
for all $r>0$, with the integral on the right-hand side of \eqref{Lapl_inv} being absolutely convergent. {The proof of \eqref{Lapl_inv} is in fact quite straightforward. Whenever $\lambda= r + i v$, \begin{equation*}
\widehat{\Delta}(r + i v,x)=\int_{-\infty}^\infty 
e^{-i v t}\,
 {\mathbf 1}_{(0,\infty)}(t) e^{-r t} \Delta(t,x)\,\mathrm{d}t.
\end{equation*}
Up to the scaling factor $\sqrt{2\pi}$, the function 
$v \mapsto \widehat{\Delta}(r+iv,x)$ is the Fourier transform of the integrable function 
$t \mapsto {\mathbf 1}_{(0,\infty)}(t) e^{-rt } \Delta(t,x)$. Formula 
\eqref{Lapl_inv} then follows if we can prove that $v \mapsto \widehat{\Delta}(r+iv,x)$ is integrable. 
This is where the regularity properties of $\Delta$ come in. 
Indeed, 
using the fact that $\Delta(0,x)=0$
and that $\partial_{t} \Delta(0,x)= \tfrac12 \partial_{xx} \Delta(0,x) =0$
 for the value of $x$ in consideration, we have, for $\lambda \in \cc_{+}$,
\begin{equation}
\label{eq:bound:hat Delta}
\vert \widehat{\Delta}(\lambda,x) \vert =  \biggl\vert \frac1{\lambda^2} \int_0^\infty e^{-\lambda t}\,\partial_{tt} \Delta(t,x)\,\mathrm{d}t \biggr\vert \leq 
\frac1{\vert \lambda \vert^2} \frac{1}{\mathrm{Re}\,\lambda} \sup_{x \in [\overset{\circ}{x}_1,\overset{\circ}{x}_2]} \sup_{t \geq 0} \vert \partial_{tt} \Delta(t,x) \vert,
\end{equation}
which implies that the right-hand side of \eqref{Lapl_inv} is absolutely convergent.}

{Our next goal is to use \eqref{Lapl_inv}
to extend $\Delta(t,\cdot)$ analytically to a complex neighborhood of
$(\overset{\circ}{x}_1,\overset{\circ}{x}_2)$ for a given $t >0$, that is, to extend $\Delta$ to pairs
$(t,x) \in (0,\infty) \times \cc$ such that $\text{Re} \ x \in 
(\overset{\circ}{x}_1,\overset{\circ}{x}_2)$
and 
$|\text{Im} \, x|$ is small (but possibly non-zero). 
This is in fact not so straightforward because, at this stage of the proof, nothing guarantees \textit{a priori}
that 
the right-hand side of 
\eqref{Lapl_inv} is absolutely convergent whenever $x$ has a non-trivial imaginary part.}

{In the fourth step below, we prove that, for any 
$\varepsilon 
\in (0,(\overset{\circ}{x}_2-\overset{\circ}{x}_1)/2)$, 
$\lambda \in \cc$ such that 
$\text{Re}(\lambda) \geq 1$,
and $x \in \cc$ such that 
$\text{Re}(x) \in (\overset{\circ}{x}_1+\varepsilon,\overset{\circ}{x}_2-\varepsilon)$
and $\vert \text{Im} \ x\vert \leq  {\varepsilon/2}$, 
\begin{equation}
\label{eq:exp:bound:5}
\vert
\widehat{\Delta}(\lambda,x) \vert \leq
\frac{C}{\vert \lambda \vert^2} 
\exp \Bigl( - \frac{\varepsilon}4  \sqrt{\vert \lambda \vert} \Bigr).
\end{equation}
Observing from \eqref{eq:widehat:Delta:cc} that the integrand in the right-hand side of
\eqref{Lapl_inv} is  holomorphic in $x \in \cc_{+}$ and 
using the Cauchy representation formula for holomorphic functions 
together with the Lebesgue differentiation theorem under the integral sign,  
we deduce that $\Delta(t,\cdot)$ and its time derivatives $\partial_{t}^{k} \Delta(t,\cdot)$, $k \geq 1$, are holomorphic on the domain $\{x \in \cc : 
\ 
\text{Re}(x) \in 
(\overset{\circ}{x}_1+\varepsilon,\overset{\circ}{x}_2-\varepsilon),
\ \vert \text{Im} \ x\vert \leq \varepsilon/2\}$. Moreover, for any $k \in {\mathbb N}$, the function $\partial_{t}^k \Delta$ is continuous on $\{(t,x) \in (0,\infty)\times {\mathbb C}\!:\,\text{Re}(x)\in 
(\overset{\circ}{x}_1+\varepsilon,\overset{\circ}{x}_2-\varepsilon),\,
\vert \text{Im} \ x\vert < {\varepsilon/2}\}$ and, by Cauchy's formula, the same is true for all the derivatives 
$\partial_{t}^k \partial_{x}^\ell \Delta$, $k,\ell \in {\mathbb N}$.}  Choosing $\varepsilon>0$ as small as needed, this proves in particular the desired real analyticity of $\Delta(t,\cdot)$ in $x$ on $(\overset{\circ}{x}_1,\overset{\circ}{x}_2)$.
\medskip

{\textit{Fourth Step.}
We now prove the desired strengthening \eqref{eq:exp:bound:5} of 
\eqref{eq:bound:hat Delta}. 

Our first observation is that, in 
\eqref{eq:Delta}, $\partial_{x} \Delta$ is bounded on the whole $[0,\infty) \times [\overset{\circ}{x}_1,\overset{\circ}{x}_2]$. This follows from the interior gradient estimates for the heat equation (see e.g. \cite[chapter IV, Theorem 10.1]{Lady}) and from the fact that $\Delta$ itself is globally bounded on the whole 
$[0,\infty) \times [x_{1},x_{2}]$. By induction, the same holds for higher order derivatives. We deduce that $\partial_{t} \partial_{x} \Delta$ and $\partial_{tt} \partial _{x} \Delta$ are also bounded on $[0,\infty) \times [\overset{\circ}{x}_1,\overset{\circ}{x}_2]$. 
As a consequence, we can differentiate under the integral in the right-hand side of 
\eqref{eq:widehat:Delta} and obtain a similar representation formula for 
$\partial_{x} 
\widehat{\Delta}(\lambda,x)$. Duplicating the proof 
of 
\eqref{eq:bound:hat Delta}, we conclude that there exists a constant $C$ such that 
\begin{equation*}
\bigl\vert \partial_{x} \widehat{\Delta}(\lambda,x) \bigr\vert \leq \frac{C}{\vert \lambda \vert^2\,\text{Re} \ \lambda}, \quad \lambda \in \cc_{+},
\ x \in [\overset{\circ}{x}_1,\overset{\circ}{x}_2].
\end{equation*}

Next, we differentiate the first line in 
\eqref{eq:widehat:Delta:cc} with respect to $x \in [\overset{\circ}{x}_1,\overset{\circ}{x}_2]$. We deduce that 
\begin{equation*}
\sqrt{2 \lambda} 
\widehat{\Delta}(\lambda,x) + \partial_{x}
\widehat{\Delta}(\lambda,x) = 2 \sqrt{2 \lambda} C_{1}(\lambda) e^{\sqrt{2\lambda} x}. 
\end{equation*}
Modifying the value of the constant $C$ if necessary, we obtain, for $\text{Re}(\lambda) \geq 1$, 
\begin{equation}
\label{eq:exp:bound:1}
\bigl\vert  C_{1}(\lambda) e^{\sqrt{2\lambda} x} \bigr\vert \leq \frac{C}{\vert \lambda \vert^2}, \quad 
 x \in [\overset{\circ}{x}_1,\overset{\circ}{x}_2]
\end{equation}
and, similarly, 
\begin{equation}
\label{eq:exp:bound:2}
\bigl\vert  C_{2}(\lambda) e^{\textcolor{blue}{-}\sqrt{2\lambda} x} \bigr\vert \leq \frac{C}{\vert \lambda \vert^2}, \quad 
 x \in [\overset{\circ}{x}_1,\overset{\circ}{x}_2].
\end{equation}
Next, we consider $x \in \cc$ with $\text{Re} \ x \in [\overset{\circ}{x}_1,\overset{\circ}{x}_2]$ and write
\begin{equation}
\label{eq:exp:bound:3}
\widehat{\Delta}(\lambda,x)=C_1(\lambda)\,e^{\sqrt{2\lambda}\overset{\circ}{x}_2}
\,e^{\sqrt{2\lambda}(x-\overset{\circ}{x}_2)}
+C_2(\lambda)\,e^{-\sqrt{2\lambda}\overset{\circ}{x}_1}
\,e^{\sqrt{2\lambda}(\overset{\circ}{x}_1-x)}.
\end{equation}
For $\lambda \in \cc$ with $\text{Re}(\lambda) \geq 1$, we write
$\sqrt{2\lambda} = a + ib$, with $a>0$. Then, $a^2-b^2 \geq 2$ and
\begin{equation*}
\text{Re} \bigl( \sqrt{2 \lambda} (x- \overset{\circ}{x}_2) \bigr) = a \, 
\text{Re} ( x- \overset{\circ}{x}_2) - b\,\text{Im}(x).   
\end{equation*}
If $\vert \text{Im} \ x\vert \leq  
\text{Re} ( \overset{\circ}{x}_2-x)/2$, then
\begin{equation*}
\text{Re} \bigl( \sqrt{2 \lambda} (x- \overset{\circ}{x}_2) \bigr) = a \, 
\text{Re} ( x- \overset{\circ}{x}_2) - b \, \text{Im}(x)
\leq - \frac{a}{2}\,    
\text{Re} (  \overset{\circ}{x}_2-x)
\leq - \frac1{4} \sqrt{\vert \lambda \vert}\, \text{Re} (  \overset{\circ}{x}_2-x).
\end{equation*}
By the same argument,
if $\vert \text{Im}(x)\vert \leq  
\text{Re} (x- \overset{\circ}{x}_1)/2$, then
\begin{equation*}
\text{Re} \bigl( \sqrt{2 \lambda} ( \overset{\circ}{x}_1-x) \bigr) 
\leq - \frac14 \sqrt{\vert \lambda \vert}\, \text{Re} ( x- \overset{\circ}{x}_{1}).
\end{equation*}
Hence, the conclusion is that, for $\varepsilon \in (0,(\overset{\circ}{x}_2-\overset{\circ}{x}_1)/2)$,
$\text{Re}(x) \in (\overset{\circ}{x}_1+\varepsilon,\overset{\circ}{x}_2-\varepsilon)$
and $\vert \text{Im} \ x\vert \leq  {\varepsilon/2}$, 
\begin{equation*}
\text{Re} \bigl( \sqrt{2 \lambda} (x- \overset{\circ}{x}_2) \bigr) \leq - \frac{\varepsilon}4 \sqrt{\vert \lambda \vert}, 
\quad 
\text{Re} \bigl( \sqrt{2 \lambda} ( \overset{\circ}{x}_1-x) \bigr) \leq - \frac{\varepsilon}4 \sqrt{\vert \lambda \vert}, 
\end{equation*}
which, along with
\eqref{eq:exp:bound:1},
\eqref{eq:exp:bound:2}
and 
\eqref{eq:exp:bound:3},
yields
\begin{equation*}
\vert
\widehat{\Delta}(\lambda,x) \vert \leq
\frac{C}{\vert \lambda \vert^2} 
\exp \Bigl( - \frac{\varepsilon}4  \sqrt{\vert \lambda \vert} \Bigr).
\end{equation*}
This completes the proof.}
\end{proof}

\subsection{Further properties of the gradient}
\label{subse:3.3}

By combining 
Propositions	\ref{prop:reg:barrier}
and 
\ref{prop:analytic.1}, we obtain the following proposition.

\begin{proposition}\label{le:px.cont}
Fix an arbitrary time horizon $T>0$ together with a time $t \in [0,T]$ and assume that $\rho(t,\cdot)$ is 
locally monotone in a right neighborhood of any point in $[0,\infty)$.
Then, $\partial_x p$ is continuous on $(t,T_{reg}(t)) \times[0,\infty)$ and satisfies, on $(t,T_{reg}(t))\times(0,\infty)$, 
\begin{equation}
\label{PDE:partialxp}
\partial_t(\partial_x p)=\frac{1}{2}\partial_{xx}(\partial_x p)+\dot{\Lambda}_t\,\partial_x(\partial_x p).
\end{equation}
\end{proposition}

\begin{proof}
%
The first part is a straightforward consequence of 
Proposition	\ref{prop:reg:barrier}. Equation \eqref{PDE:partialxp}
follows from the following computation: 
\begin{equation}
\begin{split}
&\;\partial_{tx}p(t,x)=\frac{\mathrm{d}}{\mathrm{d}t}\partial_x u(t,x+\Lambda_t)
=\partial_{xt} u(t,x+\Lambda_t)+\partial_{xx}u(t,x+\Lambda_t)\,\dot{\Lambda}_t \\
& =\frac{1}{2}\partial_{xxx}u(t,x+\Lambda_t)+\partial_{xx}u(t,x+\Lambda_t)\,\dot{\Lambda}_t
=\frac{1}{2}\partial_{xxx}p(t,x)+\partial_{xx}p(t,x)(t,x)\,\dot{\Lambda}_t,
\end{split}
\end{equation}
which holds 
pointwise on $(t,T_{reg}(t))\times(0,\infty)$. 
\end{proof}


Next, we deduce the following result.

\begin{lemma}\label{le:lambda.pos}
Fix an arbitrary time horizon $T>0$ together with a time $t \in [0,T]$ and assume that $\rho(t,\cdot)$ is 
locally monotone in a right neighborhood of any point in $[0,\infty)$.
Then, for any $s\in(t,T_{reg}(t))$, we have $\dot{\Lambda}_s=\tfrac{\alpha}2 \partial_x p(s,0) >0$.
\end{lemma}

\begin{proof}
First, we verify the identity 
$\dot{\Lambda}_s=\tfrac{\alpha}2 \partial_x p(s,0)$ for 
$s \in (t,T_{reg}(t))$. 
To do so, we follow the proof of Proposition \ref{prop:reg:barrier}. For any $\eta \in (0,(T_{reg}(t)-t)/2)$, we recall from the fourth step of the proof that the initial sub-density $p(t+\eta,\cdot)$ satisfies  
	\cite[Lemma 5.3]{DIRT1} (which is stated in a slightly different framework). Choosing therein $\dot{\Lambda}$ as a drift,
the latter result states that, for $s \in (t+\eta,T_{reg}(t)-\eta)$, 
\begin{equation*}
\frac{\mathrm{d}}{\mathrm{d}s}\,{\mathbb P}( \tau \leq s  ) = \frac{1}{2} \partial_{x} p(s,0),
\end{equation*} 
which yields  
\begin{equation*}
\dot{\Lambda}_{s} = \frac{\alpha}{2} \partial_{x} p(s,{0}
).
\end{equation*}

For the second part of the statement, it suffices to show that, for every $s\in(t,T_{reg}(t))$, 
\begin{equation}\label{cdf_quad}
\frac{\mathrm{d}^2}{\mathrm{d}x^2}\Big\vert_{x=0}\,\PP\big(\inf_{r\in[0,s]} X_r>0,\,X_s\in(0,x]\big) >0. 
\end{equation}
%
%
To verify \eqref{cdf_quad} we employ the explicit formula for the joint distribution of the value and the running maximum of a standard Brownian motion (see e.g. \cite[chapter 2, Proposition 8.1]{KaSh}) and find that, for any $\eta \in (0,(T_{reg}(t)-t)/2)$ and $s \in (t+\eta,T_{reg}(t)- \eta)$, the following holds with $\widetilde{\Lambda}:=\Lambda-\Lambda_{t+\eta}$:
\begin{equation}
\begin{split}
& \;\PP\big(\inf_{r\in[0,s]} X_r>0,\,X_s\in(0,x]\big) \\
& \ge \int_{\widetilde{\Lambda}_s}^\infty p(t+\eta,y)\;\PP\big(y+\inf_{r\in[t+\eta,s]} (B_r-B_{t+\eta}) \ge \widetilde{\Lambda}_s,\;y+B_s-B_{t+\eta}\in(\widetilde{\Lambda}_s,\widetilde{\Lambda}_s+x]\big)\,\mathrm{d}y \\
& = \int_{\widetilde{\Lambda}_s}^\infty p(t+\eta,y)\int_{y-\widetilde{\Lambda}_s-x}^{y-\widetilde{\Lambda}_s} \int_0^{y-\widetilde{\Lambda}_s} \frac{2(2b-a)}{\sqrt{2\pi (s-t-\eta)^3}}\,e^{-\frac{(2b-a)^2}{2(s-t-\eta)}}\,\mathrm{d}b\,\mathrm{d}a\,\mathrm{d}y \\
&=:\int_{\widetilde{\Lambda}_s}^\infty p(t+\eta,y)\,{F}(x,y)\,\mathrm{d}y. 
\end{split}
\end{equation}
{Next, we write (with ${f}$ defined implicitly by the above):
\begin{equation*}
\begin{split}
F(x,y) = \int_{y-\widetilde{\Lambda}_s-x}^{y-\widetilde{\Lambda}_s} \int_{0}^{y-\widetilde{\Lambda}_s} f(2b-a) \,\mathrm{d}b\,\mathrm{d}a
&= \int_{0}^x \int_{-(y-\widetilde{\Lambda}_s)/2}^{(y-\widetilde{\Lambda}_s)/2}f(2b+a) \,\mathrm{d}b\,\mathrm{d}a \\
&=
\int_{0}^x \int_{-(y-\widetilde{\Lambda}_s)/2}^{(y-\widetilde{\Lambda}_s)/2} \int_{0}^a f'(2b+v) 
\,\mathrm{d}v \, \mathrm{d}b\,\mathrm{d}a,
\end{split}
\end{equation*}
where the last equality follows from the fact that $f$ is odd. Therefore, Fubini's theorem yields:
\begin{equation*}
\begin{split}
\int_{\widetilde{\Lambda}_s}^\infty p(t+\eta,y)\,{F}(x,y)\,\mathrm{d}y
&= \int_{0}^x \int_{0}^a \int_{{\widetilde{\Lambda}_{s}}}^{\infty} p(t+\eta,y) \int_{-(y-\widetilde{\Lambda}_s)/2}^{(y-\widetilde{\Lambda}_s)/2} 
f'(2b+v) 
\,\mathrm{d}b \, \mathrm{d}y \,\mathrm{d}v \,\mathrm{d}a
\\
&=  \int_{0}^x \int_{0}^a \int_{{\widetilde{\Lambda}_{s}}}^{\infty} p(t+\eta,y)\, 
\frac{ f( y - \tilde \Lambda_{s} +v) - 
f( - (y - \tilde \Lambda_{s}) +v)}2 \,\mathrm{d}y \,\mathrm{d}v \,\mathrm{d}a.
\end{split}
\end{equation*}
Thus, we obtain 
\begin{equation*}
\frac{d^2}{dx^2}\Big\vert_{x=0}\,\PP\big(\inf_{r\in[0,s]} X_r>0,\,X_s\in(0,x]\big)
\geq \int_{{\widetilde{\Lambda}_{s}}}^{\infty} p(t+\eta,y)\, 
f( y - \tilde \Lambda_{s})\,\mathrm{d}y. 
\end{equation*}
It remains to observe that
\begin{equation*}
\int_{\widetilde{\Lambda}_s}^{\infty} p(t+\eta,y)\, 
 f( y - \widetilde{\Lambda}_{s}) \, \mathrm{d}y>0,
\end{equation*}
in view of the probabilistic interpretation of the latter integral.} 
\end{proof}

We conclude this section by verifying the absolute integrability of $\partial_xp(t,\cdot)$.

\begin{lemma}
\label{lem:3.8}
For any $t>0$, 
$\lim_{x \to \infty} p(t,x)=0$ and
$\partial_xp(t,\cdot)\in L^1((0,\infty))$.
\end{lemma}
\begin{proof}
Let $(X,\Lambda)$ be a physical solution of \eqref{Stefan_prob}.
First, we assume that $X_{0-}=x>0$ and, for any $0<y_1<y_2<\infty$, obtain:
$$
\PP\bigl(X_t\,\bone_{\{\inf_{r\in[0,t]} X_r>0\}} \in [y_1,y_2]\bigr)
=\PP\bigl(\inf_{r\in[0,t]} X_r>0\,\big\vert\, X_t\in[y_1,y_2]\bigr) \,
\PP\bigl( X_t\in[y_1,y_2] \bigr).
$$
Recall that a Brownian motion started at $x$ and conditioned to be equal to $z$ at the terminal time $t$ is a Brownian bridge from $x$ to $z$ on $[0,t]$.
Denote as before by $g(t,x-\cdot)$ the density at time $t$ of a Brownian motion started at $x$, and take a limit as $y_1,y_2\rightarrow y>0$, to obtain
$$
p(t,y) = g(t,x -y-\Lambda_t)\, \PP \Bigl( \inf_{r\in[0,t]} \bigl(x + B_r + r(y + \Lambda_t - x - B_t)/t - \Lambda_r\bigr)>0\Bigr).
$$
({See also \cite[Proposition 4]{GasSotVal} for a direct derivation of this formula.})
By Proposition \ref{prop:analytic.1}, we conclude that $p(t,\cdot)\in C^{\infty}((0,\infty))$. Moreover, $\lim_{x \to \infty} p(t,x)=0$. 
Next, we notice that the function
$$
\hat{p}(t,\cdot) : (0,\infty) \ni y \mapsto \PP \Bigl( \inf_{r\in[0,t]} \bigl(x + B_r + r(y + \Lambda_t - x - B_t)/t - \Lambda_r\bigr)>0\Bigr)
$$
is $[0,1]$-valued and non-decreasing. Since $\hat{p}(t,\cdot)$ is continuously differentiable (thanks to $p(t,\cdot)\in C^{\infty}((0,\infty))$) and $g(t, \textcolor{blue}{\cdot})$ is smooth, 
we have $\partial_x\hat{p}(t,\cdot)\in L^1((0,\infty))$, with the norm bounded uniformly over $t>0$.

For a general $X_{0-}$, we have
$$
p(t,y)=
\int_0^{\infty} g(t,x-y-\Lambda_{t})\,\PP \Bigl( \inf_{r\in[0,t]} \bigl(x + B_r + r(y + \Lambda_t - x - B_t)/t - \Lambda_r\bigr)>0\Bigr)\,\mu(\mathrm{d}x),
$$
where $\mu$ is the distribution of the positive part of $X_{0-}$.
Using the observation at the end of the preceding paragraph, as well as the fast decay (in $x$) 
of $g(t,x)$ and of its $x$-derivative, we easily obtain the statement of the lemma by means of Fubini's theorem. 
\end{proof}


\section{Proof of Theorem \ref{thm1}}
\label{se:4}
In addition to the results of Sections \ref{se:2} and \ref{se:3}, we need the next lemma for our proof of Theorem \ref{thm1}.

\subsection{The number of monotonicity-changing points}

The next lemma is at the core of our analysis.

\begin{lemma}\label{zero_number}
Let $X_{0-}$ admit a density $f$ on ${(0,\infty)}$ 
that changes monotonicity finitely often on compacts {of $[0,\infty)$}. Then, the same applies to $X_{t-}\,\mathbf{1}_{\{\tau\ge t\}}$, $t>0$, for every physical solution $(X,\Lambda)$ of \eqref{Stefan_prob} started from $X_{0-}$. 
\end{lemma}

\begin{proof}
\textit{First Step.} We refer to the property of a random variable described in the lemma as (P) and argue by contradiction. To this end, we suppose
\begin{equation}\label{what_is_tstar}
t_*:=\inf\mathcal{T}:=\inf\big\{t>0:\,X_{t-}\,\mathbf{1}_{\{\tau\ge t\}}\;\text{violates (P)}\big\}<\infty.
\end{equation}
{The first (easy) step is to check that, necessarily, $t_*\in\mathcal{T}$ (which, in particular, implies  $t_*>0$).
Indeed, if $t_* \not \in {\mathcal T}$, 
then
by Proposition \ref{prop:reg:barrier}, Lemma \ref{le:lambda.pos} and Proposition \ref{prop:analytic.1}, respectively,
there exists an $\epsilon>0$ such that $X_{s-}\,\mathbf{1}_{\{\tau\ge s\}}=X_s\,\mathbf{1}_{\{\tau>s\}}$, $\lim_{x\downarrow0}\partial_x\rho(s,x)=\lim_{x\downarrow0}\partial_x p(s,x)>0$ and $\partial_x\rho(s,\cdot)=\partial_x p(s,\cdot)$ are analytic on $(0,\infty)$ for all $s\in(t_*,t_*+\epsilon)$. 
Clearly, this implies that $[t_*,t_*+\epsilon) \subset {\mathcal T}$, which contradicts \eqref{what_is_tstar}.}
%
%
\medskip

To deduce a contradiction from $t_*\in\mathcal{T}$ we pick an arbitrary $R>\Lambda_{2t_*}-\Lambda_{t_*-}$ and aim to show that $\rho(t_*,\cdot)$ changes monotonicity finitely often on $[0,R]$. For this purpose, it suffices to obtain a uniform upper bound $M<\infty$
on the number of monotonicity changes of $p(s,\cdot)$ on $[0,R]$ for all $s\in[t_*/2,t_*)$. Indeed, then the definition of $t_*$, Lemma \ref{le:lambda.pos} and a diagonalization argument would yield a sequence $[t_*/2,t_*)\ni t_n\uparrow t_*$ and $0\le x^{(n)}_1\le x^{(n)}_2\le\cdots\le x^{(n)}_{M+1}=R$, $n\in\nn$ converging when $n\to\infty$ to $0\le x_1\le x_2\le\cdots\le x_{M+1}=R$, respectively, such that each $p(t_n,\cdot)$ is non-decreasing on $[0,x^{(n)}_1),\,[x^{(n)}_2,x^{(n)}_3),\,\ldots$ and non-increasing on $[x^{(n)}_1,x^{(n)}_2),\,[x^{(n)}_3,x^{(n)}_4),\,\ldots\,$. 
 Hence, the cumulative distribution function of $X_{t_*-}\,\mathbf{1}_{\{\tau\ge t_*\}}$, as the pointwise limit on $(0,\infty)$
  of the cumulative distribution functions of $X_{t_n}\,\mathbf{1}_{\{\tau>t_n\}}$, $n\in\nn$,  would be convex on $[0,x_1),\,[x_2,x_3),\,\ldots$ and concave on $[x_1,x_2),\,[x_3,x_4),\,\ldots\,$.
 (The convexity/concavity at $0$ would follow from the right-continuity of cumulative distribution functions.) Recalling from Proposition 
\ref{prop:analytic.1}
that $\rho(t_*,\cdot)$ is smooth on $(0,\infty)$, $\rho(t_*,\cdot)$ would change monotonicity finitely often on $(0,R]$. In particular, 
$\rho(t_{*},\cdot)$ would have a limit at $0$ and the resulting extension would 
change monotonicity finitely often on $[0,R]$.


\vskip 4pt

\textit{Second Step.} To find an upper bound $M<\infty$ as desired (which is the precise purpose of the remaining steps in the proof), we make the change of variables \eqref{u_to_p} and consider the zero set of $\partial_x u$ on 
{$D_*:=\{(t,x)\in[t_*/2,t_*]\times(0,\infty):\,x>\Lambda_t\}$.} 
On $D_*$, the $C^\infty$-function $\partial_x u$ is analytic in $x$ and solves $\partial_t(\partial_x u)=\frac{1}{2}\partial_{xx}(\partial_x u)$ (cf.~Proposition \ref{prop:analytic.1}). We conclude that, for every point $(t,x)\in D_*$, either $\partial_x u(t,x)\neq 0$ and the zero set of $\partial_x u$ is empty on a neighborhood $(t-\delta_1,t+\delta_1)\times(x-\delta_2,x+\delta_2)$, or $\partial_x u(t,x)=0$ and there exists a smallest $k\in\nn \setminus \{0\}$ with $\partial_x^k\partial_x u(t,x)\neq 0$. 

By Lemma \ref{lem:aux:1} below, 
 the latter case results in a neighborhood $(t-\delta_1,t+\delta_1)\times(x-\delta_2,x+\delta_2)$ on which the zero set of $\partial_x u$ is the union of $k$ curves, each containing $(t,x)$. $2 \lfloor k/2\rfloor$ curves are given by the graphs (in the (time, space)-coordinate system) of continuous functions on $[t-\delta_{1},t]$ that are $C^{\infty}$ on $[t-\delta_{1},t)$. In the (space, time)-coordinate system, the latter curves form $\lfloor k/2 \rfloor$ strictly convex smooth paraboloids with graphs in the negative half-space only. If $k$ is odd, there is another curve given by the graph of a $C^\infty$-function on $[t-\delta_{1},t+\delta_1]$ (in the (time, space)-coordinate system; in particular, it crosses the $x$-axis).

Covering the segment $[t_*/2,t_*]\times\{R+\Lambda_{t_*-}\}$ by such neighborhoods of its elements and extracting a finite subcover we find a non-trivial rectangle $[t_*/2,t_*]\times[R+\Lambda_{t_*-},R+\Lambda_{t_*-}+\delta]$ contained in the latter. We claim that, consequently, for $N\in\NN$ large enough, $\partial_x u$ has at most a finite number of zeros along the curve $\theta$ that linearly interpolates between the points
\begin{equation*}
\begin{split}
& \Big(\frac{t_*}{2},\,R+\Lambda_{t_*-}\Big),\;\;\Big(\frac{t_*}{2}+\frac{t_*}{4N},\,R+\Lambda_{t_*-}+\delta\Big),\;\;
\Big(\frac{t_*}{2}+2\frac{t_*}{4N},\,R+\Lambda_{t_*-}\Big),\\
& \Big(\frac{t_*}{2}+3\frac{t_*}{4N},\,R+\Lambda_{t_*-}+\delta\Big),\;\;
\Big(\frac{t_*}{2}+4\frac{t_*}{4N},\,R+\Lambda_{t_*-}\Big),\,\ldots,
\;(t_*,\,R+\Lambda_{t_*-}).
\end{split}
\end{equation*} 
Indeed, it suffices to prove that the number of intersection points between $\theta$ and the aforementioned 
zero curves is finite in a given neighborhood of the form 
$(t-\delta_1,t+\delta_1)\times(x-\delta_2,x+\delta_2)$ (provided we choose
$N$ large enough, which is always possible since we just need to handle a finite number of these neighborhoods). Obviously, by choosing  
$N$ large enough, we see that, if it exists, the extra curve (say $\zeta$)  that goes through the $x$-axis has at most one intersection point with any linear segment of $\theta$ (choose the slope of $\theta$ greater than the maximum of the time derivative $\dot{\zeta}$). 
As for the other zero curves, we know that they 
are strictly convex or concave on $(t-\delta_{1},t)$. Hence, whatever 
the value of $N$, any of those other curves has at most two intersection points with any linear segment of $\theta$. 
\vskip 4pt

\textit{Third Step.} Next, we fix an $s\in[t_*/2,t_*)$. The purpose of this step is to verify the following assertion. 

\smallskip

\noindent\textbf{Assertion.} For all $x\in(\Lambda_s,\theta_s)$ with $\partial_x u(s,x)=0$, there exist a $\underline{t}\in[t_*/2,s)$ and a continuous function $\zeta:\,[\underline{t},s]\to(0,\infty)$ (we shall say that $\zeta$ is ${C}^0$ on $[\underline{t},s]$)
 such that
\begin{enumerate}[(a)]
\item $\zeta_s=x$ and $\zeta_t\in(\Lambda_t,\theta_t)$, $\partial_x u(t,\zeta_t)=0$ when $t\in(\underline{t},s]$; 
\item for all $t\in(\underline{t},s)$, there exists a neighborhood
\begin{equation*}
(t-\delta_1,t+\delta_1)\times(\zeta_t-\delta_2,\zeta_t+\delta_2)\subset \{(t,x)\in(t_*/2,s]\times(0,\infty):\,\Lambda_t<x<\theta_t\}=:\Gamma_{t_*/2,s}
\end{equation*}
on which $\partial_x u(r,y)=0$, for $r \leq s$, implies $y\ge\zeta_r$;
\item one has
\begin{equation}
\big(\underline{t},\zeta_{\underline{t}}\big)\in 
\theta_{[t_*/2,s]} \cup (\{t_*/2\}\times[\Lambda_{t_*/2},R+\Lambda_{t_*-}])\cup\Lambda_{[t_*/2,s]}
=:\partial_{\text{par}}\Gamma_{t_*/2,s}. 
\end{equation}
\end{enumerate}  
\noindent Let us fix any $x$ as in the statement of the assertion and construct the desired curve $\zeta$.
First, the local description of the zero set of $\partial_x u$ implies that
\begin{equation}
\begin{split}
\mathcal{T}_0\! &:=\!\big\{r\in[t_*/2,t_*):\,\text{(a), (b) hold, with $r$ replacing $\underline{t}$,} \\
&\hspace{100pt} \text{for a continuous function}\;\zeta:\,(r,s]\to(0,\infty)\big\}
\end{split}
\end{equation}
is non-empty. Hence, $\underline{t}:=\inf\mathcal{T}_0\in[t_*/2,t_*)$ is well-defined. 
Next, we notice that, for any $r\in\mathcal{T}_0$, the corresponding continuous function $\zeta$ is uniquely determined. Indeed, thanks to Lemma \ref{lem:aux:1} (recall also the local description of the zero set of $\partial_x u$, provided earlier), there exists a minimal curve in the left neighborhood of $s$. This shows that $\zeta$ is uniquely determined in the left neighborhood of $s$. Then, it is uniquely determined on $(r,s]$, as otherwise the property (b) would be violated for the smallest $t$ such that the two candidate functions agree on $[t,s]$. Uniqueness also implies that, for any $r,r'\in \mathcal{T}_0$, with $r'< r$, the function $\zeta$ corresponding to $r'$ coincides with the one for $r$, on $(r,s]$.
Thus, all functions corresponding to the elements of $\mathcal{T}_0$ combine to a continuous function $\zeta:\,(\underline{t},s]\to(0,\infty)$ satisfying (a), (b). Moreover, by the intermediate value theorem, the limit points of $\zeta$ as $t\downarrow\underline{t}$ form an interval. However,  $\partial_xu(\underline{t},\cdot)\equiv0$ on its interior, and in view of the analyticity of $\partial_xu(\underline{t},\cdot)$ this interval must consist of a single point, which shows that $\lim_{t\downarrow\underline{t}} \zeta_t$ is well-defined. 
Now, we claim that $(\underline{t},\underline{x}:=\lim_{t\downarrow\underline{t}} \zeta_t) \not \in \Gamma_{t_*/2,s}$. Indeed, if 
$(\underline{t},\underline{x})\in \Gamma_{t_*/2,s}$, 
we know from Lemma \ref{lem:aux:1} that
$\zeta$ can be extended to the left of $(\underline{t},\underline{x})$ by choosing the smallest zero curve in the left neighborhood of $\underline{t}$, obtaining a contradiction to the definition of $\underline{t}$. Thus, $\zeta$ satisfies (c).


\vskip 4pt

\textit{Fourth Step.} Let us consider the curves $\zeta$ as in the above Assertion, for all possible $x\in(\Lambda_s,R+\Lambda_{t_*-})$ such that $\partial_x u(s,x)=0$, and with $s$ fixed. Whenever this causes no ambiguity, we will refer to them simply as ``zero curves". We claim that no two zero curves intersect in $\Gamma_{t_*/2,s}\cup \partial_{\text{par}}\Gamma_{t_*/2,s}$. Any potential intersection in $\Gamma_{t_*/2,s}\cup\theta_{[t_*/2,s]} \cup (\{t_*/2\}\times(\Lambda_{t_*/2},R+\Lambda_{t_*-}])$ is ruled out by applying the maximum principle for classical solutions of the heat equation on the region bounded by the two intersecting curves and recalling the analyticity of $\partial_x u$ in $x$ on $D_*$. 

\medskip

It remains to exclude the scenario that two zero curves, say $\zeta^{1}<\zeta^{2}$, approach the same point $(s_*,\Lambda_{s_*})\in \Lambda_{[t_*/2,s]}$. First, we notice that $\zeta^{i}$ restricted to any compact in $D_*$ is $C^{\infty}$ at all points, except for at most a finite number of them, as follows from the local description of the zero set of $\partial_x u$. This means that, for every $\zeta^{i}$, the only possible accumulation point of such singular points is $s^*$. Therefore, we can iterate over the singular points on each curve, smoothing out the functions at these points, and obtain new curves, $\zeta^{(1)}<\zeta^{(2)}$, which are $C^{\infty}$-functions on $(s^*,s]$, taking values in $(\Lambda_t,\theta_{t})$ for every $t\in(s^*,s]$, and converging to $\Lambda_{s_*}$ as $t\downarrow s_*$. In addition, for any $\bar{\varepsilon}>0$, we can choose $\zeta^{(1)}$ and $\zeta^{(2)}$, respectively, so that they are $\bar{\varepsilon}$-close to $\zeta^{1}$ and $\zeta^{2}$, and that $|\partial_x u(t,\zeta^{(i)}_t)|\leq \bar{\varepsilon}$ for all $t\in(s^*,s]$. In fact, by making the approximation finer as we get closer to $s_{*}$, we can also make sure that $\lim_{t\downarrow s_*} \partial_x u(t,\zeta^{(i)}_t)=0$. For most of the subsequent derivations in this step, we fix an arbitrary $\bar{\varepsilon}>0$ and the associated $\zeta^{(1)}$, $\zeta^{(2)}$.

\medskip

The main challenge in this step of the proof lies in the lack of a priori continuity and boundedness of $\partial_x u$ near the \textit{boundary} point $(s_*,\Lambda_{s_*})$, which does not allow us to directly apply the maximum principle or the Feynman-Kac formula to obtain a contradiction. We, therefore, need to justify the Feynman-Kac formula 
\begin{equation}\label{final_FK}
\partial_xu(t,x)=\EE[\partial_xu(t-\tau^{t,x},x+B_{\tau^{t,x}})]=:v(t,x),\quad x\ge\zeta^{(1)}_t,\ s_*\le t\le s,\ (t,x)\neq (s_{*},\zeta_{s_{*}}^{(1)}),
\end{equation}
where $\tau^{t,x}:=\inf\{r\ge0:x+B_r=\zeta^{(1)}_{t-r}\}\wedge(t-s_*)$. 
The proof is deferred to Lemma \ref{lem:aux:2} below. As a consequence of \eqref{final_FK}, the function $v$ is well-defined except maybe when $t=s_{*}$ and $x=\zeta_{s_{*}}^{(1)}=\Lambda_{s_{*}}$, in the sense that, for 
$x\ge\zeta^{(1)}_t$ and $s_*\le t\le s$ with $(t,x) \neq (s_{*},\Lambda_{s_{*}})$, 
$\EE[ \vert \partial_xu(t-\tau^{t,x},x+B_{\tau^{t,x}}) \vert] < \infty.$
\medskip

The Markov property of standard Brownian motion and the Feynman-Kac formula \eqref{final_FK} show that, for $t\in(s_*,s)$, $x\in(\zeta^{(1)}_t,\zeta^{(2)}_t)$, the process $\partial_x u(t-r\wedge\tau^{t,x},x+B_{r\wedge\tau^{t,x}})$, $r\in[0,t-s_*]$ is given by the conditional expectations of its terminal value with respect to the filtration of $B_r$, $r\in[0,t-s_*]$, consequently a martingale. The optional stopping theorem (see e.g. \cite[chapter II, corollary 3.6]{RY}) renders the process $\partial_x u(t-r\wedge\tau^{t,x}\wedge\underline{\tau}^{t,x},x+B_{r\wedge\tau^{t,x}\wedge\underline{\tau}^{t,x}})=u(t-r\wedge\underline{\tau}^{t,x},x+B_{r\wedge\underline{\tau}^{t,x}})$, $r\in[0,t-s_*]$, with $\underline{\tau}^{t,x}:=\inf\big\{r\ge0:x+B_r\in\{\zeta^{(1)}_{t-r},\zeta^{(2)}_{t-r}\}\big\}$, a martingale as well. Therefore, 
\begin{equation}
|\partial_x u(t,x)|=\big|\EE[\partial_x u(t-(t-s_*)\wedge\underline{\tau}^{t,x},x+B_{(t-s_*)\wedge\underline{\tau}^{t,x}})]\big|=\big|\EE[\partial_x u(t-\underline{\tau}^{t,x},x+B_{\underline{\tau}^{t,x}})]\big|\le\bar{\varepsilon},
\end{equation}
for all $x\in(\zeta^{(1)}_t,\zeta^{(2)}_t)$, $t\in(s_*,s)$. Hence, for all $t\in(s_*,s)$ and $x\in(\zeta^1_t,\zeta^2_t)$, we also have $|\partial_x u(t,x)|\leq \bar{\varepsilon}$, provided $\bar{\varepsilon}>0$ is small enough. In the limit $\bar{\varepsilon}\downarrow0$, we obtain that $\partial_x u(t,\cdot)$ vanishes on a non-trivial interval, contradicting the analyticity of $u(t,\cdot)$ and ruling out $\lim_{t\downarrow s_*} \zeta^{1}_t=\Lambda_{s_*}=\lim_{t\downarrow s_*} \zeta^{2}_t$. 
\vskip 4pt

\textit{Fifth Step.} To conclude, we denote, for a given $s\in[t_*/2,t_*)$, 
\begin{equation}
\overline{x}:=\sup\big\{x\in(\Lambda_s,\theta_s):\,\partial_x u(s,x)=0\,\text{ and }\lim_{t\downarrow\underline{t}} \zeta_t=\Lambda_{\underline{t}}\big\}\vee\Lambda_s,
\end{equation}
with $\underline{t}\in[t_*/2,s)$ and $\zeta:(\underline{t},s]\to(0,\infty)$ of the Assertion in the third step (in particular, $\zeta_{s}=x$), and claim that $\partial_x u(s,\cdot)\ge 0$ on $(\Lambda_s,\overline{x}]$. Indeed, for any fixed zero $x\in(\Lambda_s,\overline{x}]$ of $\partial_x u(s,\cdot)$ and with the corresponding $T_{reg}(\underline{t})$ of Proposition \ref{prop:reg:barrier}, each $(t,\zeta_t)$, $t\in[(\underline{t}+T_{reg}(\underline{t}))/2,s)=:[t_{reg},s)$ admits a neighborhood $(t-\delta_1,t+\delta_1)\times(\zeta_t-\delta_2,\zeta_t+\delta_2)$ satisfying the property stated in part (b) of the Assertion in the third step. By adding a rectangle $(s-\delta_1,s]\times(\zeta_s-\delta_2,\zeta_s+\delta_2)$ on which $\partial_x u(r,y)=0$ implies $y\ge\zeta_r$ and extracting a finite subcover of $\zeta_{[t_{reg},s]}$ we construct left neighborhoods of $x$ and $\zeta_{t_{reg}}$ on which the signs of $\partial_x u(s,\cdot)$ and $\partial_x u(t_{reg},\cdot)$, respectively, coincide. Moreover, $\partial_x u(t_{reg},y)> 0$, $y\in(\Lambda_{t_{reg}},\zeta_{t_{reg}})$, since otherwise Lemma \ref{le:lambda.pos}, the intermediate value theorem, and the Assertion in the third step, would imply the existence of a $\underline{\widetilde{t}}$ and a {continuous} function {$\widetilde{\zeta}:\,[\underline{\widetilde{t}},t_{reg}]\to(0,\infty)$} such that 
\begin{enumerate}[(a)]
\item $\widetilde{\zeta}_t\in(\Lambda_t,\zeta_t)$ and $\partial_x u(t,\widetilde{\zeta}_t)=0$, for all $t\in(\widetilde{\underline{t}},t_{reg}]$; 
\item for all $t\in(\widetilde{\underline{t}},t_{reg})$, there exists a neighborhood
\begin{equation*}
(t-\delta_1,t+\delta_1)\times(\widetilde{\zeta}_t-\delta_2,\widetilde{\zeta}_t+\delta_2)\subset \{(t,x)\in(\widetilde{\underline{t}},t_{reg}]\times(0,\infty):\,\Lambda_t<x<\zeta_t\}
\end{equation*}
on which $\partial_x u(r,y)=0$, for $r \leq t_{reg}$, implies $y\ge\widetilde{\zeta}_r$;
\item one has {$\widetilde{\zeta}_{\widetilde{\underline{t}}}\in\{\zeta_{\widetilde{\underline{t}}},\Lambda_{\widetilde{\underline{t}}}\}$}.
\end{enumerate} 
This would contradict the findings of the third step, if  {$\widetilde{\zeta}_{\widetilde{\underline{t}}}=\zeta_{\widetilde{\underline{t}}}$}, and would contradict Lemma \ref{le:lambda.pos}, if $\widetilde{\underline{t}}>\underline{t}$ and  {$\widetilde{\zeta}_{\widetilde{\underline{t}}}=\Lambda_{\widetilde{\underline{t}}}$}. Hence, $\partial_x u(s,\cdot)\ge 0$ on $(\Lambda_s,\overline{x}]$, as claimed. In particular, $u(s,\cdot)$ does not change monotonicity on $(\Lambda_s,\overline{x}]$.

\medskip

Finally, for any $x\in(\overline{x},\theta_s)$ with $\partial_x u(s,x)=0$, we have $\zeta_{\underline{t}}>\Lambda_{\underline{t}}$
(with $\zeta$ now ending at $\zeta_{s}=x$)
 and the set $\zeta_{(\underline{t},s]}\cup\{(\underline{t},\lim_{t\downarrow\underline{t}} \zeta_t)\}=:\zeta_{[\underline{t},s]}$ can be covered by open rectangles satisfying the property (b) of the Assertion in the third step, together with the two rectangles $[\underline{t},\underline{t}+\delta_1)\times( \zeta_{\underline{t}}-\delta_2, \zeta_{\underline{t}}+\delta_2)$ and $(s-\delta_1,s]\times(\zeta_s-\delta_2,\zeta_s+\delta_2)$ on which $\partial_x u(r,y)=0$ implies $y\ge\zeta_r$. Extracting a finite subcover we deduce the existence of an $\epsilon>0$ such that $\partial_x u$ has the same sign on $\{s\}\times(x-\epsilon,x)$ and on $\theta_{(\underline{t}-\epsilon,\underline{t})}$, if $\underline{t}>t_*/2$, or on $\{s\}\times(x-\epsilon,x)$ and on $\{t_*/2\}\times(\zeta_{\underline{t}}-\epsilon, \zeta_{\underline{t}})$, if $\underline{t}=t_*/2$. Applying this observation first for $x\in(\overline{x},\overline{x}^*]$, where
\begin{equation}
\begin{split}
& x^*=\max\big\{x\in(\Lambda_{t_*/2},\theta_{t_*/2}]:\,\partial_x u(t_*/2,\cdot)\ge0\text{ or }\partial_x u(t_*/2,\cdot)\le 0 \text{ on }(\Lambda_{t_*/2},x]\big\}, \\
& \overline{x}^*=\sup\big\{x\in(\Lambda_s,\theta_s):\,\partial_x u(s,x)=0,\text{ and }\underline{t}=t^*/2,\,\lim_{t\downarrow\underline{t}} \zeta_t\in(\Lambda_{t_*/2},x^*]\text{ or }\lim_{t\downarrow\underline{t}} \zeta_t=\Lambda_{\underline{t}}\big\}\vee\Lambda_s,
\end{split}
\end{equation}
and then for $x\in(\overline{x}^*,\theta_s)$, we see
that $\partial_x u(s,\cdot)$ has at most two sign changes on $(\Lambda_s,\overline{x}^*]$, and its number of sign changes on $(\overline{x}^*,\theta_s]$ cannot exceed the number of zeros of $\partial_x u$ on $\theta_{[t_*/2,t_*)} \cup (\{t_*/2\}\times(x^*,R+\Lambda_{t_*-}])$.
Since the latter is finite, due to the definition of $x^*$ and the analyticity of $\partial_x u(t_*/2,\cdot)$, we complete the proof.
\end{proof}

Our proof of Theorem \ref{thm1} combines Lemma \ref{zero_number} and the results of Sections \ref{se:2} and 
\ref{se:3}.  

\begin{proof}[Proof of Theorem \ref{thm1}]
Lemma \ref{zero_number} gives the first assertion of the theorem and ensures the existence of the limit $\lim_{x\downarrow0} \rho(t,x)$ for all $t>0$. In case \textit{(i)}, it suffices to use the local uniqueness result of \cite[Theorem 1.8]{HLS} in conjunction with a straightforward adaptation of the local existence result in \cite[Theorem 4.1]{DIRT1} to the setting of \eqref{Stefan_prob}. In case \textit{(ii)}, Propositions \ref{prop:sec2.p0.upperbd} and \ref{prop:Holder.1} 
guarantee the $1/2$-H\"older continuity of 
$\Lambda$ on $[t,t+\epsilon)$, for some $\epsilon>0$. In case \textit{(iii)}, the identity \eqref{jump_size} stems directly from the definition of a physical solution, see \eqref{phys_cond}. In all cases, Propositions \ref{prop:reg:barrier} and \ref{prop:analytic.1} yield an $\epsilon>0$ such that $\Lambda\in C^1((t,t+\epsilon))$ and the densities $p(s,\cdot)$, $s\in(t,t+\epsilon)$ are real analytic and form a classical solution of the Dirichlet problem in \eqref{Dir_problem}. The relation $\dot{\Lambda}_s=\frac{\alpha}{2}\partial_x p(s,0)$, $s\in(t,t+\epsilon)$ is now an immediate consequence of $\Lambda_s=\alpha\mathbb{P}(\min_{r\in[0,s]} X_r\le 0)$, $s\in(t,t+\epsilon)$ and \cite[Lemma 2.1(iv)]{DIRT3}. 
\end{proof}

\subsection{Proofs of auxiliary results}

\begin{lemma}
\label{lem:aux:1}
Let $t>0$ and $x> \Lambda_{t}$ be such that 
$\partial_{x} u(t,x)= 0$ and call 
$k\in\nn {\setminus \{0\}}$ the smallest integer such that $\partial_x^k\partial_x u(t,x)\neq 0$. 
Then, 
there is a neighborhood $(t-\delta_1,t+\delta_1)\times(x-\delta_2,x+\delta_2)$ on which the zero set of $\partial_x u$ is the union of $k$ curves. {$2 \lfloor k/2\rfloor$ curves are given by the graphs (in the (time, space)-coordinate system) of continuous functions 
 on $[t-\delta_{1},t]$ that are $C^{\infty}$ on $[t-\delta_{1},t)$. In the (space, time)-coordinate system, the latter curves form 
 $\lfloor k/2 \rfloor$ strictly convex smooth paraboloids with graphs in the negative half-space only. If $k$ is odd, there is another curve given by
 the graph of a $C^\infty$-function on $[t-\delta_{1},t+\delta_1]$ (in the (time, space)-coordinate system; in particular, it crosses the $x$-axis).}
\end{lemma}

\begin{proof}
{The simple fact that $\partial_{x}^k\partial_x u(t+s,x+y) \neq 0$
for $(s,y)$ in the neighborhood of $0$ implies that, 
for a given $s$ in the neighborhood of $0$, 
the equation $\partial_{x} u(t+s,x+y)=0$, with $y$ in a neighborhood of $0$ (the neighborhood being independent of $s$), has at most $k$ roots. 
In fact, by
Proposition  
\ref{prop:analytic.1}, $\partial_{x} u(t+s,x+z)$ can be analytically extended in $z$ to a complex neighborhood of $0$ (again, the neighborhood is independent of $s$), and, then, 
for $(s,z)$ in a neighborhood of $0$, it holds $\partial_{z}^k \partial_{x} u(t+s,x+z) \neq 0$. Therefore, for a given $s$ in the neighborhood of $0$, 
the equation $\partial_{x} u(t+s,x+z)=0$, with $z$ in a complex neighborhood of $0$ (the neighborhood being independent of $s$), has at most $k$ (complex) roots (counting multiplicity). 
\vskip 4pt

\textit{First Step.} Next, we describe these roots on a case-by-case basis. To do so, we follow \cite{AF}, with the only difference being that the procedure therein is applied to time-space analytic functions. Arguing as in \cite[procedure following the display after (5.2)]{AF}, we first freeze $y$ (instead of $s$)
in a real neighborhood of $0$ 
and find $\lfloor k/2 \rfloor$ negative roots $s$ of $\partial_{x} u(t+s,x+y)$ of the form $s=y^2 \sigma_{j}(y)$, for $j \in \{1,\ldots,\lfloor k/2\rfloor\}$, where 
$\sigma_{j}(y)$, $j \in \{1,\ldots,\lfloor k/2 \rfloor\}$ are smooth non-positive-valued functions of $y$ in a neighborhood of $0$. The proof of this fact relies on the expansion of $\partial_x u(t+y^2\sigma,x+y)$ as $y\to 0$ in the form $y^k\sum_{m+2n=k}
[1/(m!n!)]
 \partial^m_x\partial^n_t \partial_x u(t,x)\sigma^n + y^{k+1} H(\sigma,y)$, for a smooth function $H$, which allows us to apply \cite[procedure following the display after (5.2)]{AF}. For $y$ in the neighborhood of $0$, the curves $\sigma_{j}(y)$, $j \in \{1,\ldots,m\}$ are known to be away from $0$, implying via a simple differentiation argument that the curves $y^2 \sigma_{j}(y)$, $j \in \{1,\ldots,m\}$ are (strictly) decreasing for $y>0$ and (strictly) increasing for $y<0$; moreover, a second derivative computation reveals that
 they are strictly convex. In other words, going backward in time from $(t,x)$, we (locally) find $2 \lfloor k/2 \rfloor$ curves of zeros of $\partial_{x} u$ that are continuous up to and including $t$ but that are smooth up to and excluding $t$ (because the curves behave like $\sqrt{-s}$ as $s \rightarrow 0$). If $k$ is even, we have (locally) exhausted all the roots of $\partial_{x} u(t+s,x+y)$ for $s<0$. 
 
When $s=0$, the $k$ roots are obviously at $y=0$ (counting multiplicity). 

When $s>0$, we may repeat the same procedure as before by expanding 
$\partial_x u(t-y^2\sigma,x+ i y)$ as $y\to 0$, with $i^2=-1$ (or, equivalently,
by expanding  
$\partial_x u(t+z^2\sigma,x+ z)$ as $z\to 0$, with $z$ complex, which is licit since the time derivatives are also analytically extended to a complex neighborhood of $(t,x)$ thanks to Proposition  
\ref{prop:analytic.1}). We, then, find $k$ purely imaginary roots 
to $\partial_x u(t+s,x+z)$, for $s>0$ close to $0$. When $k$ is even, this exhausts all the possible complex roots of  
$\partial_x u(t+s,x+z)$ and, in particular, there are no real roots.  
\vskip 4pt

\textit{Second Step.} In order to complete the picture, it remains to address the case of $k$ being odd. By the above procedure, we already have $2 \lfloor k/2\rfloor$ roots $(s,y)$ with $s<0$ and 
$2 \lfloor k/2\rfloor$ (purely) imaginary roots $(s,i y)$ with $s>0$. In both cases, we are missing one root. 
Writing $k=2 \ell+1$, we then follow 
\cite[procedure following the display after (5.3)]{AF}, again, paying attention to the fact that the solution therein is time-space analytic, while it is not the case here. In this regard, the main difficulty is to handle the case $r=\infty$, where $r$
is the first positive integer such that $\partial_{x}^{k+2r} u(t,x) \neq 0$ (or, equivalently, $\partial_{t}^r \partial_{x}^k u(t,x) \neq 0$). In the time-space analytic setting, $r=\infty$ forces $\partial_{x}^k u(s,x)=0$ for $s$ in the neighborhood of $0$, yielding $y=0$ as the missing root. This logic does not apply in our setting, and we must argue differently. 

For completeness, we start by considering the case $r<\infty$. Adapting 
\cite[procedure following the display after (5.3)]{AF}, we now expand 
$\partial_x u(t+s,x+\eta s^r)$ as $s\to 0$ 
(for a given $\eta$):
\begin{equation*}
\begin{split}
\partial_x u(t+s,x+\eta s^r) &= \sum_{m+2n \geq k, n+rm \leq \ell+r} \frac1{2^n m! n!} \partial_{x}^{m+2n+1}u(t,x)  s^{n+r m} \eta^m 
+ s^{\ell+r+1} L(s,\eta) 
\\
&= s^{\ell+r} \Bigl( 
\frac1{2^{\ell+r} (\ell+r)!} \partial_{x}^{k+2r}u(t,x) 
+
\frac1{2^\ell  \ell!} \partial_{x}^{k+1}u(t,x) \eta
\Bigr)
+ s^{\ell+r+1} L(s,\eta),
\end{split}
\end{equation*}
where $L(s,\eta)$ is a smooth function that may vary from line to line, and where, in the second equality, we used 
the fact that the conditions $m+2n \geq k$ and $n+rm \leq \ell+r$ imply $m \in \{0,1\}$.
By the implicit function theorem, we find, for any $s$ in a neighborhood of $0$, a root of the form 
$y=\eta(s) s^r$, where $\eta$ is a smooth function of $s$. This yields a smooth curve of zeros that crosses the $y$-axis. 
This root behaves polynomially in $s$ and, hence, must differ from the roots we have already found. It is the last missing root.
\vskip 4pt

\textit{Third Step.} It remains to consider the case $r=\infty$. Although this is not really useful in our analysis, we notice that this case may only occur when $t < t_{*}$ and 
$\limsup_{y \downarrow 0} p(t,y)>0$. Indeed, if $\partial_{x}^{k+2 l} u(t,x) =0$ for any $l \in {\mathbb N}$, we obtain 
$\partial_{x} u(t,x) = \cdots =\partial_{x}^{2 \ell+1} u(t,x) =
\partial_{x}^{2 \ell+3} u(t,x) = \cdots= \partial_{x}^{2 \ell + 2 l + 1}=0$, for any $l \in {\mathbb N}$, which proves that $u(t,x+y)$ is even in $y$ in a neighborhood of $0$. Put differently, $u(t,\cdot)$ is locally symmetric with respect to $x$. Recalling that 
$u(t,\cdot)$ is analytic on the entire $(\Lambda_t,\infty)$ (see Proposition \ref{prop:analytic.1}), we conclude that 
$u(t,\cdot)$ and $u(t,-\cdot)$ are two (real-)analytic functions (with different domains) that coincide on a non-empty interval. Hence, 
$u(t,\cdot)$ extends to an analytic function on the entire ${\mathbb R}$ and is symmetric with respect to $x$.
Then, following the end of the proof of 
Proposition
\ref{prop:analytic.1}, we deduce that the density
$\rho(t,\cdot)$ 
changes monotonicity finitely often on compacts of $[0,\infty)$, from which we conclude that $t < t_{*}$. Now, if $\limsup_{y \downarrow 0} p(s,y) = 0$, we get $u(s,\Lambda_{s})=0$, so that, by the symmetry w.r.t. $x$, $u(t,\cdot)$ has another zero greater than $x$. By the maximum principle for the heat equation and by analyticity, this forces $u(t,\cdot)$ to be identically zero, which is of course absurd.
 
Returning to our analysis, we claim that the conclusion of the second step also holds true in the case $r=\infty$:
i.e., there exists a smooth curve of zeros of $\partial_x u$ that crosses the $y$-axis; in addition, all of the derivatives of this curve (viewed as a function of time $s$) vanish at $s=0$. 

Indeed, invoking the Malgrange preparation theorem (see \cite[chapter 2, Theorem 7.1]{ChowHale}) we know that, 
for $(s,y)$ in a neighborhood of $(0,0)$, $\partial_{x} u(t+s,x+y)$ can be written in the form 
$q(s,y) \Gamma_{s}(y)$, where $q(s,y)$ is a smooth non-zero function 
and, for $s$ fixed, $\Gamma_{s}(y):= y^k - \sum_{l=0}^{k-1} a_{l}(s) y^l$ is a polynomial function of order $k$ (whose coefficients $a_l(s)$ are smooth real-valued functions of $s$). 
Clearly, when $s=0$, $\Gamma_{s}(y)$ degenerates into $y^k$. In particular, by continuous dependence of the roots upon the coefficients, all the real roots of $\Gamma_{s}$ are in the neighborhood of $0$ when $|s|$ is small enough. 
From the first part of our analysis, we already know 
that $\Gamma_{s}$ has $k-1=2 \lfloor k/2\rfloor$ simple real roots when $s<0$ and, hence, the missing root must be simple and real. We also know that it cannot have more than one real root when $s>0$ (as otherwise $\partial_{x} u(t+s,x+y)$ would have at least two additional zeros) and, hence, the remaining root must be simple and real as well. Denoting by $\zeta(s)$ this remaining root, we deduce from the implicit function theorem that $\zeta$ is $C^{\infty}$ on $(t-\delta,t)$ and on 
$(t,t+\delta)$, for some $\delta >0$. 
\vskip 4pt

\textit{Fourth Step.}
To obtain smoothness at $s=0$, we argue as follows:
For $s<0$, the Malgrange preparation theorem says that the leading term in the expansion of 
$\partial_{x} u(t+s,x)$ as $s \rightarrow 0$ must behave like the product of all $k$ roots, namely 
$\partial_{x} u(t+s,x)=q(s,0) c(s) s^{\lfloor k/2\rfloor} \zeta(s)$, where $c(s)$ is a non-zero smooth function of $s$. Since $r=\infty$, we easily deduce 
from the Taylor expansion of 
$\partial_{x} u(t+s,x)$
that 
$\vert \zeta(s) \vert$ is less than $C \vert s \vert^{\varrho}$, for all $\varrho \in {\mathbb N}$. 
By differentiating the relationship 
$\partial_{x} u(t+s,x)=q(s,x) s^{\lfloor k/2\rfloor} \zeta(s)$
with respect to $s$ and by arguing in a similar manner, we deduce that 
$\vert \zeta'(s) \vert$ is less than $C \vert s \vert^{\varrho}$, for all $\varrho \in {\mathbb N}$. Iterating this argument, we deduce that
$\zeta(s)$ is infinitely differentiable on $(t-\delta,t]$, with vanishing left-derivatives of all orders at $s=0$. 
 
A similar argument can be applied to $s>0$, by using the complex roots of $\partial_{x} u(t+s,x+y)$, but 
the derivation of the identity 
$\partial_{x} u(t+s,x)=q(s,0) c(s) s^{\lfloor k/2\rfloor} \zeta(s)$
 is more involved. It requires an extension of Malgrange's decomposition to a complex neighborhood of $x$. For this purpose, it is worth recalling that the proof of Malgrange's theorem, as exposed in the monograph 
\cite{ChowHale} is based upon the
polynomial division theorem  
\cite[chapter 2, theorem 7.3]{ChowHale}
with $G(y,s)=y^k$ or $G(y,s)=\partial_{x} u(t+s,x+y)$. The key fact in our case is that both
choices 
have a natural extension to a neighborhood of $(0,0)\in{\mathbb C} \times {\mathbb R}$, the extensions 
$z^k$ and $\partial_{x} u(t+s,x+z)$ being holomorphic in the complex variable $z$ (and having derivatives in $s$ that are also holomorphic in $z$). This makes the choice of $\tilde G$ in the proof trivial, see \cite[page 48]{ChowHale}. In particular, $\tilde G$ may be assumed to be independent of $\lambda$ and to be holomorphic 
in the first argument. Consequently, $q$ in the statement 
of \cite[chapter 2, Theorem 7.3]{ChowHale} may be assumed holomorphic in the first variable when considered on a neighborhood of $0$. As a result, 
the decomposition $\partial_{x} u(t+s,x+y) = q(s,y) \Gamma_{s}(y)$ extends
to $\partial_{x} u(t+s,x+z) = q(s,z) \Gamma_{s}(z)$, for $(s,z)$ in a neighborhood of $(0,0)\in{\mathbb R} \times {\mathbb C}$,
where 
$q(s,z)$ 
is holomorphic in the variable $z$. Obviously, $q$ may be assumed to be non-zero on the latter neighborhood,  
from which we deduce that $\Gamma_{s}(z)$ and 
$\partial_{x} u(t+s,x+z)$ have the same zeros. Hence, the roots of $\Gamma_{s}$, for $s>0$, are given precisely by $\zeta(s)$, and the $k-1=2\lfloor k/2\rfloor$ purely imaginary zeros of $\partial_{x} u(t+s,x+z)$ have already been found. Taking $z=0$, we still have $\partial_{x} u(t+s,x)=q(s,0) c(s) s^{\lfloor k/2\rfloor} \zeta(s)$ and we conclude as in the case $s<0$.} 
 \end{proof}

\begin{lemma}
\label{lem:aux:2}
For arbitrary (fixed) $\bar{\varepsilon}>0$ and $s>s_{*}>0$, let $\zeta:\,[s_{*},s]\to{\mathbb R}$ be a continuous function, such that $\zeta_{s_{*}}=\Lambda_{s_{*}}$ and, on $(s_{*},s]$, $\zeta$ is $C^\infty$, is strictly greater than $\Lambda$, and 
satisfies 
$\vert \partial_{x} u(\cdot,\zeta_{\cdot}) \vert \leq \bar{\varepsilon}$. Then, for $t \in [s_{*},s]$ and $x \geq \zeta_{t}$, 
with $(t,x) \neq (s_{*},\Lambda_{s_{*}})$, 
\begin{equation}
\label{final_FK_proof}
\partial_xu(t,x)=\EE[\partial_xu(t-\tau^{t,x},x+B_{\tau^{t,x}})],
\end{equation}
where $\tau^{t,x}:=\inf\{r\ge0:x+B_r=\zeta_{t-r}\}\wedge(t-s_*)$.  
\end{lemma}

\begin{proof}
The idea of the proof is to show that $v$, defined by 
$v(t,x) = 
\EE[\partial_xu(t-\tau^{t,x},x+B_{\tau^{t,x}})]$, for $t \in [s_{*},s]$ and $x \geq \zeta_{t}$, with 
$(t,x) \neq (s_{*},\Lambda_{s_{*}})$, coincides with $\partial_{x} u(t,x)$. To achieve this, we, first, identify $-\int_{x}^{\infty} v(t,y)\,\mathrm{d}y$ with $u(t,x)$, which is easier than identifying $v(t,x)$ with $\partial_{x} u(t,x)$, since $u(t,\cdot)$ is more regular than $\partial_{x} u(t,\cdot)$ in the right neighborhood of $\Lambda_{t}$. 
\vskip 4pt

\textit{First Step.} We begin by noticing that the function $v$ is well-defined and locally bounded in the interior of the region ${\mathcal D}(s_{*},s) :=\{ (t,x) : t \in [s_{*},s], \ x \geq \zeta_{t}\}$. Since $\partial_{x} u$ remains bounded along 
$\zeta$, we only need to verify that $\EE[ | \partial_{x} u(t-s^*,x+B_{t-s^*})| 
\bone_{\{t-s_{*}=\tau^{t,x}\}}]<\infty$. Recalling the notation $g(s,y)$ for the heat kernel at time $s$ and point $y$, we infer the latter inequality from
$\int_{\Lambda_{s_{*}}}^{\infty} \vert \partial_x u(s_*,y) \vert  \, g(t-s_{*},x-y) \,\mathrm{d}y<\infty$, which follows from Lemma \ref{lem:3.8}.

Moreover, the absolute integrability of $\partial_xu(s_*,\cdot)$ also implies that, for a collection of smooth  
functions $\phi^{\varepsilon}:\,{\mathbb R}\rightarrow[0,1]$, $\varepsilon \in (0,1/2)$, with $\phi^{\varepsilon}(x) = 0$ if $x\leq \Lambda_{s_{*}} + \varepsilon$ or 
$x \geq \Lambda_{s*} + 2/\varepsilon$
 and $\phi^{\varepsilon}(x)=1$ if $x \in [ \Lambda_{s_{*}} + 2 \varepsilon, \Lambda_{s_{*}} + 1/\varepsilon]$, 
 it holds that 
\begin{equation}
\label{final_FK_lim}
\begin{split}
v(t,x) &= \lim_{\varepsilon \downarrow 0} \EE\bigl[\partial_xu(t-\tau^{t,x},x+B_{\tau^{t,x}}) \phi^{\varepsilon} (x+ B_{\tau^{t,x}})\bigr]
=: \lim_{\varepsilon \downarrow 0} v^{\varepsilon}(t,x),   
\end{split}
\end{equation}
for $x\ge\zeta_t$ and $s_*\le t\le s$, with $(t,x) \neq (s_{*},\Lambda_{s_{*}})$. 
Notice that, for any $\varepsilon >0$, the function $\partial_{x} u(t,x) \phi^{\varepsilon}(x)$ is smooth and bounded 
in $(t,x) \in {\mathcal D}(s_{*},s)$. Indeed, by the interior estimates for the heat equation (see e.g.~\cite[chapter IV, Theorem 10.1]{Lady}), $\partial_{x} u(t,x)$ remains bounded over $(t,x)$ changing on any compact inside $\mathring{D}$, with $\mathring{D}$ as in the statement of Proposition \ref{prop:analytic.1}. Hence, letting $\tau^{\epsilon,t,x}:=\tau^{t,x} \wedge (t-s_{*}-\epsilon)$ for $\epsilon>0$, we get 
\begin{equation*}
\begin{split}
v^{\varepsilon}(t,x) &= \lim_{\epsilon \downarrow 0}
\EE\bigl[\partial_xu(t-\tau^{\epsilon,t,x},x+B_{\tau^{\epsilon,t,x}}) \phi^{\varepsilon}(x+ B_{\tau^{\epsilon,t,x}})\bigr]
=: \lim_{\epsilon \downarrow 0} v^{\varepsilon,\epsilon}(t,x),  
\end{split}
\end{equation*}
for $x\ge\zeta_t$ and $s_* <  t\le s$.
Since $\partial_{x} u(t,x) \phi^{\varepsilon}(x)$ is bounded and smooth in $(t,x) \in {\mathcal D}(s_{*}+\epsilon,s)$,
the function $v^{\varepsilon,\epsilon}$ must coincide with the smooth solution of the heat equation on the domain ${\mathcal D}(s_{*}+\epsilon,s)$ with $\partial_{x} u(t,x) \phi^{\varepsilon}(x)$ as the boundary condition on 
$\partial {\mathcal D}(s_{*}+\epsilon,s)$ ({see
\cite[Theorem 10.4.1]{KrylovHolder}, which we can apply here after changing $u(t,x)$ to $u(t,x+\zeta_t)$}). Hence, for $s_{*} + \epsilon < t \leq s$ and $x > \zeta_{t}$, 
we have 
$\partial_t v^{\varepsilon,\epsilon}(t,x)=\frac{1}{2}\partial_{xx}v^{\varepsilon,\epsilon}(t,x)$.
Clearly, by the same argument as the one we used to show the local boundedness of $v$, we can find, on a neighborhood of $(t,x)$, an upper bound for $|v^{\varepsilon,\epsilon}|$ that is independent of $\varepsilon$ and $\epsilon$ (at least for $\epsilon>0$ small enough). 
Invoking, again, the interior estimates for the heat equation, we also obtain uniform bounds for the absolute values of the derivatives of $v^{\varepsilon,\epsilon}$
on a neighborhood of $(t,x)$. Letting $\epsilon$ and then $\varepsilon$ tend to $0$, 
we deduce that $v$ is smooth in the interior of 
${\mathcal D}(s_{*},s)$
and satisfies 
$\partial_t v^{\varepsilon,\epsilon}(t,x)=\frac{1}{2}\partial_{xx}v^{\varepsilon,\epsilon}(t,x)$ in the classical sense, 
for $x>\zeta_t$ and $t\in(s_*,s)$.
\vskip 4pt

\textit{Second Step.} Next, we claim that 
$\lim_{(t,x) \rightarrow (r,\zeta_r)} v(t,x)=\partial_x u(r,\zeta_r)$, $r\in(s^*,s]$. 
The proof relies on three key observations. 

The first observation is that, in the formula 
\eqref{final_FK_lim}, we can choose $\varepsilon>0$ small enough to make 
$\EE[\partial_xu(t-\tau^{t,x},x+B_{\tau^{t,x}}) 
\bone_{\{t-s_{*}=\tau^{t,x},x+B_{\tau^{t,x}}\leq \Lambda_{s*}+2\varepsilon\}}]$ as small as needed, uniformly in $(t,x)$ in a neighborhood of $(r,\zeta_{r})$. 
Indeed, it suffices to make 
$\int_0^{\Lambda_{s_{*}}+2\varepsilon} \vert \partial_x u(s_*,y) \vert  \, g(t-s_{*},x-y) \,\mathrm{d}y$ small by means of 
Lemma \ref{lem:3.8}. 

The second observation is that, on the event 
$\{t-s_{*}=\tau^{t,x},x+B_{\tau^{t,x}} > \Lambda_{s*}+2\varepsilon\}$,
$\vert \partial_xu(t-\tau^{t,x},x+B_{\tau^{t,x}}) \vert$ remains bounded by a constant $C_{\varepsilon}$ (since $\partial_{x} u$ is bounded away from $\partial D$). Combining this with the fact that, for every $\epsilon >0$, 
 $\lim_{x \downarrow \zeta_{t}} {\mathbb P}(\tau^{t,x} > \epsilon)=0$, uniformly in $t$ in a neighborhood of $r$, we obtain
\begin{equation*}
\limsup_{x \downarrow \zeta_{t}}
\EE\bigl[ \vert \partial_xu(t-\tau^{t,x},x+B_{\tau^{t,x}}) \vert 
\,\bone_{\{t-s_{*}=\tau^{t,x},x+B_{\tau^{t,x}}> \Lambda_{s*}+\varepsilon\}} \bigr]
\leq C_{\varepsilon}
\lim_{x \downarrow \zeta_{t}} \PP \bigl( \tau^{t,x}=t-s_{*} \bigr)  = 0,
\end{equation*}
with the limit on the right-hand side being uniform in $t$ on a neighborhood of $r$. 

The last observation is that, on the event 
$\{ \tau^{t,x} < t-s_{*}\}$, $\vert \partial_xu(t-\tau^{t,x},x+B_{\tau^{t,x}}) \vert$
is bounded by $\bar \varepsilon$ and that 
$\partial_x u(\cdot,\zeta_\cdot)$ is continuous on $(s^*,s]$. Therefore,
\begin{equation*}
\begin{split}
\lim_{x \downarrow \zeta_{t}}
\EE\bigl[  \partial_xu(t-\tau^{t,x},x+B_{\tau^{t,x}})\,
\bone_{\{\tau^{t,x} < t-s_{*}\}} \bigr]
&=
\lim_{x \downarrow \zeta_{t}}
\EE\bigl[  \partial_xu (t-\tau^{t,x},\zeta_{t-\tau^{t,x}})\, 
\bone_{\{\tau^{t,x} < t-s_{*}\}}\bigr]
\\
&= \partial_xu (t,\zeta_{t}),
\end{split}
\end{equation*}
and, once again, 
the limit is uniform in $t$ on a neighborhood of $r$. Invoking, once more, the fact that 
$\partial_x u(\cdot,\zeta_\cdot)$ is continuous on $(s^*,s]$, we deduce that $\lim_{(t,x) \rightarrow (r,\zeta_r)} v(t,x)=\partial_x u(r,\zeta_r)$, $r\in(s^*,s]$. 
\vskip 4pt

\textit{Third Step.} Next, we show that $\lim_{t\downarrow s_*} v(t,x)=\partial_x u(s_*,x)$, $x>\Lambda_{s_*}$ locally uniformly.
Indeed, 
with the same notation as in 
\eqref{final_FK_lim}, the fact that $\partial_{x} u(t,x) \phi^{\varepsilon}(x)$ is smooth 
in $(t,x) \in {\mathcal D}(s_{*},s)$ implies 
 that, for any $\varepsilon >0$, $\lim_{t \downarrow s_{*}}v^{\varepsilon}(t,x) = \partial_{x} u(s_{*},x) \phi^{\varepsilon}(x)$, the limit 
 being locally uniform in $x > \Lambda_{s_{*}}$. 
 
Thus, it remains to prove that, for $\delta >0$, we can make the distance $\vert v(t,x) - v^{\varepsilon}(t,x) \vert$ as small as needed by choosing, uniformly in $x \in [\Lambda_{s_{*}}+\delta,\Lambda_{s_{*}}+1/\delta]$, $t$ close enough to $s_{*}$ and $\varepsilon>0$ small enough. This follows from the following two facts. First, 
\begin{equation*}
\begin{split}
&\;\EE\Bigl[\bigl\vert \partial_xu(t-\tau^{t,x},x+B_{\tau^{t,x}}) \bigr\vert \bigl( 1 - \phi^{\varepsilon} (x+ B_{\tau^{t,x}})
\bigr) \,
\bone_{\{t-s_{*}=\tau^{t,x}\}}
\Bigr] 
\\
&\leq 
\int_{\Lambda_{s_{*}}}^{\Lambda_{s_{*}}+2\varepsilon} \vert \partial_x u(s_*,y) \vert  \, g(t-s_{*},x-y) \,\mathrm{d}y
+
\int_{\Lambda_{s_{*}}+1/\varepsilon}^{\infty} \vert \partial_x u(s_*,y) \vert  \, g(t-s_{*},x-y) \,\mathrm{d}y,
\end{split}
\end{equation*}
with the right-hand side of the above tending to $0$ as $\varepsilon\rightarrow0$, uniformly in $(t,x) \in [s_{*},s] \times [\Lambda_{s*}+\delta,\Lambda_{s_{*}}+1/\delta]$. Second, 
\begin{equation*}
\begin{split}
&\EE\Bigl[ \bigl\vert \partial_xu(t-\tau^{t,x},x+B_{\tau^{t,x}}) \bigr\vert \bigl( 1 - \phi^{\varepsilon} (x+ B_{\tau^{t,x}})
\bigr) 
\bone_{\{\tau^{t,x}<t-s_{*}\}}
\Bigr] 
\leq \bar{\varepsilon}\, {\mathbb P} (   \tau^{t,x} < t-s_{*}  ),
\end{split}
\end{equation*} 
and the latter can be made as small as needed by choosing $t$ close enough to $s_{*}$, uniformly in $x \in [\Lambda_{s_{*}}+\delta,\Lambda_{s_{*}}+1/\delta]$.
\vskip 4pt

\textit{Fourth Step.} Our next goal is to identify $u(t,x)$, for
$x \geq \zeta_t$ and $s_* \leq t \le s$,  with 
\begin{equation*}
\widetilde{u}(t,x):=-\int_x^\infty v(t,y)\,\mathrm{d}y.
\end{equation*}
The function $\widetilde u$ is well-defined thanks to the estimate
\begin{equation}\label{v_est}
\begin{split}
|v(t,x)|\le\int_{\Lambda_{s_*}}^\infty |\partial_x u(s_*,y)|\,g(t-s_*,x-y) \,\mathrm{d}y+\bar{\varepsilon}\PP(\tau^{t,x}<t-s_*), \end{split}
\end{equation}
for $x \geq \zeta_t$ and $s_*\leq t\le s$ (the term $g(t-s_{*},x-y)\,\mathrm{d}y$ being understood as the Dirac mass at $x$ when 
$t=s_{*}$). 
By Lemma \ref{lem:3.8} and the sub-Gaussianity of $\min_{[0,s-s_*]} B$, we conclude that the right-hand side of the above is integrable in $x$. Hence, 
$\widetilde{u}$ is a well-defined and bounded function of $(t,x)$. We claim that it is smooth in the interior of ${\mathcal D}(s_{*},s)$. 
Indeed,
with the same notation as in 
\eqref{final_FK_lim}, we can write
\begin{equation*}
\widetilde{u}(t,x)=
\lim_{\varepsilon \downarrow 0}\, \widetilde{u}^{\varepsilon}(t,x), \quad 
\widetilde{u}^{\varepsilon}(t,x) := 
-\int_x^\infty v^{\varepsilon}(t,y)\,\mathrm{d}y,\quad x>\zeta_t,\;\;s_*<t\le s.
\end{equation*}
Now, the analogue of 
\eqref{v_est} for $v^{\varepsilon}$ shows that 
$v^{\varepsilon}(t,x)$ decays exponentially fast to $0$ as $x$ tends to $\infty$, uniformly in $t \in [s_{*},s]$. 
Invoking, again, the interior estimates for the heat equation, we deduce that the derivatives of $v^{\varepsilon}$ also decay 
exponentially fast in $x$ when $t$ is restricted to a compact subset of $(s_{*},s]$. In particular, it is easy to see that, for any 
$\varepsilon>0$, 
$\widetilde{u}^{\varepsilon}$ satisfies the heat equation 
in the interior of 
${\mathcal D}(s_{*},s)$. 
By \eqref{v_est} again, the family $\widetilde{u}^{\varepsilon}$, $\varepsilon \in (0,1/2)$ is uniformly bounded on compact subsets of ${\mathcal D}(s_{*},s)$. Another application of the interior estimates for the heat equation shows that 
the derivatives of $\widetilde{u}^{\varepsilon}$, $\varepsilon \in (0,1/2)$ are also uniformly bounded on compact subsets of ${\mathcal D}(s_{*},s)$. Thus, we deduce that $\tilde{u}$ is smooth in the interior of ${\mathcal D}(s_{*},s)$
and satisfies $\partial_t\widetilde{u}(t,x)=\frac{1}{2}\partial_{xx}\widetilde{u}(t,x)$, $x>\zeta_t$, $t\in(s_*,s)$.

At the boundary, we 
clearly have $\lim_{(t,x) \downarrow (r,\zeta_r)} \partial_x\widetilde{u}(t,x)=\lim_{(t,x)\downarrow (r,\zeta_r)} v(t,x)=\partial_x u(r,\zeta_r)$, for $r\in(s_*,s]$. We also claim that $\lim_{(t,x) \downarrow (r,\zeta_r)} \widetilde{u}(t,x)= \widetilde u(r,\zeta_r)$, for $r\in(s_*,s]$. To prove the latter claim, we split 
$\widetilde{u}(t,x)$ into three parts:
$$\widetilde{u}(t,x) = \int_{x}^{\zeta_{r}+\epsilon} v(t,y)\,{\mathrm d}y + 
 \int_{\zeta_{r}+\epsilon}^K v(t,y)\,{\mathrm d}y
 + \int_{K}^{\infty} v(t,y)\,{\mathrm d}y,$$
 for $\zeta_{t} \leq x \leq \zeta_{r}+\epsilon$,  with $\epsilon>0$ small and $K>0$ large.
 By the continuity of $v$ at $(r,\zeta_{r})$, we can choose $\epsilon>0$ and $\vert t-r \vert$
small enough to make 
$\int_{x}^{\zeta_{r}+\epsilon} v(t,y)\,
{\mathrm d}y - 
\int_{\zeta_{r}}^{\zeta_{r}+\epsilon} v(r,y)\,{\mathrm d}y$ as small as needed. 
Then, by \eqref{v_est}, we can choose $K$ large enough to make 
$\int_{K}^{\infty} v(t,y)\,{\mathrm d}y$ small, uniformly in $t$ changing in a neighborhood of $r$. 
Finally, 
for fixed values of $\epsilon$ and $K$, we can decrease $\vert t-r \vert$ if necessary to 
make 
$\int_{\zeta_{r}+\epsilon}^K (v(t,y)-v(r,y))\,{\mathrm d}y$ as small as needed.

Moreover, we claim that $\lim_{t\downarrow s_*} \widetilde{u}(t,x)=-\int_x^\infty \partial_x u(s_*,y)\,\mathrm{d}y=u(s_*,x)$, $x>\Lambda_{s_*}$ locally uniformly. 
The second of the two equalities is a direct consequence of Lemma \ref{lem:3.8}. 
As for the first one, we notice 
that, for any $\varepsilon >0$, $\lim_{t\downarrow s_*} \widetilde{u}^{\varepsilon}(t,x)=-\int_x^\infty \partial_x u (s_*,y) \phi^{\varepsilon}(y)\,\mathrm{d}y$,  locally uniformly in $x>\Lambda_{s_*}$. The latter follows
from the following two facts: $v^{\varepsilon}$ decays exponentially fast in space, uniformly in time, and 
$\lim_{t \downarrow s_{*}} v^{\varepsilon}(t,x) = \partial_{x} u(s_{*},x) \phi^{\varepsilon}(x)$, 
 locally uniformly in 
$x>\Lambda_{s_{*}}$. Then, it remains to prove that, for any $\delta >0$,  
we can make the distance $\vert \widetilde u (t,x) - \widetilde u^{\varepsilon}(t,x) \vert$
as small as needed by choosing, uniformly in $x \in [\Lambda_{s_{*}}+\delta,\infty)$, 
$t$ close enough to $s_{*}$ and $\varepsilon>0$ small enough. By the second step, we already know that  
we can make the distance $\vert v(t,x) - v^{\varepsilon}(t,x) \vert$
as small as needed by choosing, uniformly in $x \in [\Lambda_{s_{*}}+\delta,\Lambda_{s_{*}}+1/\delta]$, 
$t$ close enough to $s_{*}$ and $\varepsilon>0$ small enough. Hence, it suffices to establish 
\begin{equation}
\lim_{\delta \rightarrow 0}\,\int_{\Lambda_{s*}+1/\delta}^{\infty} \vert v^{\varepsilon}(t,y) - v(t,y) \vert \,\mathrm{d}y = 0,
\end{equation}
uniformly in $t \in [s_{*},s]$ and $\varepsilon \in (0,1/2)$. The latter follows from the appropriate versions of \eqref{v_est}. 
\vskip 4pt

\textit{Fifth Step.} All in all, $\widetilde{u}$ is a classical solution of the problem
\begin{equation}
\begin{split}
&\partial_t\widetilde{u}=\frac{1}{2}\partial_{xx}\widetilde{u}\quad\text{on}\quad\{(t,x)\in(s_*,s)\times[0,\infty):\,x>\zeta_t\}, \\
&\widetilde{u}(s_*,x)=u(s_*,x),\quad x>\Lambda_{s_*}\quad\text{and}\quad
\partial_x\widetilde{u}(t,\zeta_t)=\partial_xu(t,\zeta_t),\quad t\in(s_*,s),
\end{split}
\end{equation}
and the same is true for $u$. Both $u$ and $\tilde u$ are continuous and have continuous derivatives at any point $(t,\zeta_{t})$ of the boundary, for $t \in (s_{*},s)$. Hence, for any $t\in(s_*,s)$, $x>\zeta_t$, and with the reflected Brownian motion $R^{t,x}$ in the time-dependent domain $\{(r,y)\in[0,t-s_*]\times[0,\infty):\,y>\zeta_{t-r}\}$, started from $x$, and its boundary local time $\ell^{t,x}$ (see \cite[section 2]{BCS2} for more details on such processes), we have 
\begin{equation}
\begin{split}
& \widetilde{u}(t,x)=\EE\bigg[\widetilde{u}(s_*+\epsilon,R^{t,x}_{t-s_*-\epsilon})
-\int_0^{t-s_*-\epsilon} \partial_x u(t-r,\zeta_{t-r})\,\mathrm{d}\ell^{t,x}_r\bigg]\quad\text{and} \\ & u(t,x)=\EE\bigg[u(s_*+\epsilon,R^{t,x}_{t-s_*-\epsilon})-\int_0^{t-s_*-\epsilon} \partial_x u(t-r,\zeta_{t-r})\,\mathrm{d}\ell^{t,x}_r\bigg],
\end{split}
\end{equation}
for all $\epsilon\in(0,t-s_*)$ (cf.~\cite[Theorem 2.8]{BCS1}). Relying on $\PP(R^{t,x}_{t-s_*}=\Lambda_{s_*})=0$ (see \cite[Theorem 2.2]{BCS2}), as well as on the locally uniform convergences $\lim_{t\downarrow s_*} \widetilde{u}(t,x)=u(s_*,x)=\lim_{t\downarrow s_*} u(t,x)$, $x>\Lambda_{s_*}$, and on the boundedness of $\widetilde{u}$, $u$, and $|\partial_x u(t,\zeta_t)|\leq \bar{\varepsilon}$, $r\in[s^*,t]$, we consider $\epsilon\downarrow0$ to obtain:
\begin{equation}
\widetilde{u}(t,x)=\EE\bigg[u(s_*,R^{t,x}_{t-s_*})-\int_0^{t-s_*}\partial_x u(t-r,\zeta_{t-r})\,\mathrm{d}\ell^{t,x}_r\bigg]=u(t,x). 
\end{equation} 
Thus, $\partial_x u(t,x)=\partial_x\widetilde{u}(t,x)=v(t,x)=\EE[\partial_xu(t-\tau^{t,x},x+B_{\tau^{t,x}})]$, yielding the Feynman-Kac formula \eqref{final_FK_proof}. 
\end{proof}


\section{Proof of Theorem \ref{thm:uniq}}
\label{se:5}

In this section, we prove the uniqueness of the physical solution $(X,\Lambda)$ of \eqref{Stefan_prob}, under the assumptions of Theorem \ref{thm1} (which are in force throughout the section). The strategy of the proof is to represent a physical solution as  a fixed point of a mapping that has a contraction property. We start with the following technical lemma.

\begin{lemma}\label{le:stability.critical.1.n}
Let $\mathcal{X}\geq0$ be a random variable with a bounded density $\rho$ on $(0,\infty)$ such that
\begin{equation*}
\rho(x) \leq \frac{1}{\alpha}-\psi(x),\quad x>0,
\end{equation*}
where $\psi$ is non-decreasing and strictly positive on a non-trivial interval $(0,\delta_0)$, and  $\psi(0+)=0$. Then, for any $t\ge0$, there exist $T>t$ and $\varepsilon_0,\delta>0$ such that, for all $s\in[t,T]$ and all measurable $\widetilde{L}:[t,T]\rightarrow[0,\varepsilon_0]$, we have
$$
\PP\bigl(0<\inf_{r\in[t,s]} (\mathcal{X} + \widetilde{B}_{r} - \widetilde{L}_r)\leq x\bigr) \leq \frac{x}{\alpha} -  \frac{1}{2}\int_0^x \psi(z)\,\mathrm{d}z,\quad x\in[0,\delta],
$$
where $\widetilde{B} = (\widetilde{B}_r)_{r \geq t }$ is a standard Brownian motion started from $\widetilde{B}_t=0$.
\end{lemma}
\begin{proof}
The proof is a modification of the ``Second case" in the proof of Proposition \ref{prop:sec2.p0.upperbd}. We denote by $F$ the cumulative distribution function of $\rho$ and obtain for $s\ge t$:
\begin{equation}\label{eq.uniqueness.mainTechLemma.eq1}
\begin{split}
&\;\PP\bigl(0<\inf_{r\in[t,s]} (\mathcal{X} + \widetilde{B}_{r} - \widetilde{L}_r)\leq x\bigr)
\\ 
&=\int_{(-\varepsilon,0]} \bigl( F(x - y) - F(-y) \bigr)\,\widetilde{\nu}_s(\mathrm{d}y)
 +  \int_{(-\infty,-\varepsilon]}
 \bigl( F(x - y) - F(- y) \bigr)\,\widetilde{\nu}_s(\mathrm{d}y),
\end{split}
\end{equation}
with an arbitrary constant $\varepsilon>0$ and with $\widetilde{\nu}_s$ being the law of $\inf_{r\in[t,s]} (\widetilde{B}_{r} - \widetilde{L}_r)$.

We estimate the right-hand side of \eqref{eq.uniqueness.mainTechLemma.eq1} using methods similar to those in  the proof of Proposition \ref{prop:sec2.p0.upperbd}. We begin with the first integral in (\ref{eq.uniqueness.mainTechLemma.eq1}), which poses the main difference to the proof of Proposition \ref{prop:sec2.p0.upperbd}. Since $\psi$ is non-decreasing on $(0,\delta_0)$, we conclude that, for all $\delta_1\in(0,\delta_0)$,
\begin{equation*}
F(x-y) - F(-y) = \int_{0}^{x} \rho(z - y)\,\mathrm{d}z
\leq \frac{x}{\alpha} - \int_0^x \psi(z-y)\,\mathrm{d}z
\leq \frac{x}{\alpha} - \int_0^x \psi(z)\,\mathrm{d}z,\quad x,-y\in[0,\delta_1/2]. 
\end{equation*}
Thus, for all $x,\varepsilon\in(0,\delta_1/2]$, the first term on the right-hand side of (\ref{eq.uniqueness.mainTechLemma.eq1}) does not exceed
\begin{equation}
\label{eq:5:1}
\frac{x}{\alpha}\,\widetilde{\nu}_s((-\varepsilon,0])
- \int_0^x \psi(z)\,\mathrm{d}z \cdot\widetilde{\nu}_s((-\varepsilon,0]).
\end{equation}

Next, we bound the second integral in (\ref{eq.uniqueness.mainTechLemma.eq1}). Notice that, for any $z\in\mathbb{R}$,
$$
\PP\bigl(\inf_{r\in[t,s]} (\widetilde{B}_{r} - \widetilde{L}_r)\leq z\bigr) \leq \PP\bigl(\inf_{r\in[t,s]} \widetilde{B}_{r} \leq z+\varepsilon_0\bigr).
$$
In addition, due to the fast decay, as $s\downarrow t$, of the density of $\inf_{r\in[t,s]} \widetilde{B}_{r}$ we have, for $\varepsilon\geq\varepsilon_0$,
\begin{equation*}
\begin{split}
\PP\bigl(\inf_{r\in[t,s]} \widetilde{B}_{r} \leq \varepsilon_0 - 2\varepsilon\bigr)
&\leq e^{-(\varepsilon-\varepsilon_0)^2/(2(s-t))}\,\PP\bigl(\inf_{r\in[t,s]} \widetilde{B}_{r} \leq -\varepsilon\bigr) \\
&\leq e^{-(\varepsilon-\varepsilon_0)^2/(2(s-t))}\,\PP\bigl(\inf_{r\in[t,s]} (\widetilde{B}_{r} - \widetilde{L}_r) \leq -\varepsilon\bigr).
\end{split}
\end{equation*}
Hence, there exists a $T=T(\varepsilon-\varepsilon_0)>t$ such that, for all $s\in[t,T]$,
\begin{equation*}
\begin{split}
&\;\PP\bigl(\inf_{r\in[t,s]} (\widetilde{B}_{r} - \widetilde{L}_r)\leq -2\varepsilon\bigr) \\
&\leq e^{-(\varepsilon-\varepsilon_0)^2/(2(s-t))}\,\PP\bigl(\inf_{r\in[t,s]} (\widetilde{B}_{r} - \widetilde{L}_r)\leq -\varepsilon\bigr)
\\
&=e^{-(\varepsilon-\varepsilon_0)^2/(2(s-t))} \Bigl(\PP\bigl(\inf_{r\in[t,s]} (\widetilde{B}_{r} - \widetilde{L}_r)\leq -2\varepsilon\bigr)
+\PP\bigl(-2\varepsilon < \inf_{r\in[t,s]} (\widetilde{B}_{r} - \widetilde{L}_r)\leq -\varepsilon\bigr)\Bigr).
\end{split}
\end{equation*}
Setting $\varepsilon=\delta_1/3$ (upon decreasing $ \varepsilon_{0}>0$ if necessary to ensure $\varepsilon_0<\delta_1/3$) we conclude that, for any $\gamma>0$, there exists a $T>t$ such that 
$$
\PP\bigl(\inf_{r\in[t,s]} (\widetilde{B}_{r} - \widetilde{L}_r)\leq -2\varepsilon\bigr)
\leq \gamma \PP\bigl(-2\varepsilon < \inf_{r\in[t,s]} (\widetilde{B}_{r} - \widetilde{L}_r)\leq -\varepsilon\bigr),\quad s\in[t,T].
$$
Moreover, for the same choice of $\varepsilon$, one can find a $C_1<1/\alpha$ so that
$$
F(x - y) - F(- y) \leq C_1 x,\quad x\in[0,\delta_1+y),\quad y\in[-2\varepsilon,-\varepsilon].
$$
Therefore, for $\gamma>0$ satisfying $\gamma\|\rho\|_{L^{\infty}} + C_1<1/\alpha$ and all $s\in[t,T]$, $x\in(0,\delta_1/2]$, 
\begin{equation}
\label{eq:5:2}
\begin{split}
\int_{(-\infty,-\varepsilon]}
\bigl( F(x - y) - F( - y) \bigr)\,\widetilde{\nu}_s(\mathrm{d}y)
&\leq   \|\rho\|_{L^{\infty}}\,x\,\widetilde{\nu}_s((-\infty,-2\varepsilon]) 
+ C_1\,x\,\widetilde{\nu}_s((-2\varepsilon,-\varepsilon]) \\
&\leq x(\gamma\|\rho\|_{L^{\infty}} + C_1) \widetilde{\nu}_s((-2\varepsilon,-\varepsilon])
\\
&< \frac{x}{\alpha} \,\widetilde{\nu}_s((-\infty,-\varepsilon]).
\end{split}
\end{equation}
Collecting the estimates \eqref{eq:5:1} and \eqref{eq:5:2}, and decreasing $T>t$ if necessary to guarantee
$$
\widetilde{\nu}_s((-\varepsilon,0]) \geq \frac{1}{2},\quad s\in[t,T],
$$
we obtain
$$
\PP\bigl(0<\inf_{r\in[t,s]} (\mathcal{X} + \widetilde{B}_{r} - \widetilde{L}_r)\leq x\bigr) \leq \frac{x}{\alpha} - \frac{1}{2}\int_0^x \psi(z)\,\mathrm{d}z,
$$
for all $s\in[t,T]$, all measurable $\widetilde{L}:[t,T]\rightarrow[0,\varepsilon_0]$, and all $0\le x\le\delta_1/2=:\delta$.
\end{proof}

\medskip

The following proposition proves the local uniqueness of the physical solution by establishing the aforementioned contraction property. The proof is similar to the arguments presented in \cite{FP2} and \cite{LS2}. Recall that the assumptions of Theorem \ref{thm1} are in force throughout this section.

\begin{proposition}\label{prop:uniqueness.stability}
Given an initial condition $X_{0-}$ satisfying the conditions of Theorem \ref{thm1}, let $(X^1,\Lambda^1)$ and $(X^2,\Lambda^2)$ be two physical solutions of \eqref{Stefan_prob} starting from $X_{0-}$ and coinciding on $[0,t)$, for some $t\in[0,\infty)$. (If $t=0$, we just have that $(X^1,\Lambda^1)$ and $(X^2,\Lambda^2)$ coincide at $t=0-$.) Then, there exists a $T>t$ such that
$$
X^2_s-X^1_s=\Lambda^1_s - \Lambda^2_s = 0,
\quad s\in[t,T].
$$
\end{proposition}
\begin{proof}
Notice that $\Lambda^1_t=\Lambda^2_t$ and $X^1_t=X^2_t$, as the two solutions must have the same jump size at $t$. Moreover, $\mathcal{X}:=X^i_{t}\,\bone_{\{\tau^i>t\}}$ (with $\tau^i$ defined in an obvious manner as in \eqref{Stefan_prob}) satisfies the conditions of Lemma \ref{le:stability.critical.1.n} by Theorem \ref{thm1}. 

Next, we write
$$
\widetilde \Lambda^i_s:=\Lambda^i_s - \Lambda^i_t = \alpha\PP\big(\mathcal{X}>0,\,\inf_{r\in[t,s]} (\mathcal{X} + \widetilde B_r - \widetilde\Lambda^i_r) \leq 0 \big),\quad s\geq t.
$$
Thus,
\begin{equation*}
\begin{split}
\widetilde{\Delta}_s:=\widetilde\Lambda^1_s - \widetilde\Lambda^2_s &= \alpha
\Big(\PP\big(\mathcal{X}>0,\inf_{r\in[t,s]} (\mathcal{X}\!+\!\widetilde B_r\!-\!\widetilde\Lambda^1_r) \leq 0 \big) - \PP\big(\mathcal{X}>0,\inf_{r\in[t,s]} (\mathcal{X}\!+\!\widetilde B_r\!-\!\widetilde\Lambda^2_r) \leq 0 \big)\!\Big)
\\
&\leq \alpha \PP\big(0<\inf_{r\in[t,s]} (\mathcal{X}+ \widetilde{B}_{r} - \widetilde\Lambda^2_r) \leq \sup_{r\in[t,s]} \widetilde\Delta_r \big),\quad s\geq t. 
\end{split}
\end{equation*}
Combining this with Lemma \ref{le:stability.critical.1.n} we deduce that 
\begin{equation*}
\widetilde{\Delta}_s\le \sup_{r\in[t,s]} \widetilde\Delta_r-\frac{\alpha}{2}\int_0^{\sup_{r\in[t,s]} \widetilde\Delta_r} \psi(z)\,\mathrm{d}z,\quad s\in[t,T].
\end{equation*}

Decreasing $T>t$ if necessary to make the right-hand side of the latter display non-decreasing in $\sup_{r\in[t,s]} \widetilde\Delta_r$ (recall that $\widetilde{\Delta}$ is right-continuous and $\psi(0+)=0$) and taking the running supremum of both sides we arrive at
\begin{equation*}
0\leq \sup_{r\in[t,s]} \widetilde{\Delta}_r\le \sup_{r\in[t,s]} \widetilde\Delta_r-\frac{\alpha}{2}\int_0^{\sup_{r\in[t,s]} \widetilde\Delta_r} \psi(z)\,\mathrm{d}z,\quad s\in[t,T].
\end{equation*}
Since $\psi$ is strictly positive on $(0,\delta_0)$, for some $\delta_0>0$, we have $\sup_{r\in[t,s]} \widetilde{\Delta}_s=0$, i.e., $\widetilde{\Lambda}^1_s\le\widetilde{\Lambda}^2_s$, for all small enough $s>t$. Reversing the roles of $\widetilde{\Lambda}^1$ and $\widetilde{\Lambda}^2$ we complete the proof of the proposition.  
\end{proof}

\medskip

Theorem \ref{thm:uniq} is an easy consequence of Proposition \ref{prop:uniqueness.stability}. Indeed, assuming that there exist two distinct physical solutions $(X^1,\Lambda^1)$ and $(X^2,\Lambda^2)$  of \eqref{Stefan_prob} and letting 
\begin{equation*}
t:=\inf\{s\ge 0:\,\Lambda^1_s\neq\Lambda^2_s\}\in[0,\infty) 
\end{equation*}
we see that Proposition \ref{prop:uniqueness.stability} contradicts the definition of $t$.

\bigskip\bigskip

\bibliographystyle{amsalpha}
\bibliography{Main}

\bigskip\bigskip\medskip

\end{document}